\documentclass{article}

\newcommand{\vect}[1]{\boldsymbol{#1}} 
\newcommand{\R}{\mathbb{R}} 
\newcommand{\1}{\vect{1}} 
\newcommand{\0}{\vect{0}}
\newcommand{\I}{\mathcal{I}} 
\newcommand{\lyp}{\mathcal{L}} 
\newcommand{\diag}[1]{\textrm{diag}\left({#1}\right)} 
\newcommand{\id}[1]{\vect{1}_{\{{#1}\}}} 
\newcommand{\Proj}{\mathcal{P}} 

\usepackage{threeparttable}
\usepackage{enumitem}

\usepackage{microtype}
\usepackage{graphicx}
\usepackage{subcaption}
\usepackage{booktabs} 

\usepackage{amsmath}
\usepackage{amssymb}
\usepackage{mathtools}
\usepackage{amsthm}
\usepackage{algorithm}
\usepackage{algorithmic}

\theoremstyle{plain}
\newtheorem{theorem}{Theorem}[section]

\newtheorem{lemma}[theorem]{Lemma}

\theoremstyle{definition}

\theoremstyle{remark}

\usepackage{hyperref}

 \usepackage[preprint]{neurips_2026}

\usepackage[utf8]{inputenc} 
\usepackage[T1]{fontenc}    
\usepackage{hyperref}       
\usepackage{url}            
\usepackage{booktabs}       
\usepackage{amsfonts}       
\usepackage{nicefrac}       
\usepackage{microtype}      
\usepackage{xcolor}         

\title{Inexact Bregman Sparse Newton Method for Efficient Optimal Transport}

\author{%
  Jianting Pan\thanks{These authors contributed equally to this work.},
  Ji'an Li\footnotemark[1],
  Sirong Dai\footnotemark[1],
  Ming Yan\thanks{Corresponding author.}\\
  School of Data Science,
  The Chinese University of Hong Kong, Shenzhen \\
  \texttt{{jiantingpan,jianli1,sirongdai}@link.cuhk.edu.cn, yanming@cuhk.edu.cn}
}

\begin{document}

\maketitle

\begin{abstract}
Computing exact optimal transport (OT) distances at large scale remains computationally challenging. Entropy-regularized OT can be solved efficiently, but its solution is biased toward the regularized problem, and small regularization parameters often lead to slow convergence and numerical instability. In this paper, we propose \emph{Inexact Bregman Sparse Newton} (IBSN), an efficient method for high-accuracy computation of the original OT problem. Building on a verifiable inexact Bregman proximal framework, IBSN develops a sparse Newton inner solver tailored to the semi-dual formulation, which yields a lower-dimensional Newton system. To accelerate the inner subproblem solves, we introduce a Hessian sparsification strategy tailored to the semi-dual formulation, which reduces the cost of Newton direction computation while preserving key spectral properties. We establish convergence guarantees for the resulting algorithm, including convergence of the inner sparse Newton iterations and global convergence of the overall method to an optimal OT solution. Numerical experiments on synthetic and real datasets show that IBSN is effective in high-accuracy regimes and compares favorably with representative first-order and second-order OT solvers. 
The implementation and code are at~\url{https://github.com/PanT12/IBSN-code}.

\end{abstract}

\section{Introduction}

Optimal transport (OT) has emerged as a fundamental tool for measuring the difference between probability distributions. Unlike standard metrics, OT offers a geometrically meaningful distance. This unique property has led to widespread applications across machine learning~\cite{gordaliza2019obtaining,oneto2020exploiting,montesuma2024recent}, computer vision~\cite{bonneel2023survey,rubner2000earth,solomon2014earth}, and statistics~\cite{goldfeld2024statistical,klatt2020empirical}.

The discrete OT problem can be characterized by the following linear programming problem:
\begin{equation}
\begin{aligned}
\label{opt: problem}
\min_{X \in \Omega}\ & \langle C, X\rangle \quad 
\text{where} \ \Omega := \{X \in \R^{m \times n} \mid X \1_n=\vect{a}, X^{\top} \1_m=\vect{b}, X \geq 0\},
\end{aligned}
\end{equation}
where $C \in \R_{+}^{m \times n}$ is the cost matrix, $\vect{a} \in \R_{++}^m$ and $\vect{b} \in \R_{++}^n$ are two vectors satisfying $\1_m^\top \vect{a} = \1_n^\top \vect{b} = 1$; $\1_m$ and $\1_n$ are the vectors of all ones in $\R^m$ and $\R^n$, respectively. While this formulation is linear, solving large-scale problems is computationally challenging. 

To address this, an entropic regularization term $\phi(X) = \sum_{i=1}^m \sum_{j=1}^n X_{ij}\log X_{ij}$ was added to the objective, which is known as entropy-regularized optimal transport (EOT)~\cite{cuturi2013sinkhorn}:
\begin{equation}
\label{opt: EOT}
\min_{X \in \Omega}\ \langle C, X\rangle + \eta \phi(X),
\end{equation}
where $\eta > 0$ is a regularization parameter. This modification yields a $\eta$-strongly convex problem on $\Omega$ that can be solved efficiently via the Sinkhorn algorithm~\cite{cuturi2013sinkhorn}. Despite its simplicity and scalability, it is a first-order method with sublinear convergence, meaning it often requires many iterations to reach high accuracy~\cite{wang2025sparse}. Recently, to accelerate convergence, researchers have turned to second-order methods such as the Newton method~\cite{tang2024safe, wang2025sparse, tang2024accelerating}. By combining Newton updates with sparse Hessian approximations and efficient linear solvers, these methods achieve fast local convergence and handle large datasets much better than standard Sinkhorn. However, the EOT solution remains only an approximation of the original OT problem. While reducing the regularization parameter improves approximation accuracy, excessively small regularization values cause the algorithm to slow down and lead to severe numerical instability, such as overflow and underflow. 

Instead of settling for the approximate solution provided by EOT, recent research has shifted focus back to solving the original OT problem exactly. This is typically achieved by viewing EOT~\eqref{opt: EOT} as a subproblem within a larger iterative framework, such as the Bregman proximal point algorithm. A natural approach is to solve these subproblems exactly until convergence. \cite{wu2025pins} accelerated this process by applying Newton methods with sparsification techniques. However, requiring an exact solution at every step created a heavy computational burden. Inexact frameworks reduce this burden by allowing approximate subproblem solutions~\cite{xie2020fast,eckstein1998approximate,solodov2000inexact}, but their stopping criteria can be difficult to verify or justify theoretically. The recent framework of~\cite{yang2022bregman} addresses this issue with an easily verifiable stopping condition. However, when implemented with first-order inner solvers, it may still become slow as the required inner accuracy tightens. This motivates our use of an efficient sparse Newton-type solver within the inexact Bregman framework.

In this paper, we propose \textbf{I}nexact \textbf{B}regman \textbf{S}parse \textbf{N}ewton (IBSN), an efficient algorithm for solving the original optimal transport problem~\eqref{opt: problem}. IBSN adopts the inexact Bregman proximal framework as its outer backbone, but redesigns the inner solver using the semi-dual formulation. The semi-dual formulation itself is standard, while it achieves a lower-dimensional Newton system and yields a different Hessian structure. To exploit this structure, we develop a novel sparsification strategy that preserves the key spectral properties of the semi-dual Hessian. It enables efficient Newton updates while retaining fast local convergence. We establish convergence guarantees for the resulting method and demonstrate its effectiveness on synthetic and real OT benchmarks.


\noindent{\bf Contribution.} The main contributions of this work are summarized as follows:
\begin{itemize}[topsep=0pt,itemsep=-1ex,partopsep=1ex,parsep=1ex]
    \item We develop a sparse Newton solver for the semi-dual formulation of~\eqref{opt: EOT}. Unlike existing sparse Newton methods based on the $(m+n)$-dimensional two-dual system, our method operates on an $n$-dimensional Newton system. We design a sparsification strategy for the semi-dual Hessian and establish its spectral properties and approximation-error bounds.
    \item We integrate the proposed sparse Newton solver into the inexact Bregman framework and establish convergence guarantees for both inner and outer iterations (proof in Appendix~\ref{app: Proofs of Theoretical Results}). Extensive experiments demonstrate its effectiveness against iterative OT solvers.
\end{itemize}

\noindent{\bf Notation.} For any $\vect{x} \in \R^n$, we denote its $i$-th entry by $x_i$. For an index set $\I \subseteq [n]:= \{1,2,\dots,n\}$, $\vect{x}_{\I} \in \R^{|\I|}$ represents the subvector of $\vect{x}$ corresponding to the indices in $\I$, where $|\I|$ denotes the cardinality of $\I$. The span of vector $\vect{x}$ is defined as $\text{span}\{\vect{x}\} = \{c\vect{x} \mid c \in \R\}$. We denote $\1_n$ and $\0_n$ as the $n$-dimensional vectors of all ones and all zeros, respectively. We use $\lambda_{\min}^+(\cdot)$ and $\lambda_{\max}(\cdot)$ to represent the smallest positive and largest eigenvalues of real symmetric matrices, respectively. For a matrix $A \in \R^{m \times n}$, let $A_{i\cdot}$ be the vector of the $i$-th row of $A$. The kernel of matrix $A$ is defined as $\text{ker}(A) = \{\vect{x} \mid A\vect{x} = \0_m\}$. We denote $\|\cdot\|_p$ as the $\ell_p$ norm of a vector or $\ell_p$ induced norm of a matrix. For simplicity, we use $\|\cdot\|$ to denote the Euclidean norm of a vector or the $\ell_2$ induced norm of a matrix.

\section{Related Work}

\noindent{\bf Applications of OT.} The versatility of OT has led to its adoption in a broad range of scientific and engineering domains. In machine learning, OT distances underpin generative adversarial networks~\cite{daniels2021score,arjovsky2017wasserstein}, domain adaptation~\cite{courty2016optimal,courty2017joint}, and distributionally robust optimization~\cite{onken2021ot} by offering geometry-aware measures of distributional discrepancy. In computer vision and graphics, OT provides a principled framework for color transfer~\cite{rabin2014adaptive}, texture mixing~\cite{rabin2011wasserstein}, and point cloud registration~\cite{shen2021accurate,puy2020flot}. These diverse applications highlight the growing importance of efficient OT computation as a core building block for modern data science.

\noindent{\bf Computational OT.} Algorithmic advances in OT have followed two main directions. The first direction leverages entropy regularization for its computational benefits. Initial progress relied on first-order methods, most notably the Sinkhorn algorithm~\cite{cuturi2013sinkhorn} and its variants (e.g., Greenkhorn~\cite{altschuler2017near} and stabilized Sinkhorn~\cite{schmitzer2019stabilized}). For further acceleration, Newton-based methods (as seen in~\cite{tang2024safe,tang2024accelerating,wang2025sparse,brauer2017sinkhorn,wu2025pins}) exploit second-order information and sparsity. The second direction tackles the original OT problem directly, often by embedding the EOT form into an iterative framework. The strategy of employing EOT as a subproblem within a Bregman framework is a natural and common approach. Other direct solvers include Halpern-accelerated approaches (HOT)~\cite{zhang2025hot}, primal-dual hybrid gradient method (PDHG)~\cite{applegate2021practical}, Douglas-Rachford splitting algorithm~\cite{mai2021fast}, and extragradient method~\cite{li2025fast}. 

\noindent{\bf Bregman proximal point.} 
Bregman divergence has been widely adopted as a proximal regularizer in large-scale optimization, particularly when the decision variables lie in the probability simplex~\cite{bauschke2017descent,beck2003mirror,bolte2018first,ben2001ordered,pan2025efficient}. Compared to Euclidean proximal mappings, Bregman proximal mappings naturally encode the geometry of distributions and yield algorithms with improved stability and interpretability.

\noindent{\bf Inexact optimization of inner subproblems.} A common strategy in large-scale optimization, including OT, is to solve inner subproblems only approximately rather than to full precision~\cite{xie2020fast,yang2022bregman,burachik1997enlargement,eckstein1998approximate}. Such inexactness reduces per-iteration cost while still guaranteeing global convergence, provided that the approximation error is controlled relative to the outer iteration. 

\section{Background}
\label{sec: Background}



In this section, we review the subproblem formulation and its corresponding semi-dual problem based on Bregman divergence.

To compute the optimal transport plan, we employ a Bregman proximal point framework. In this iterative scheme, each iteration solves a regularized subproblem of the form:
\begin{equation}
\label{opt: subproblem}
X^{k+1} \in \arg \min_{X \in \Omega} \left\{\langle C, X\rangle + \eta D_{\phi}(X, X^k)\right\},
\end{equation}
where $D_{\phi}(X, Y) = \sum_{i,j}X_{ij} \log \frac{X_{ij}}{Y_{ij}} - \sum_{i,j}X_{ij} + \sum_{i,j}Y_{ij}$ is the Bregman divergence based on the negative entropy function $\phi$. It has been shown that the subproblem~\eqref{opt: subproblem} can be equivalently stated via its two-dual formulation~\cite{tang2024safe}:
\begin{equation}
\begin{aligned}
\label{opt: two-dual formulation of subproblem}
    \min_{\vect{\gamma} \in \R^n, \vect{\zeta} \in \R^m}   Z^k(\vect{\gamma}, \vect{\zeta}) :=  -\vect{a}^\top \vect{\zeta} - \vect{b}^\top \vect{\gamma}
     + \eta \sum_{i=1}^m \sum_{j=1}^n X_{ij}^k\exp\{(\zeta_i + \gamma_j - C_{ij})/\eta\}.
\end{aligned}
\end{equation}
Under this framework, the Sinkhorn algorithm can be interpreted as a simple alternating minimization algorithm that iteratively minimizes $Z^k(\vect{\gamma}, \vect{\zeta})$ with respect to $\vect{\gamma}$ and $\vect{\zeta}$. 

To simplify the optimization landscape, especially for our Newton-type method, we work with the semi-dual formulation. This form is derived by eliminating one set of dual variables. Without loss of generality, we proceed to eliminate the dual variable $\vect{\zeta}$ by letting $\frac{\partial Z^k(\vect{\gamma}, \vect{\zeta})}{\partial \vect{\zeta}} = \0_m$, i.e., 
\begin{equation}
\label{equ: best zeta given gamma}
    \zeta(\vect{\gamma})_i = \eta \log a_i - \eta \log \Big(\sum_{j=1}^n X_{ij}^k \exp\{(\gamma_j - C_{ij})/\eta\}\Big).
\end{equation}
Substituting this back into the two-dual problem~\eqref{opt: two-dual formulation of subproblem}, we obtain the semi-dual formulation that minimizes over only $\vect{\gamma}$:
\begin{equation}
\label{opt: semi-dual problem}
\begin{aligned}
    \min_{\vect{\gamma} \in \R^n} \lyp^k(\vect{\gamma}) := - \sum_{j=1}^n \gamma_jb_j + \eta \sum_{i=1}^m a_i \log \Big(\sum_{j=1}^n X_{ij}^k \exp \left((\gamma_j - C_{ij})/{\eta}\right)\Big).
\end{aligned}
\end{equation}

The following lemma guarantees the existence of an optimal solution for this formulation~\eqref{opt: semi-dual problem}.

\begin{lemma}
\label{lem: semi-dual solution}
    The semi-dual problem~\eqref{opt: semi-dual problem} has a solution. 
\end{lemma}

Let $\vect{\gamma}^{k\star}$ denote an optimal solution to~\eqref{opt: semi-dual problem}, then the corresponding primal update $X^{k+1}$ can be recovered via $X^{k+1} = \diag{\vect{a}} P^k(\vect{\gamma}^{k\star})$, where $P^k(\vect{\gamma})$ is a transport probability matrix with entries:
\begin{equation}
\label{equ: p update}
    \left[P^k(\vect{\gamma})\right]_{ij} 
    = \frac{X_{ij}^k \exp\left((\gamma_j - C_{ij})/{\eta}\right)}{\sum_{l=1}^n X_{il}^k \exp\left((\gamma_l - C_{il})/{\eta}\right)}.
\end{equation}



The partial gradient $g^k$ of $\lyp^k(\vect{\gamma})$ with respect to $\vect{\gamma}$ and its corresponding Hessian $H^k$ are given by
\begin{equation}
\label{equ: grad and hess}
\begin{aligned}
g^k(\vect{\gamma}) = (P^k(\vect{\gamma}))^{\top} \vect{a} - \vect{b}, \quad H^k(\vect{\gamma}) = \frac{1}{\eta} \left(\diag{P^k(\vect{\gamma})^\top\vect{a}} - P^k(\vect{\gamma})^\top \diag{\vect{a}} P^k(\vect{\gamma})\right).
\end{aligned}
\end{equation}

This highly smooth and structured formulation provides major advantages for using Newton methods to solve~\eqref{opt: semi-dual problem}. First, the existence and easy computation of the Hessian enable the use of the Newton method. This guarantees a much faster rate of convergence compared to the sublinear convergence typical of first-order algorithms like Sinkhorn. Second, compared to the classical two-dual formulation, this approach reduces the number of dual variables from $(m+n)$ to $n$. This significantly reduces the memory requirement for the Hessian (from $(m+n) \times (m+n)$ to $n \times n$) and the computational effort required to compute the Newton direction. Third, the Hessian's inherent structure is ideally suited to the sparsification technique we detail in the next subsection. This mathematical property is key to reducing the computational cost without sacrificing accuracy.


\section{Methodology}
\label{sec: Methodology}

This section details the Inexact Bregman Sparse Newton (IBSN) framework, a Newton-type scheme for solving the original OT problem~\eqref{opt: problem}. Section~\ref{subsec: Sparsifying the Hessian matrix} describes the Hessian sparsification strategy designed to enhance computational efficiency. In addition, Section~\ref{subsec: The proposed algorithm} outlines the complete algorithmic framework, detailing the inexact Newton update mechanism and overall procedure. 

\subsection{Sparsifying the Hessian matrix}
\label{subsec: Sparsifying the Hessian matrix}
In this subsection, we detail our strategy for accelerating the computation of Newton directions. For brevity, in the current subproblem iteration $k$, we denote the matrix, gradient, and Hessian by $P = P^k(\vect{\gamma})$, $\vect{g} = g^k(\vect{\gamma})$, and $H = H^k(\vect{\gamma})$, respectively in this subsection. 

The classical Newton method requires solving the linear system $H (\Delta\vect{\gamma}) = -\vect{g}$. For large-scale OT problems, this operation can be computationally expensive, particularly when Hessian-vector products involve dense matrices. Our approach exploits the structure of the transport probability matrix. When the regularization parameter $\eta$ is small, the transport plan and the associated probability matrix often contain many entries that are close to zero. This motivates replacing $P$ with a normalized sparse approximation, so that Hessian-vector products can be computed more efficiently while preserving the dominant curvature information.

Specifically, we construct $P_\rho$ by thresholding small entries of $P$ and then renormalizing each row to preserve the row-stochastic structure. The row indexed by $i^\star = \arg\max_{i\in[m]}a_i$ is kept unchanged, which is essential for maintaining the spectral lower bound used in the subsequent analysis. The full procedure is summarized in Algorithm~\ref{alg: Sparsifying the Hessian matrix}. In practice, the `sparsified' Hessian $H_\rho$ does not need to be formed explicitly; storing the sparse matrix $P_\rho$ is sufficient for efficient Hessian-vector products.

\begin{algorithm}
\caption{Sparsifying the Hessian matrix}
\label{alg: Sparsifying the Hessian matrix}
\begin{algorithmic}[1]
    \STATE {\bfseries Input:} Outer iteration $k$, dual variable $\vect{\gamma}$, threshold $\rho \in [0,1/n)$, $i^\star = \arg \max_{i \in [m]} a_i$.
    \STATE Compute $P \gets P^k(\vect{\gamma})$ using~\eqref{equ: p update}.
    \STATE Construct the sparsified matrix $P_{\rho}$ by setting 
    $$
    [P_{\rho}]_{ij} \gets \frac{P_{ij}\id{P_{ij} \geq \rho}}{\sum_{l=1}^n P_{il}\id{P_{il} \geq \rho}}\ 
    \forall i \ne i^\star,\qquad\ [P_{\rho}]_{i^\star \cdot} \gets P_{i^\star\cdot}.
    $$ 
    \STATE Compute $H_{\rho} \gets \left(\diag{P_\rho^\top \vect{a}} - P_{\rho}^\top \diag{\vect{a}} P_{\rho} \right)/\eta$.
    \STATE {\bfseries Output:} $H_{\rho}$.
\end{algorithmic}

\end{algorithm}

This construction naturally raises two important questions. First, \textbf{is $H_{\rho}$ invertible on $\1_n^\perp$?} This property is crucial for implementing the Newton step, as computing the direction $\Delta \vect{\gamma}$ requires solving the linear system $H_{\rho}(\Delta \vect{\gamma}) = -\vect{g}$. Ensuring invertibility guarantees that the system admits a unique solution and that the Newton direction is well defined. Second, \textbf{how accurate is $H_{\rho}$ as an approximation of the true Hessian $H$}? A quantitative understanding of this discrepancy guides the choice of the sparsification threshold $\rho$, balancing computational efficiency with approximation accuracy.

To address these questions, we first establish the structural properties of $H_{\rho}$. The following result shows that the sparsified Hessian constructed by Algorithm~\ref{alg: Sparsifying the Hessian matrix} remains positive definite on the feasible subspace, ensuring that the Newton step is always well defined.

\begin{theorem}
\label{thm: positive semidefiniteness}
The matrix $H_{\rho}$ output by Algorithm~\ref{alg: Sparsifying the Hessian matrix} is symmetric and positive semidefinite, and its kernel satisfies $\text{ker}(H_\rho) = \text{span}\{\1_n\}$ for any $\rho \in [0,1/n)$. 
\end{theorem}

Theorem~\ref{thm: positive semidefiniteness} implies that, for any $\rho\in[0,1/n)$, $H_{\rho}$ is positive definite on the orthogonal complement $\1_n^{\perp}$ of the kernel. Note that the gradient $\vect{g}$ satisfies $\mathbf{1}_n^\top \vect{g} = 0$, so the Newton system can be efficiently solved using the conjugate gradient (CG) method. The next result gives spectral bounds for the Hessian on this restricted subspace.

\begin{theorem}
\label{thm: min and max eigenvalues}
    Let $H_{\rho}$ be the matrix output by Algorithm~\ref{alg: Sparsifying the Hessian matrix}, then for any $\rho \in [0,1/n)$,
    \begin{align*}
        \lambda_{\min}^+ \left(H_{\rho}\right) \geq n a_{i^\star} p^2 / \eta, \ \quad 
        \lambda_{\max} \left(H_{\rho} \right) \leq 1 / (2\eta),
    \end{align*}
    where $p := \min_{j} [P]_{i^\star j}$.
\end{theorem}

These bounds ensure that $H_{\rho}$ is well-conditioned on the feasible subspace, providing stability for the Newton update and reliable convergence behavior across different sparsification levels.

To address the second question regarding accuracy, we establish a quantitative bound on the approximation error between the exact and sparsified Hessians.

\begin{theorem}
\label{thm: error bound for hessian matrix}
Let $H_{\rho}$ be the matrix output by Algorithm~\ref{alg: Sparsifying the Hessian matrix}, then
$
\|H - H_{\rho}\| \leq \frac{6mn\|\vect{a}\|_{\infty}}{\eta} \rho.
$
\end{theorem}


Theorem~\ref{thm: error bound for hessian matrix} shows that the approximation error increases with the problem size $m, n$ and decreases with the regularization parameter $\eta$. To maintain a small and stable error across different scales, the threshold $\rho$ can be chosen proportional to $\eta / (mn)$. This choice balances sparsity and accuracy while ensuring consistent performance across varying values of $m$, $n$, and $\eta$.

The preceding analysis confirms that the Hessian sparsification technique constructed via Algorithm~\ref{alg: Sparsifying the Hessian matrix} yields an approximation $H_{\rho}$ that is both structurally robust and quantitatively accurate. However, even with these sparsification and second-order acceleration techniques, solving the problem~\eqref{opt: semi-dual problem} to high precision in every iteration remains computationally prohibitive. To fully address this efficiency bottleneck, we adopt the inexact Bregman framework proposed by~\cite{yang2022bregman}. This framework allows the subproblems $\eqref{opt: subproblem}$ to be solved approximately while still ensuring global convergence to the optimal OT solution. The details of this inexact criterion and its integration into our method are shown in the next subsection. 

\subsection{The proposed algorithm}
\label{subsec: The proposed algorithm}

\begin{algorithm}[t]
\caption{Inexact Bregman Sparse Newton (IBSN)}
\label{alg: IBSN}
\begin{algorithmic}[1]
    \STATE {\bfseries Input:} $X^0 = \vect{a}\vect{b}^\top$, $\vect{\gamma}^{0,0}$, $\vect{\zeta}^{0,0}$, $\{\mu_k\}_k$, $\eta \in (0, m/2)$, Armijo parameter $\sigma = 10^{-4}, \beta=0.8$.
    \FOR{$k=0,..., K-1$}
    \STATE Solve subproblem~\eqref{opt: two-dual formulation of subproblem} using the Sinkhorn algorithm with $\vect{\gamma}^{k,0}, \vect{\zeta}^{k,0}$ to obtain $\vect{\gamma}^{k,1}, \vect{\zeta}^{k,1}$. 
    \STATE Update $\vect{\gamma}^{k,1} \gets \vect{\gamma}^{k,1} - \frac{1}{n}\1_n^\top \vect{\gamma}^{k,1}\1_n$. 
    \FOR{$v=1,2,...$}
    \STATE Compute the gradient $\vect{g}^{k,v} =  g^k(\vect{\vect{\gamma}}^{k, v})$ and update $\rho \gets \frac{\eta}{mn} \left\|\vect{g}^{k,v}\right\|$.
    \STATE Construct $H^{k,v}_{\rho}$ by Algorithm~\ref{alg: Sparsifying the Hessian matrix}.
    \STATE Compute $\Delta \vect{\gamma}^{k,v} \gets -\left(H^{k,v}_{\rho} + \|\vect{g}^{k,v}\|I\right)^{-1}\vect{g}^{k,v}$.
    \STATE Select the stepsize $t^{k,v}$ by Armijo backtracking.
    \STATE Compute $\vect{\gamma}^{k, v+1} \gets \vect{\gamma}^{k, v} + t^{k,v} \Delta \vect{\gamma}^{k,v}$.
    \STATE Compute $X^{k,v+1} \gets \diag{\vect{a}} P^k(\vect{\gamma}^{k, v+1})$.
    \IF{$D_{\phi}(\Proj_{\Omega}(X^{k,v+1}), X^{k,v+1}) \leq \mu_k$}
    \STATE Obtain $X^{k+1} \gets X^{k, v+1}$, $\vect{\gamma}^{k+1,0} \gets \vect{\gamma}^{k, v+1}$.
    \STATE Calculate $\vect{\zeta}^{k+1, 0}$ by~\eqref{equ: best zeta given gamma} and \textbf{break}.
    \ENDIF
    \ENDFOR
    \ENDFOR
    \STATE {\bfseries Output:} $X^{K}$.
\end{algorithmic}    
\end{algorithm}

Algorithm~\ref{alg: IBSN} summarizes the proposed Inexact Bregman Sparse Newton (IBSN) method for solving the OT problem~\eqref{opt: problem}. At each outer iteration, IBSN first applies a few Sinkhorn iterations to the two-dual subproblem~\eqref{opt: two-dual formulation of subproblem}, producing a warm start for the subsequent Newton refinement in the semi-dual space.

During the inner Newton phase, the gradient $\vect{g}^{k,v}$ is used to adapt the sparsification threshold $\rho = \eta\|\vect{g}^{k,v}\|/(mn)$. This adaptive choice satisfies the admissible range $\rho < 1/n$ required by Algorithm~\ref{alg: Sparsifying the Hessian matrix} under the mild condition $\eta < m/2$. Indeed, since both $(P^k(\vect{\gamma}))^\top \vect{a}$ and $\vect{b}$ are probability vectors, $\|\vect{g}^{k,v}\|\le2$, which gives $\rho\le\frac{2\eta}{mn} < \frac1n$. A larger residual yields a sparser Hessian approximation and reduces the cost of early iterations, while $\rho$ decreases as the iterate approaches the subproblem solution, leading to a more accurate curvature approximation. The shifted system $\left(H^{k,v}_{\rho}+\|\vect{g}^{k,v}\|I\right)\Delta\vect{\gamma}^{k,v}
= -\vect{g}^{k,v}$ is then solved to obtain a stable descent direction, followed by Armijo backtracking.

The inner loop terminates once the inexact Bregman criterion $D_{\phi}(\Proj_{\Omega}(X^{k,v+1}), X^{k,v+1}) \leq \mu_k$ is satisfied. The resulting primal and dual variables are used to initialize the next outer iteration. This criterion follows the inexact framework of~\cite{yang2022bregman} and avoids solving every subproblem to high precision. In implementation, the rounding procedure of~\cite{altschuler2017near} can be used as an efficient approximation of $\Proj_{\Omega}$; see Appendix~\ref{app: Converting a matrix to a transportation plan}. Note that our convergence analysis assumes exact solutions of the Newton systems, whereas the practical implementation uses CG with a prescribed stopping tolerance.

\section{Theoretical Analysis}
\label{sec: Theoretical Analysis}

This section presents the convergence analysis of the proposed IBSN framework. We first establish the inner convergence of the Newton iterations with sparsified Hessians, then prove the global convergence of the outer iterations to the exact optimal transport solution. 

\subsection{Inner convergence analysis}

This subsection establishes the crucial convergence properties of the inner iterations in the IBSN framework, demonstrating that our sparse Newton steps converge to the unique solution of each EOT subproblem. We start with the following lemma, which validates our approach by confirming that the entire optimization sequence remains confined to this desirable subspace.

\begin{lemma}
\label{lem: sequence restrict in the subspace}
Let $\{\vect{\gamma}^{k,v}\}$ be the sequence generated by Algorithm~\ref{alg: IBSN} with $\1_n^\top\vect{\gamma}^{0,0} = 0$, then $\1_n^\top\vect{\gamma}^{k,v} = 0$ for $\forall k, v$.
\end{lemma}


Lemma~\ref{lem: sequence restrict in the subspace} ensures that the convergence analysis can be safely restricted to the subspace $\1_n^\perp$, effectively eliminating the redundant degree of freedom and restoring the strong convexity of $\lyp^k$.

We then establish that the inner iterates $\vect{\gamma}^{k,v}$ generated by Algorithm~\ref{alg: IBSN} converge to the unique global minimizer of $\lyp^k$ within the subspace $\1_n^\perp$.

\begin{theorem}
\label{thm:inner-convergence}
For a fixed outer $k$-th iteration, let $\{\vect{\gamma}^{k,v}\}_v$ be the sequence generated by Algorithm~\ref{alg: IBSN} with $\1_n^\top\vect{\gamma}^{k,0} = 0$, then we have 
$\|\vect{g}^{k,v}\| \to 0$ as $v\to\infty$ and
$\vect{\gamma}^{k,v} \to \vect{\gamma}^{k\star}$, where $\vect{\gamma}^{k\star}$ is the unique minimizer of $\lyp^k$ in the slice $\{{\vect{\gamma} \mid \1_n^\top \vect{\gamma} = 0}\}$.
\end{theorem}

Next, we examine the local convergence rate achieved by the sparse Newton iterations.

\begin{theorem}
\label{thm:sparse-newton-rate}
For a fixed outer $k$-th iteration, let $\vect{\gamma}^{k\star}$ denote the unique solution in the slice $\{\vect{\gamma} \mid \1_n^\top\vect{\gamma}=0\}$ of $\lyp^k$. 
When the iterates $\vect{\gamma}^{k,v}$ are sufficiently close to $\vect{\gamma}^{k\star}$ such that the full Newton step $t^{k,v} = 1$ is accepted, there exists a constant $L > 0$ such that
\begin{equation*}
\|\vect{\gamma}^{k, v+1} - \vect{\gamma}^{k\star}\| \leq L \|\vect{\gamma}^{k, v} - \vect{\gamma}^{k\star}\|^2.
\end{equation*}

\end{theorem}

Together, Theorems~\ref{thm:inner-convergence} and~\ref{thm:sparse-newton-rate} establish that the proposed sparse Newton updates converge globally to the unique solution of each subproblem and exhibit quadratic local convergence once sufficiently close to the optimum.
This demonstrates that the sparsified Newton steps preserve the accuracy and convergence speed of the classical method, making this approach feasible for large-scale optimal transport problems.

\paragraph{Complexity Analysis of a Newton Direction.} We briefly discuss the cost of computing a Newton direction for the EOT subproblem~\eqref{opt: EOT}. Assume for simplicity that $m=n$, the matrix $P_\rho \in \R^{n \times n}$ has an average of $l$ nonzero entries per row. 
Then $P_\rho$ contains $O(l n)$ nonzero entries, and each Hessian-vector product with $H_\rho$ can be computed in $O(l n)$ operations. Consequently, each CG iteration also costs $O(l n)$. When $l \ll n$, as is typically observed after effective sparsification, this is substantially cheaper than the $O(n^2)$ cost associated with a dense representation.

Let $\epsilon_{\rm cg}$ be the CG tolerance and $\kappa$ denote the condition number of $H_\rho + \|\vect g\| I$ restricted to $\1_n^\perp$. The number of CG iterations required is $O(\sqrt{\kappa}\log(1/\epsilon_{\rm cg}))$~\cite[Chapter 10]{shewchuk1994introduction}. Applying the spectral bounds from Theorem~\ref{thm: min and max eigenvalues}, we can bound this condition number as
\begin{align*}
\kappa \leq \lambda_{\max}(H)/\lambda_{\min}(H) \leq \frac{1/(2\eta)}{na_{i^\star}p^2 / \eta} = O(p^{-2}),
\end{align*}
where $p=\min_j P_{i^\star j}$ and $a_{i^\star}\geq 1/n$. For the preserved row $i^\star$, let $s_j := \eta\log X^k_{i^\star j}+\gamma_j-C_{i^\star j},$. We can bound $p$ from below:
\begin{align*}
p = \min_j P_{i^\star j}=\min_j\frac{\exp(s_j/\eta)}{\sum_{l=1}^n \exp(s_l/\eta)}\geq\frac{\exp(\min_l s_l/\eta)}{n\exp(\max_l s_l/\eta)}:=\frac{1}{n}\exp\left(-\frac{\Delta}{\eta}\right),
\end{align*}
where $\Delta=\max_l s_l-\min_l s_l$ represents the spread of the exponentiated terms. 

Therefore, ignoring the logarithmic dependence on $\epsilon_{\rm cg}$, the total cost of a single CG solve is $O\left(ln\,p^{-1}\right) = O\left(ln^2\exp\left(\Delta/\eta\right)\right)$. This theoretical bound shows that the linear system becomes increasingly expensive to solve as the problem dimension $n$ increases or the regularization parameter $\eta$ decreases. In particular, a smaller $\eta$ sharpens the softmax distribution, which can exponentially amplify the factor $\exp(\Delta/\eta)$. 

\subsection{Global convergence}





We now establish the global convergence of the outer IBSN iterations. Let $f(X)=\langle C,X\rangle$, $f^\star=\min_{X\in\Omega}f(X)$, and $\widetilde X^k:=\Proj_\Omega(X^k)$. The outer iteration is a specialization of the inexact Bregman proximal point framework in~\cite{yang2022bregman}. In our OT formulation, the residual and subgradient-error terms in their framework vanish, so the convergence behavior is determined by the inexactness sequence~$\{\mu_k\}$.

If $\sum_{k=0}^\infty \mu_k<\infty$, then $\{X^k\}$ and $\{\widetilde X^k\}$ converge to the same optimal solution of problem~(\ref{opt: problem}). For the polynomial schedule $\mu_k=O((k+1)^{-p})$, any $p>1$ therefore guarantees sequence convergence. The last-iterate rate depends on the stronger weighted summability of $\{\mu_k\}$: $p=2$ yields
\begin{equation}
f(\widetilde X^K)-f^\star=O\left(\frac{\eta\log K}{K}\right),
\end{equation}
whereas any $p>2$ gives
\begin{equation}
f(\widetilde X^K)-f^\star=O\left(\frac{\eta}{K}\right).
\end{equation}
The detailed specialization and rate analysis are provided in Appendix~\ref{app:inexact-bregman-convergence}.

\section{Numerical Experiments}
\label{sec: Numerical Experiments}
In this section, we evaluate the performance of the proposed IBSN algorithm through a series of numerical experiments on both synthetic and real datasets and compare its empirical performance with existing state-of-the-art optimal transport solvers: 
1. PINS algorithm~\cite{wu2025pins} (based on its implementation at the time of our submission, prior to its second version); 2. HOT algorithm~\cite{zhang2025hot}; 3. Inexact Bregman framework with Sinkhorn (IBSink)~\cite{yang2022bregman}; 4. IPOT algorithm~\cite{xie2020fast}; 5. Extragradient (ExtraGrad) method~\cite{li2025fast}. All numerical experiments are implemented in Python 3.9.18 on a MacBook Pro running macOS 13.0 with an Apple M2 Pro processor and 16 GB RAM. 


\begin{figure*}[t]
\centering
{
\resizebox*{0.32 \textwidth}{!}{\includegraphics{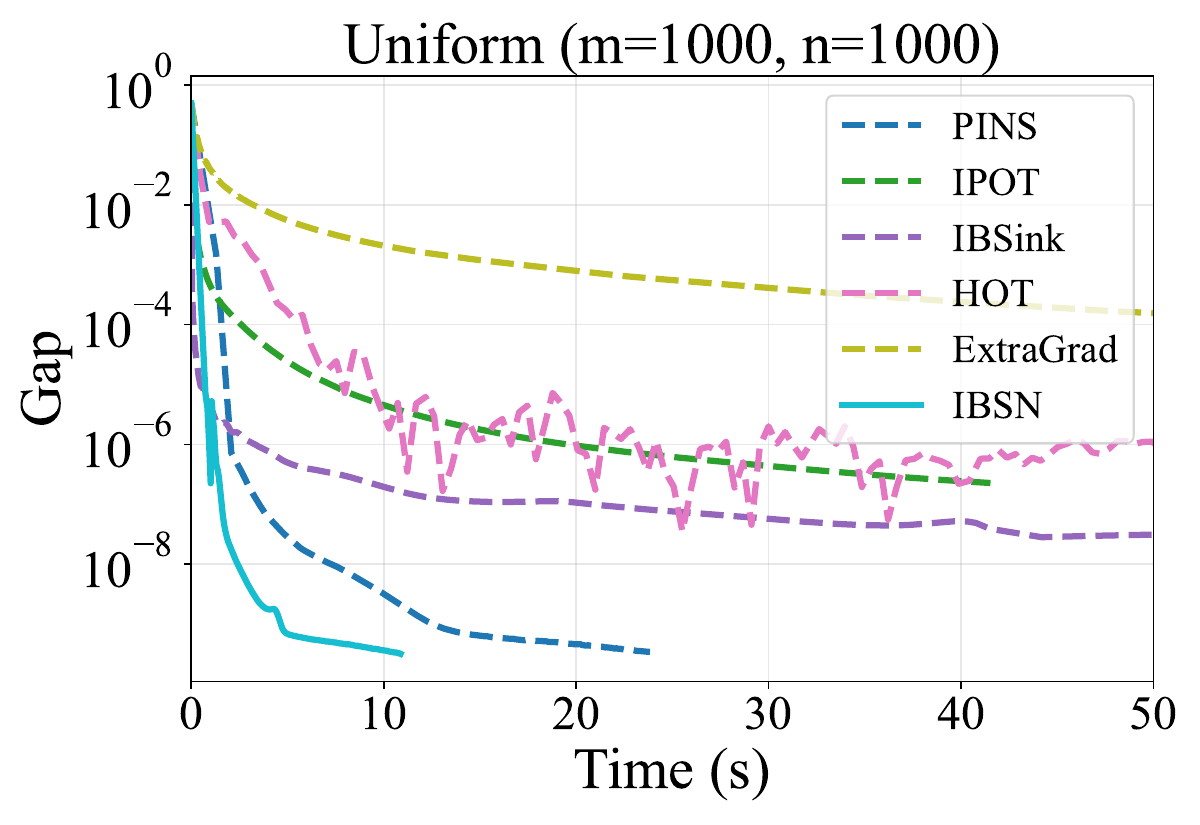}}}\hfill
{
\resizebox*{0.32 \textwidth}{!}{\includegraphics{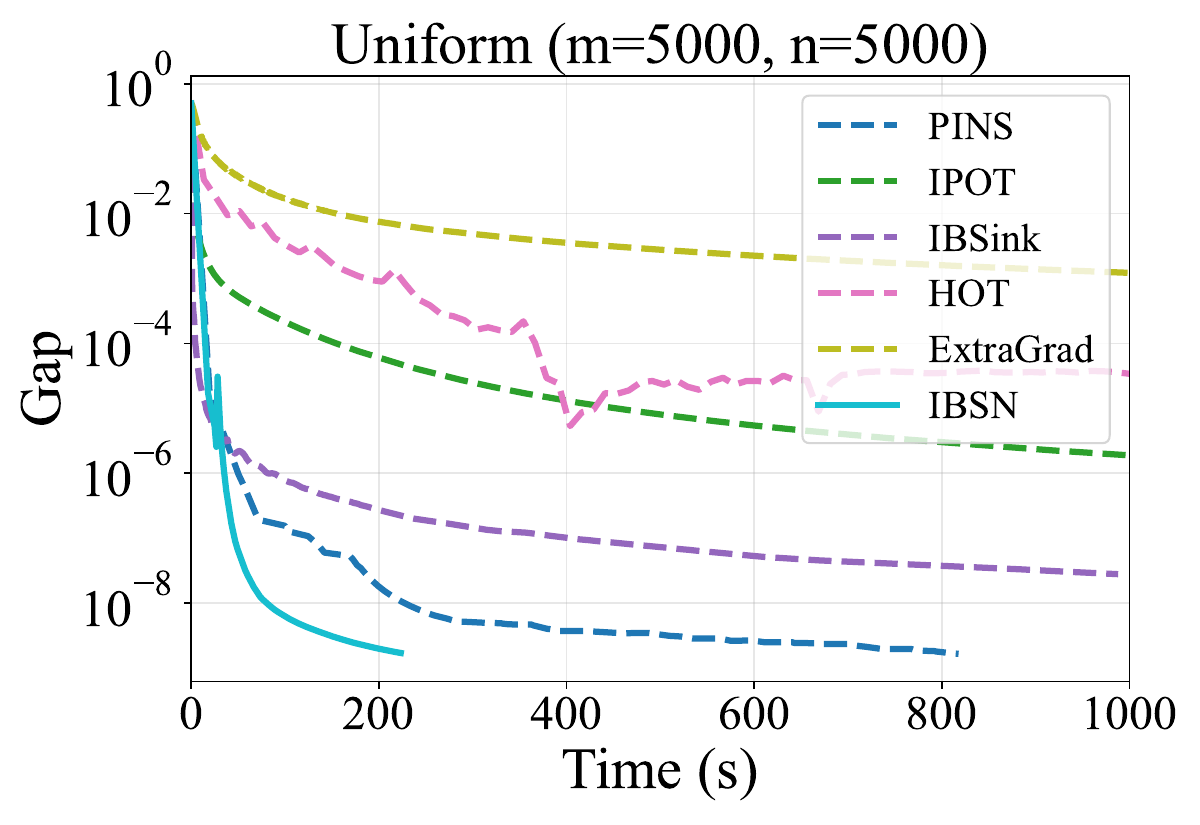}}}\hfill
{
\resizebox*{0.32 \textwidth}{!}{\includegraphics{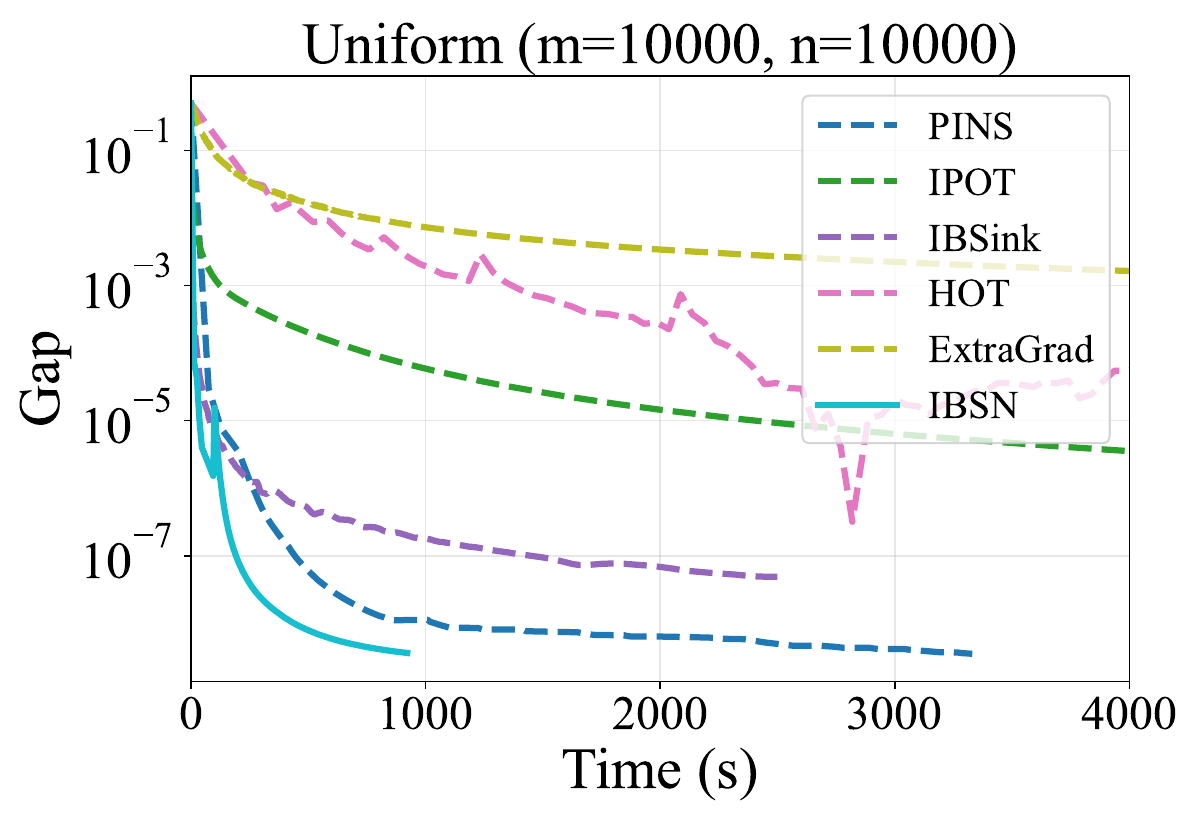}}}
{
\resizebox*{0.32 \textwidth}{!}{\includegraphics{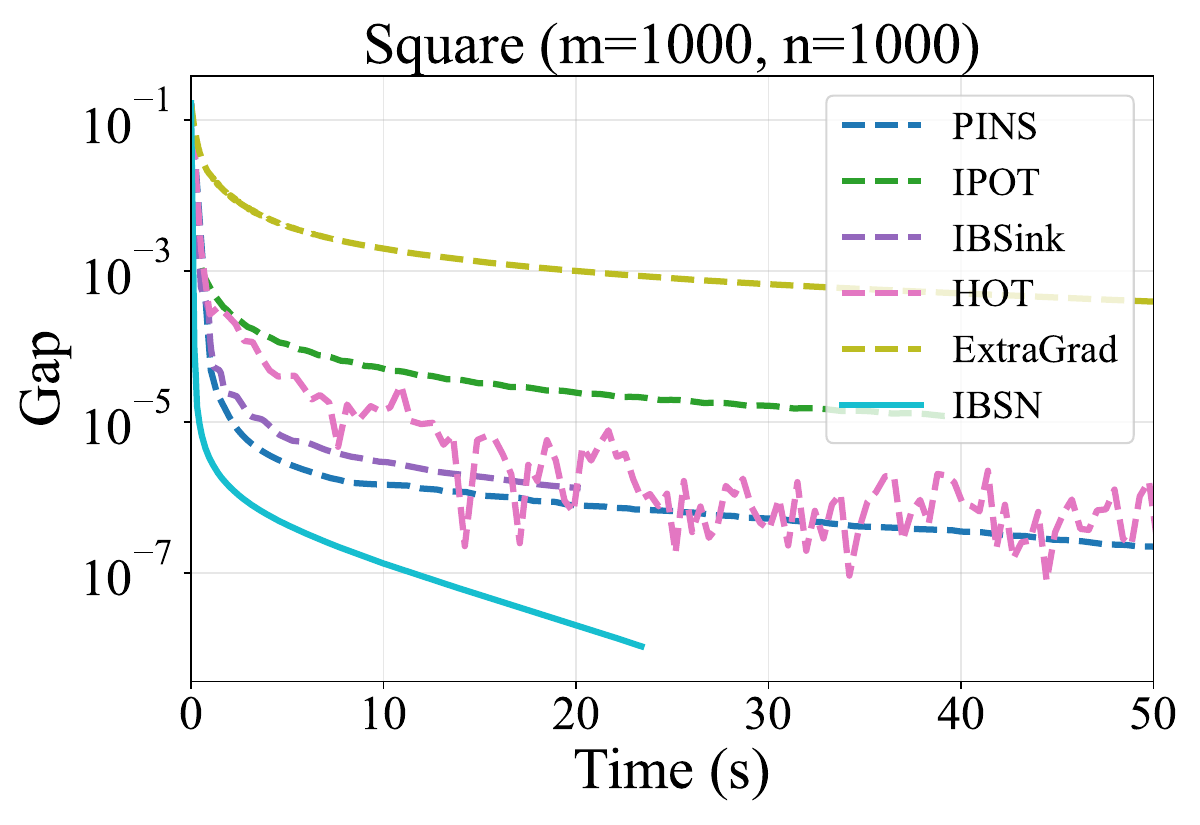}}}\hfill
{
\resizebox*{0.32 \textwidth}{!}{\includegraphics{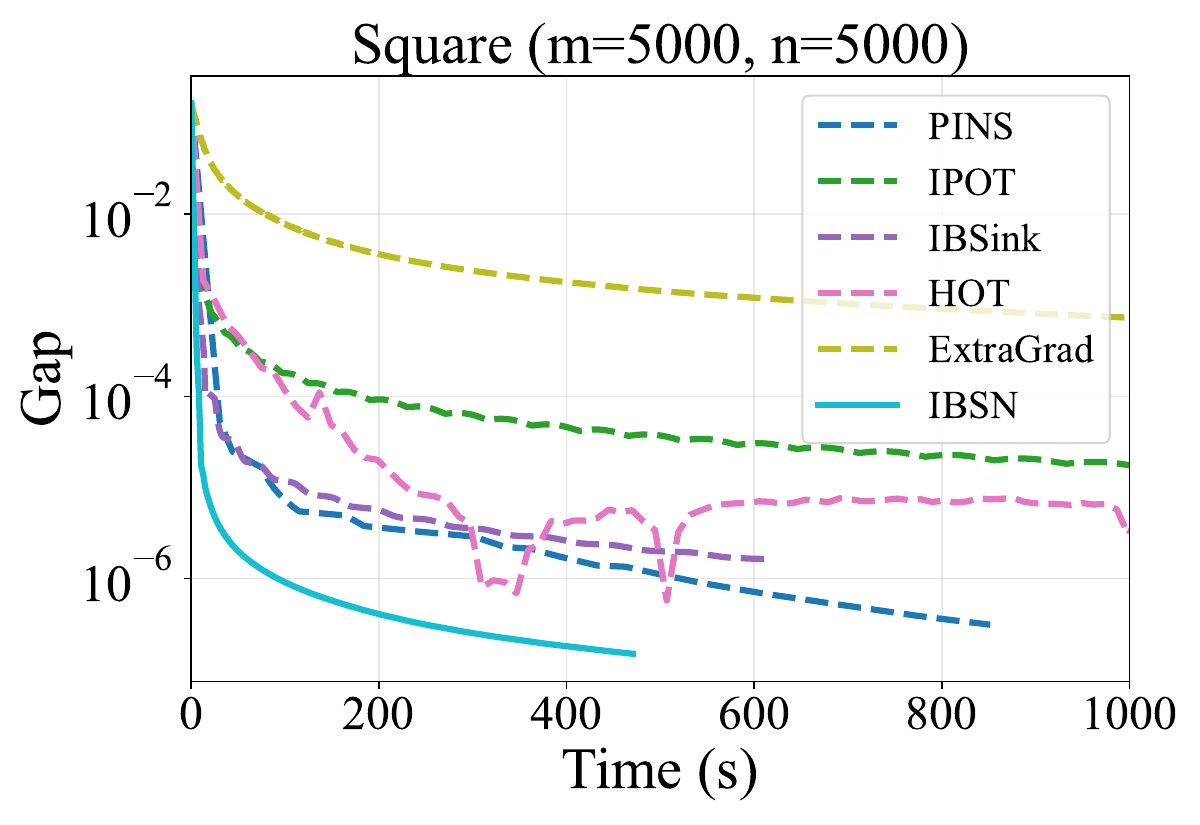}}}\hfill
{
\resizebox*{0.32 \textwidth}{!}{\includegraphics{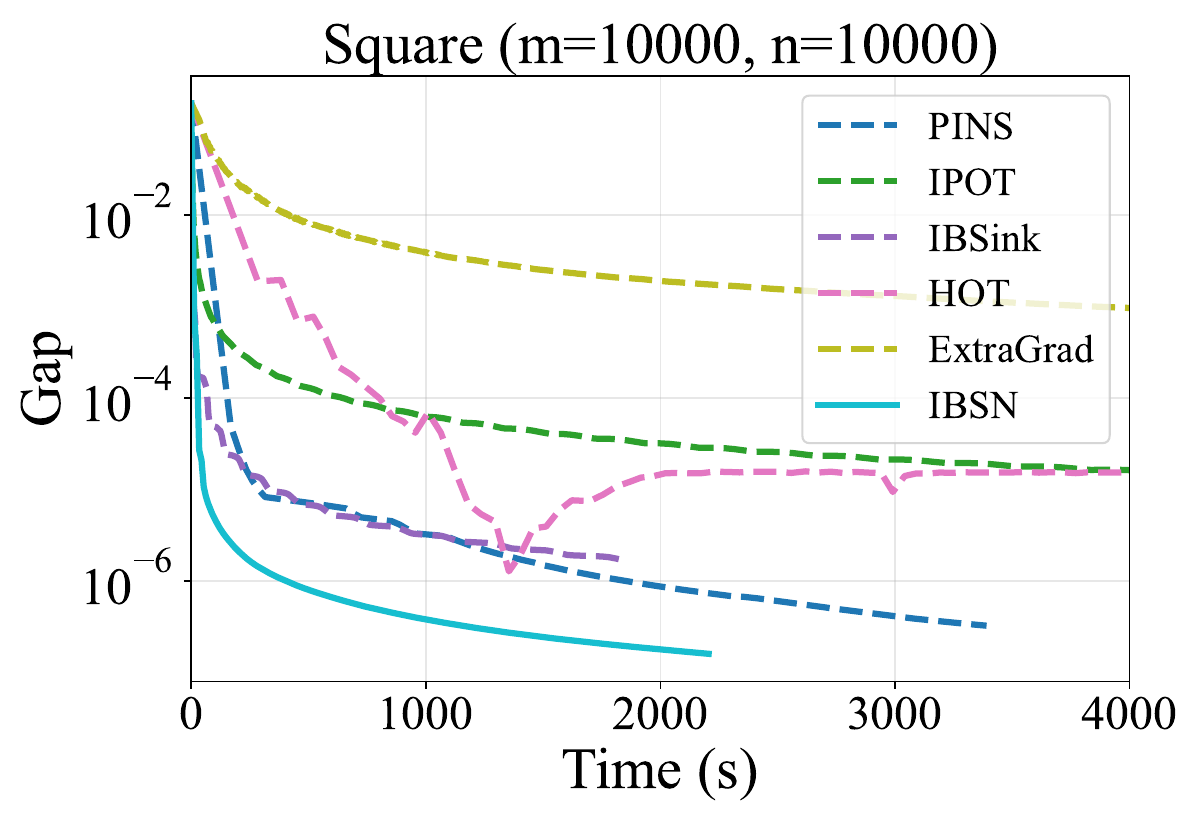}}}
\caption{Performance of different algorithms on synthetic data. Top: Uniform. Bottom: Square.}
\label{fig: synthetic comparison of methods}
\end{figure*}

\paragraph{Setup} To ensure consistency, we normalize all cost matrices to have a unit infinity norm, i.e., $C \leftarrow C / \max_{i,j}|C_{ij}|$. We initialize the transport plan as $X^0 = \vect{a}\vect{b}^\top$. We use the relative KKT residual~\cite{yang2022bregman} ($\Delta_{\text{kkt}}$) as the unified stopping criterion for all methods (see Appendix~\ref{app: Derivation of relative KKT residual}). An algorithm is terminated once $\Delta_{\text{kkt}} < 10^{-11}$ or the maximum number of iterations is attained. 

Within the inner iteration of IBSN, the Sinkhorn warm-up phase stops once the gradient norm is less than $10^{-3}$. In addition, we terminate the Newton phase when
\begin{equation}
\label{equ: stopping criterion}
\begin{aligned}
D_{\phi} (\Proj_{\Omega} (X^{k+1}) , X^{k+1}) \leq \min\{m,n\} \cdot \mu_k, \quad \text{where}\ \mu_k = \max\left\{ 10^{-4}/(k+1)^2, 10^{-11} \right\},
\end{aligned}
\end{equation}
where $\Proj_{\Omega}$ is computed using the rounding procedure in~\cite{altschuler2017near}. To save computational time, we only evaluate the condition~\eqref{equ: stopping criterion} after a simpler gradient norm check ($\left\| \nabla \lyp^k(\vect{\gamma}^{k, v+1}) \right\| < \mu_k$) passes (as suggested in~\cite{yang2022bregman}). To compare the solution quality, we utilize the objective gap relative to a high-precision ground truth $X^\star$ (obtained via Gurobi with tolerance $10^{-8}$):
\begin{align*}
\text{Gap} = \left|\langle C, \Proj_{\Omega}(X) \rangle - \langle C, X^\star\rangle \right|.
\end{align*}

\subsection{Synthetic data}
\label{subsec: Synthetic data}

We generate synthetic datasets using the following procedure. The marginal distributions $\vect{a}$ and $\vect{b}$ are sampled independently from a uniform distribution on $(0,1)$ and subsequently normalized to satisfy $\1_m^\top\vect{a} = \1_n^\top\vect{b} = 1$. We evaluate performance across problem dimensions $m = n \in \{1000, 5000, 10000\}$ using two types of cost matrices:

\begin{itemize}    
    \item \textbf{Uniform}: Each entry of $C$ is drawn independently from the uniform distribution on $(0,1)$. 
    \item \textbf{Square}: The cost matrix $C$ is defined by the squared $\ell_2$-distance between grid coordinates. 
\end{itemize}

We set $\eta = 10^{-4}$ for both IBSN and PINS, $\eta=10^{-3}$ for IBSink, $\eta=10^{-1}$ for IPOT. These values are selected based on empirical tuning and correspond to the best observed trade-off between convergence speed and solution accuracy for each method. A more detailed analysis of the impact of the regularization parameter $\eta$ is provided in Appendix~\ref{app: Impact of the regularization parameter}.

The experimental results are presented in Figure~\ref{fig: synthetic comparison of methods}. The second-order methods (IBSN and PINS) converge faster than the first-order methods in most cases. This exploits sparse matrix operations to reduce per-iteration cost and achieve fast convergence with second-order methods. In addition, IBSN surpasses PINS in overall efficiency. This advantage stems from our semi-dual formulation, which further reduces the cost of the second-order updates.



\subsection{Real data}


\begin{figure*}[t]
\includegraphics[width=0.32\textwidth]{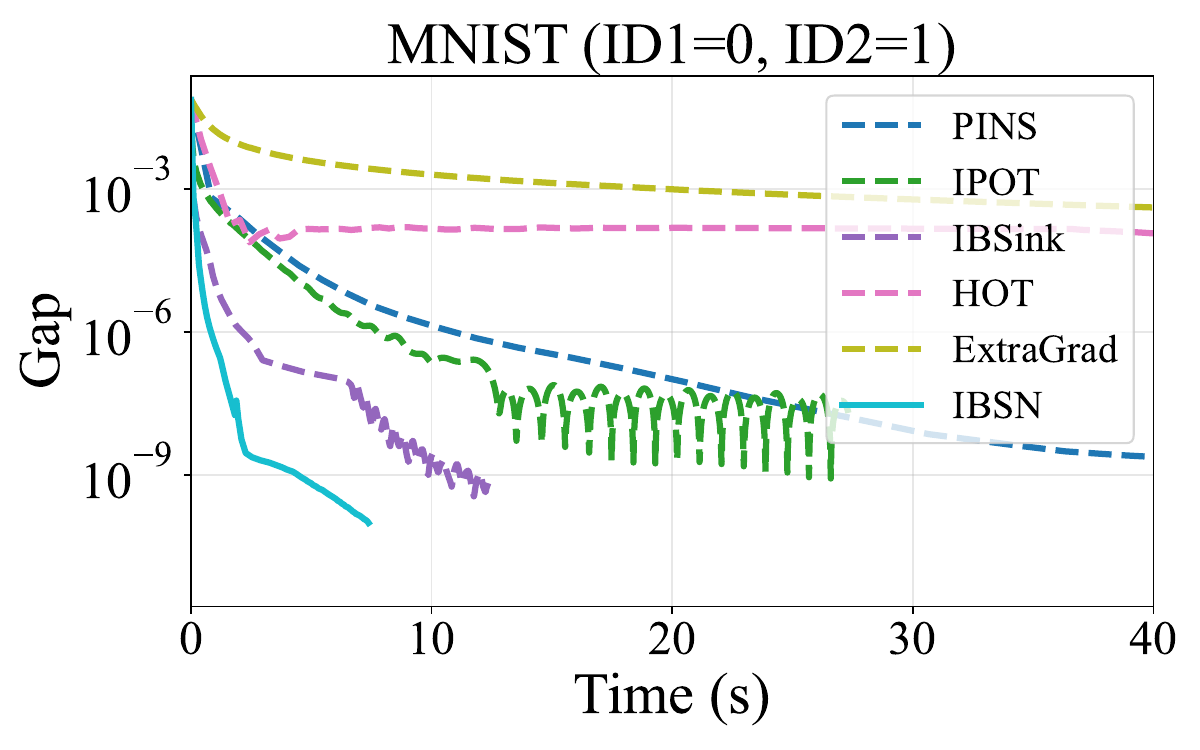}\hfill
\includegraphics[width=0.32\textwidth]{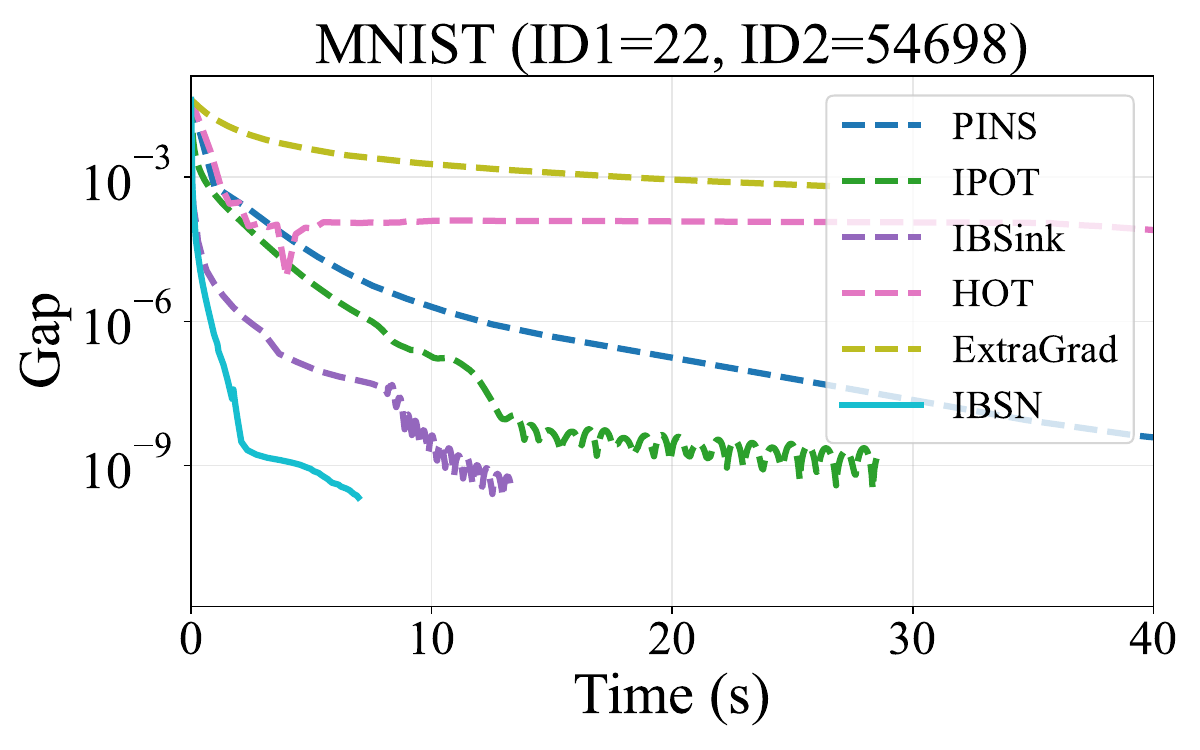}\hfill
\includegraphics[width=0.32\textwidth]{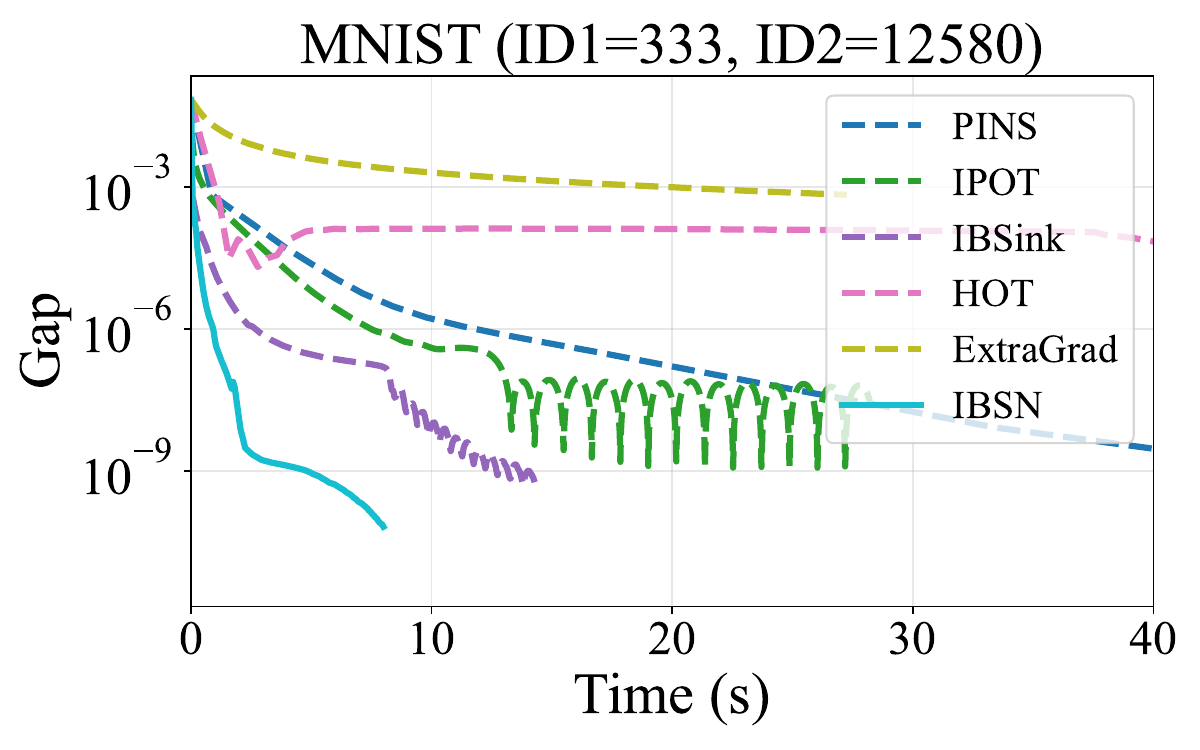}
\includegraphics[width=0.32\textwidth]{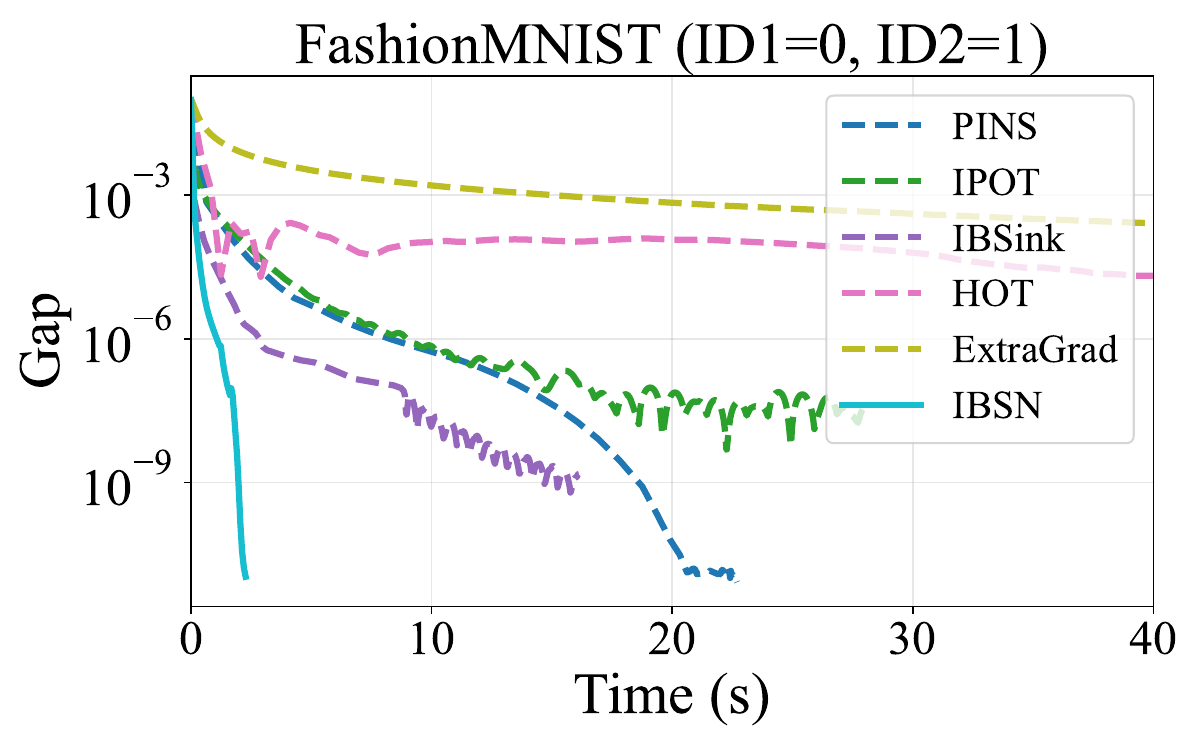}\hfill
\includegraphics[width=0.32\textwidth]{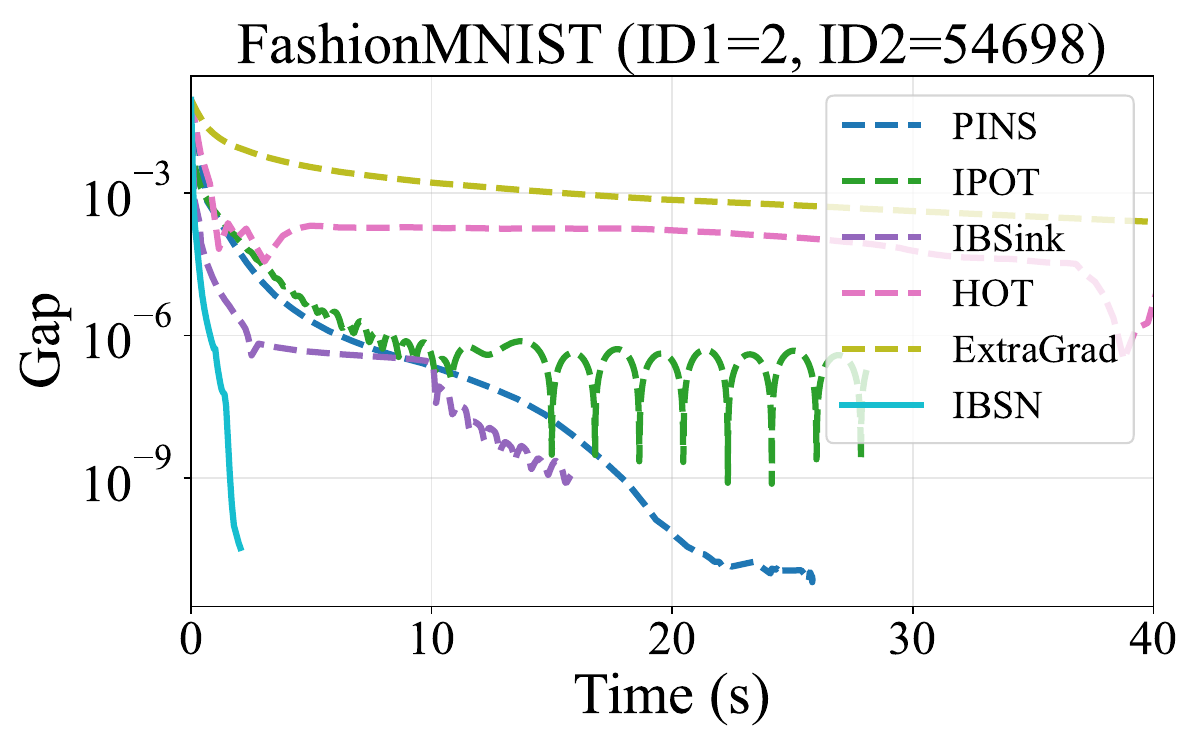}\hfill
\includegraphics[width=0.32\textwidth]{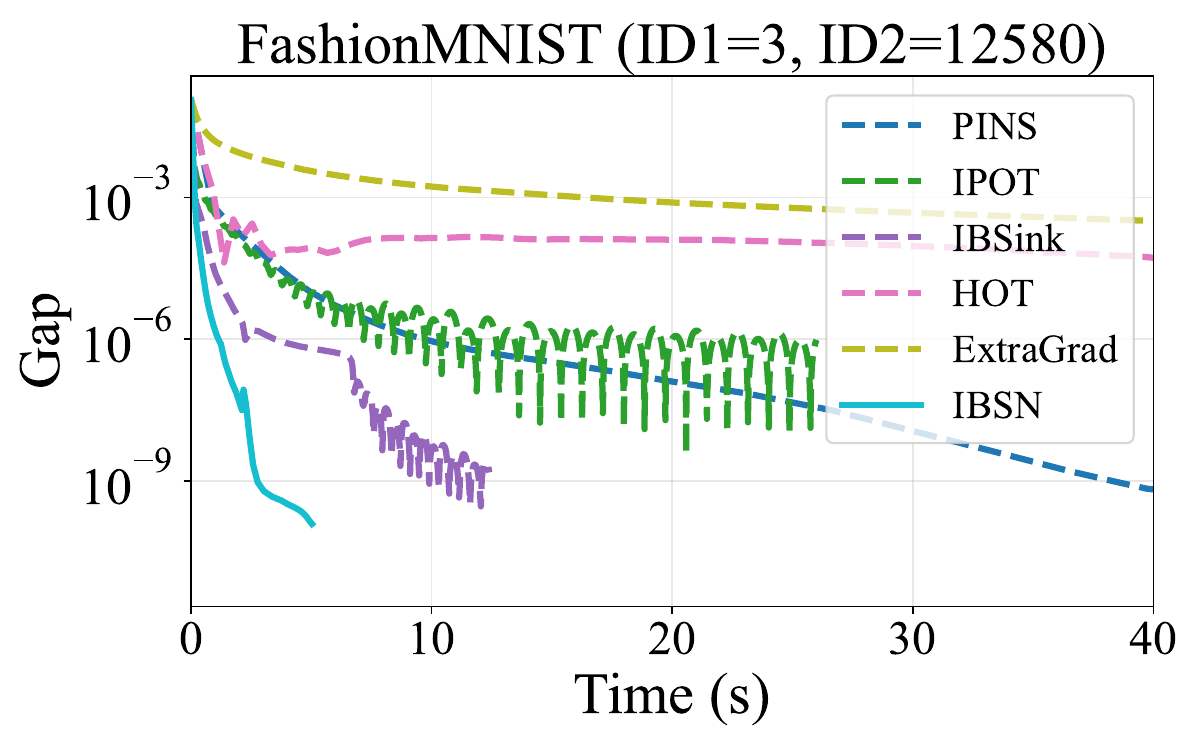}
\caption{Performance of different algorithms on MNIST (Top) and Fashion-MNIST (Bottom).}
\label{fig: Performance of different algorithms on real data}
\end{figure*}

\begin{figure*}[t]
\includegraphics[width=0.32\textwidth]{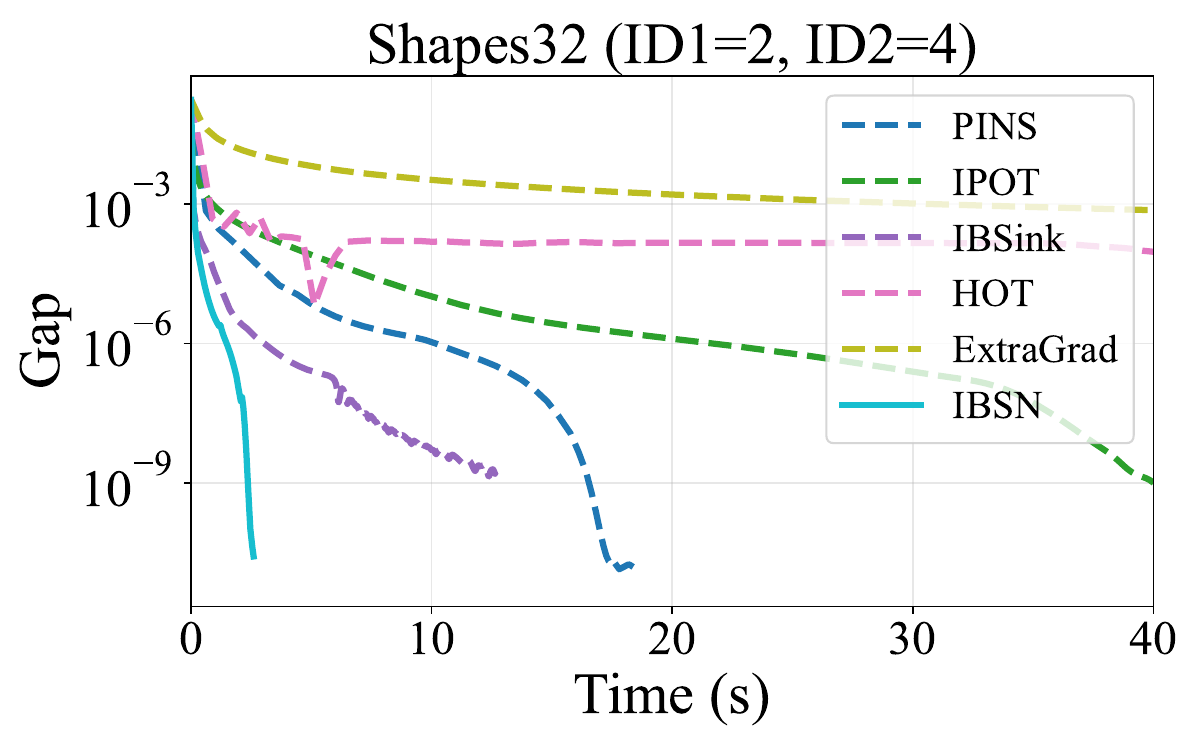}\hfill
\includegraphics[width=0.32\textwidth]{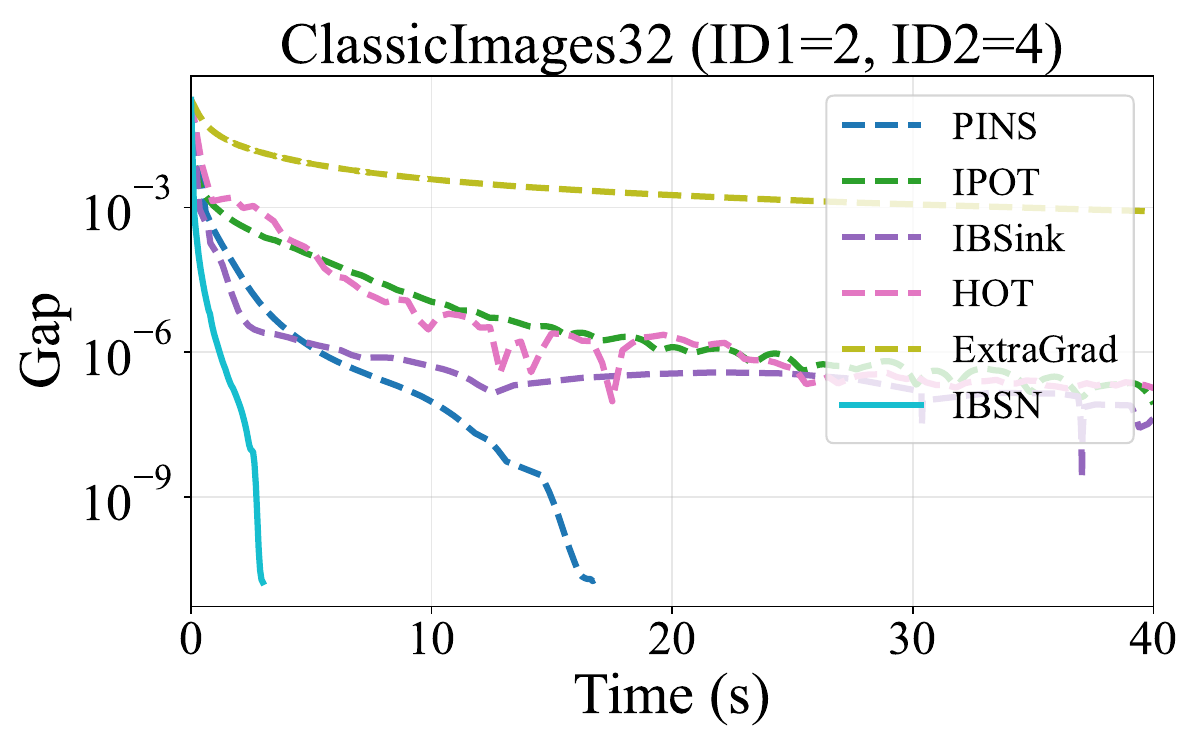}\hfill
\includegraphics[width=0.32\textwidth]{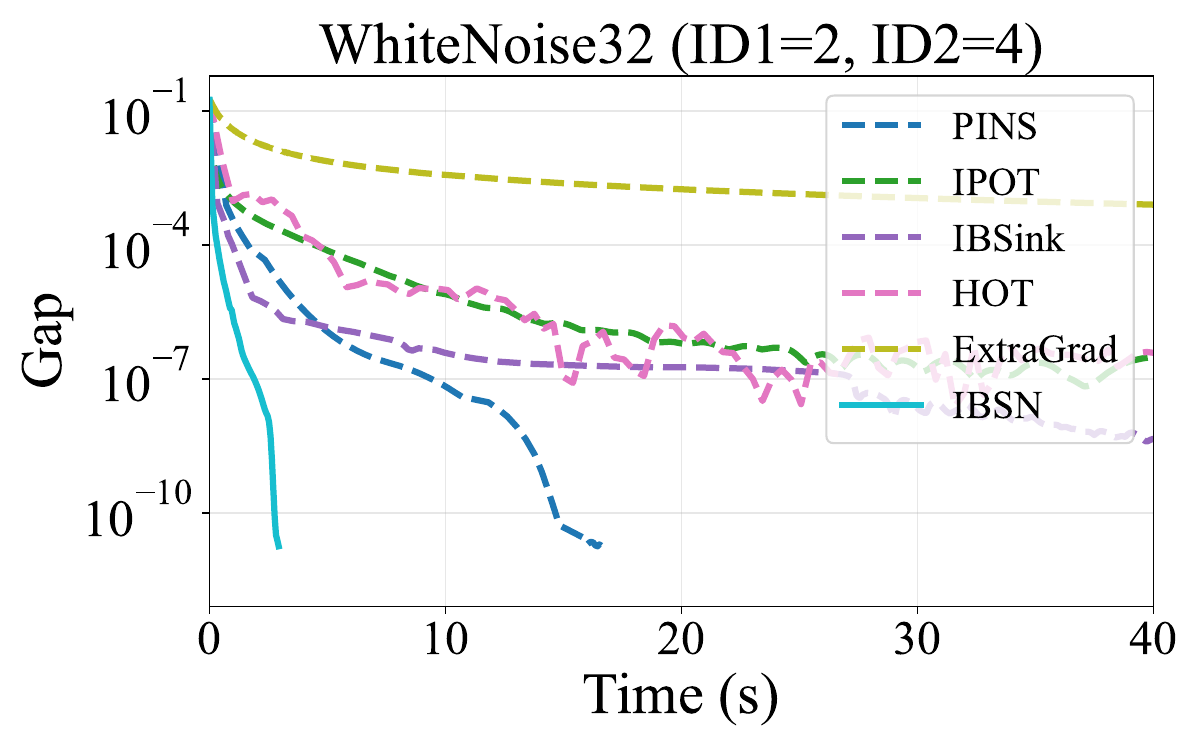}
\includegraphics[width=0.32\textwidth]{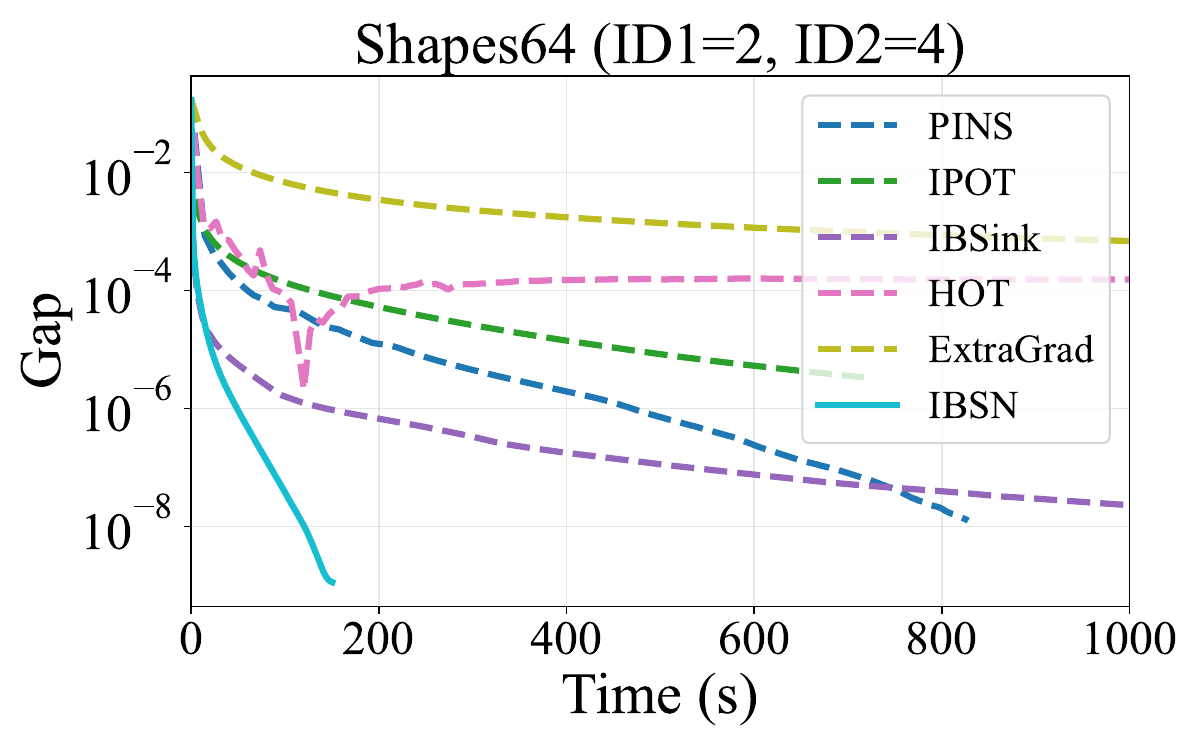}\hfill
\includegraphics[width=0.32\textwidth]{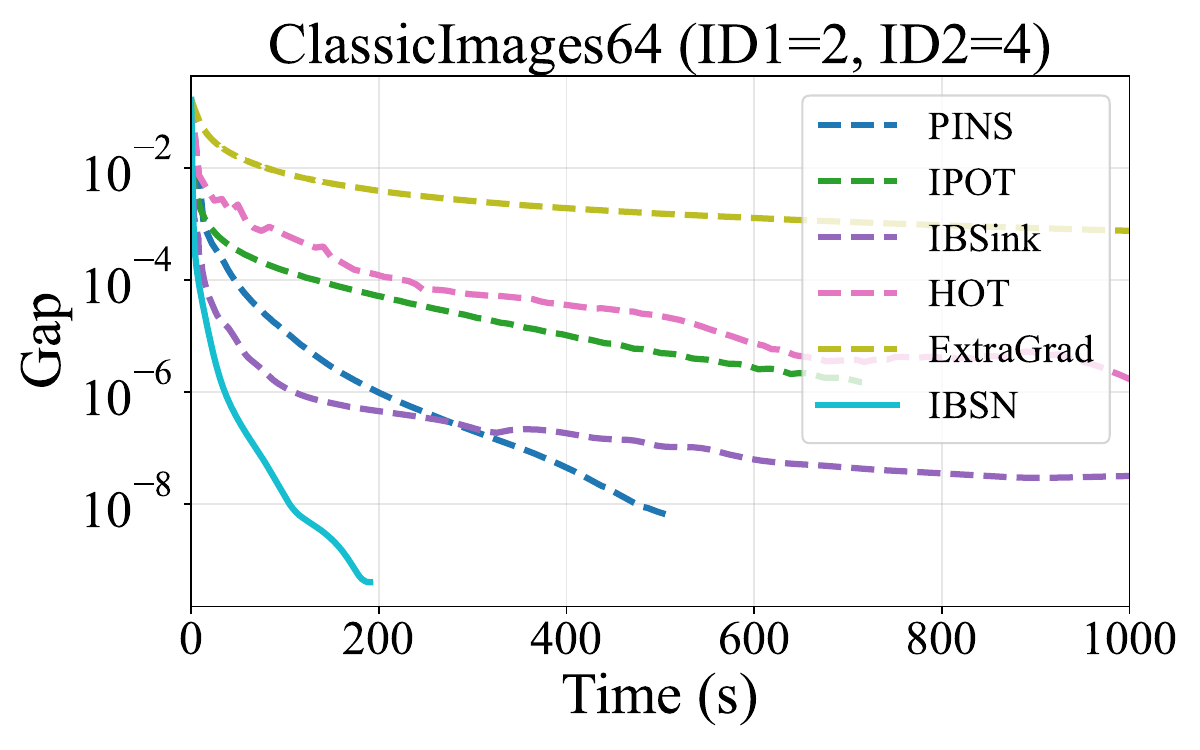}\hfill
\includegraphics[width=0.32\textwidth]{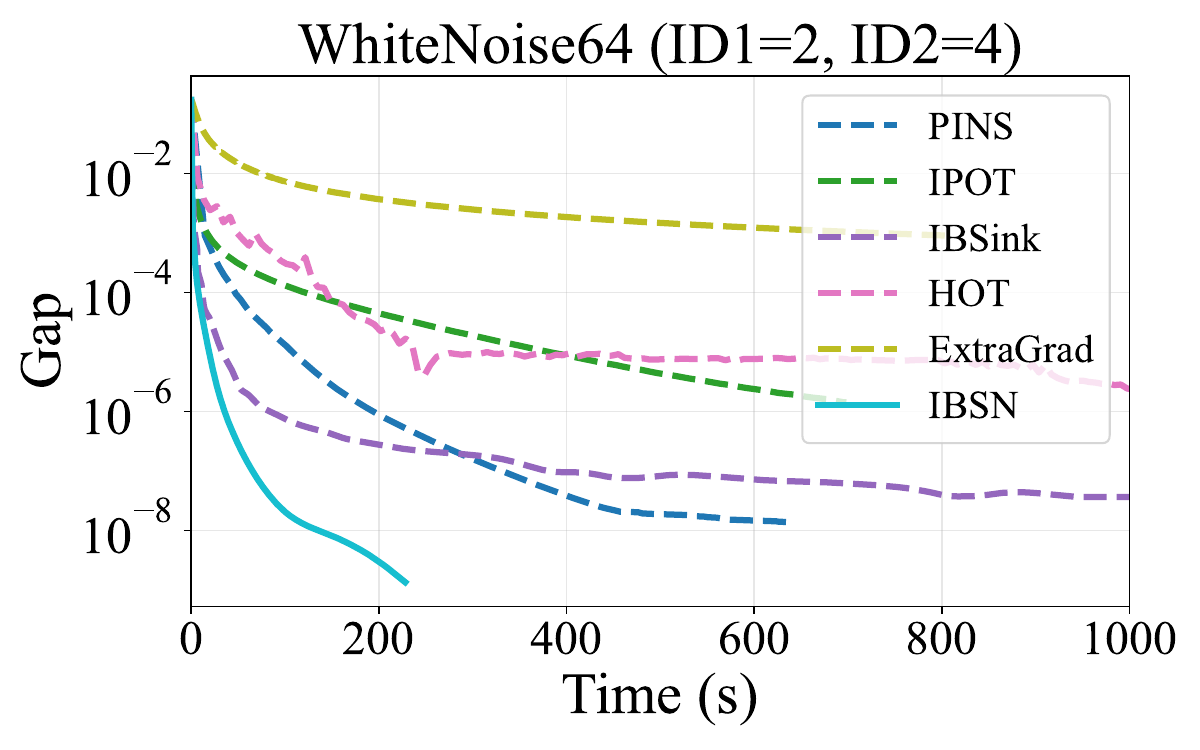}
\caption{Performance of different algorithms on DOTmark dataset. Top: $32 \times 32$. Bottom: $64 \times 64$.}
\label{fig: Performance of different algorithms on DOTmark dataset}
\end{figure*}

\begin{figure*}[t]
\includegraphics[width=0.32\textwidth]{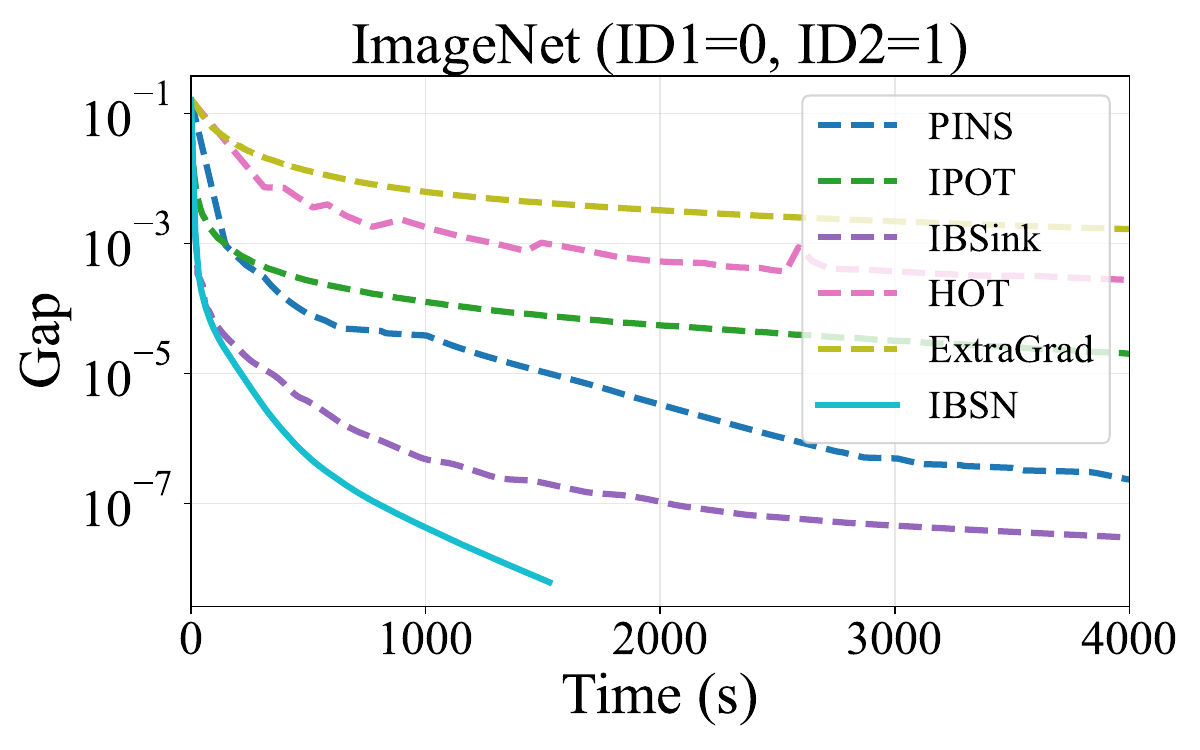}\hfill
\includegraphics[width=0.32\textwidth]{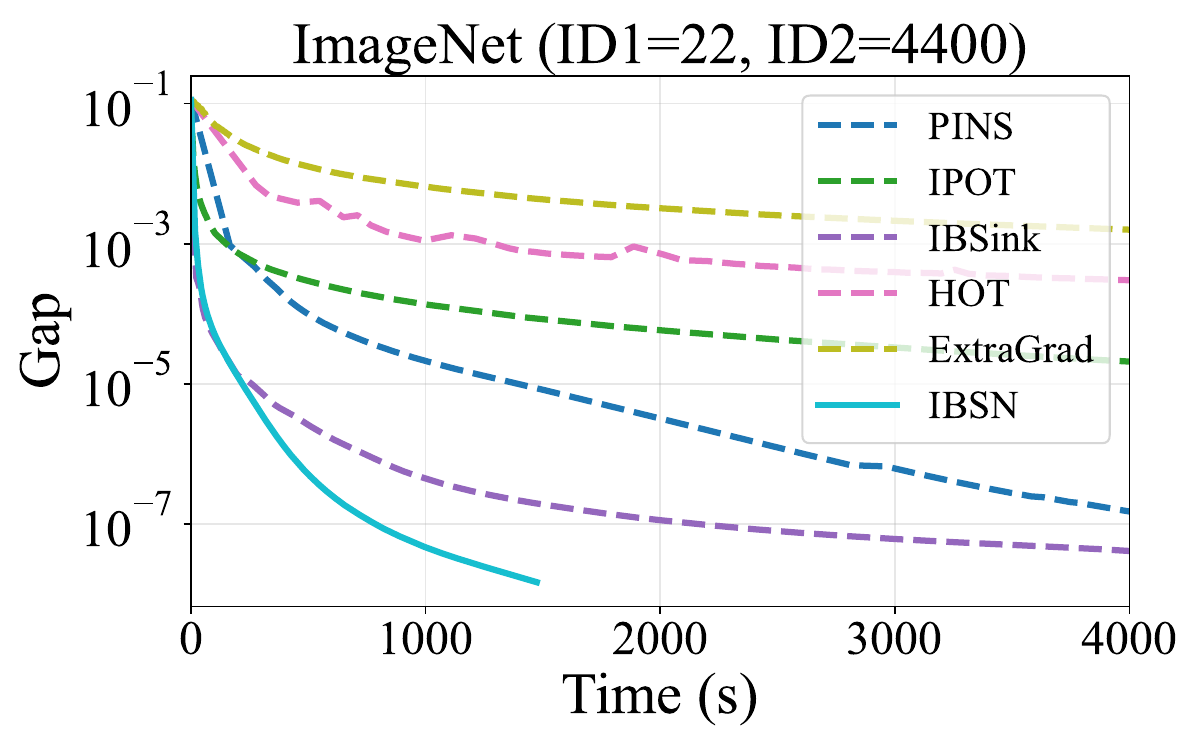}\hfill
\includegraphics[width=0.32\textwidth]{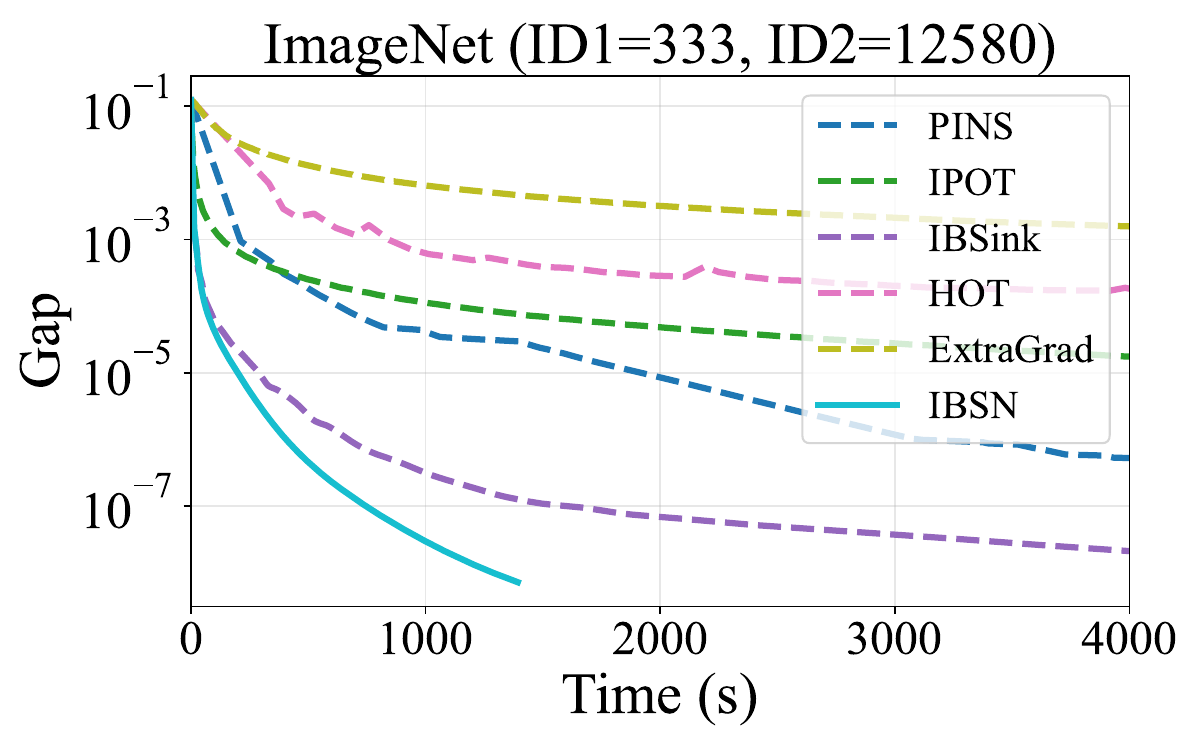}
\caption{Performance of different algorithms on the ImageNet dataset.}
\label{fig: Performance of different algorithms on the ImageNet dataset.}
\end{figure*}

We further evaluate the proposed IBSN algorithm on real-image OT tasks to examine its effectiveness in practical scenarios. Three benchmark datasets are considered:
\begin{itemize}
    \item \textbf{(Fashion-)MNIST}: We compute the OT map between pairs of images from the MNIST~\cite{lecun2002gradient} and the Fashion-MNIST~\cite{xiao2017fashion} datasets. The vectors $\vect{a}$ and $\vect{b}$ are obtained by flattening and normalizing the pixel intensity values, while the cost matrix $C$ is generated via the squared $\ell_2$-distance between pixel indices. The problem size is set to $m = n = 784$.
    \item \textbf{DOTmark}: The DOTmark dataset~\cite{schrieber2016dotmark} provides a comprehensive benchmark suite for evaluating OT algorithms. In our experiments, we select two images each from three categories (\texttt{Shapes}, \texttt{ClassicImages}, \texttt{WhiteNoise}) in the dataset, with image size $32 \times 32$ ($m = n = 1024$) and $64 \times 64$ ($m=n=4096$). The cost matrix $C$ and marginal vectors $\vect{a}, \vect{b}$ are constructed in the same manner as in the (Fashion-)MNIST experiments.  
    \item \textbf{ImageNet}: We further evaluate the algorithms on larger-scale image pairs from the ImageNet dataset~\cite{deng2009imagenet}. Each image is resized to $100 \times 100$, resulting in an OT problem of size $m=n=10000$. The marginal vectors $\vect{a}$ and $\vect{b}$ and the cost matrix $C$ are constructed in the same manner as in the (Fashion-)MNIST experiments.
\end{itemize}

We set $\eta = 10^{-3}$ for IBSN, PINS and IBSink, $\eta=10^{-1}$ for IPOT. These choices of $\eta$ are determined in the same way as in Section~\ref{subsec: Synthetic data}. 

Figure~\ref{fig: Performance of different algorithms on real data} shows the results on the (Fashion-)MNIST dataset. As observed, first-order methods (HOT and ExtraGrad) remain comparatively slow. Notably, the convergence trajectory of IBSink exhibits three distinct stages, a behavior attributable to the inexactness of the framework. Initially, the algorithm converges rapidly due to the loose inexactness tolerance $\mu_k$. As the tolerance tightens, IBSink enters a second stage in which it requires significantly more iterations to satisfy the inner stopping criterion, leading to a slowdown. Finally, as the transport plan $X$ approaches the optimal $X^\star$, convergence accelerates again. In contrast, our proposed IBSN algorithm overcomes this mid-stage bottleneck by leveraging second-order information. This allows it to maintain rapid convergence even as the tolerance $\mu_k$ becomes more stringent. Similar patterns are observed in the DOTmark dataset (Figure~\ref{fig: Performance of different algorithms on DOTmark dataset}) and the ImageNet dataset results (Figure~\ref{fig: Performance of different algorithms on the ImageNet dataset.}), where the performance gap between IBSN and competing algorithms is even more pronounced. 
More experiment results are given in Appendix~\ref{app: Additional Experiment Details}.

\section{Conclusion}

In this paper, we present the Inexact Bregman Sparse Newton (IBSN) method for efficiently solving the exact optimal transport problem. Building on the verifiable inexact Bregman proximal framework, we develop a sparse Newton solver tailored to the semi-dual formulation, with a structure-preserving Hessian sparsification strategy that reduces the computational cost of second-order updates. By adaptively coupling the sparsification threshold to the optimality residual, our method improves the Hessian approximation accuracy near the solution while retaining fast local convergence. The verifiable inexact stopping criterion further avoids unnecessarily accurate subproblem solves. Our theoretical analysis establishes convergence guarantees for both the inner Newton iterations and the outer framework. Extensive experiments on synthetic and real-world datasets demonstrate the efficiency of IBSN relative to representative iterative OT solvers. Limitations of the proposed method are discussed in Appendix~\ref{app: limitations}.

\section*{Acknowledgements}
We thank Prof. Lei Yang from Sun Yat-sen University for valuable guidance and insightful discussions. We also thank the authors of PINS for sharing their code.

This work was partially supported by the Guangdong Major Project of Basic Research [grant number 2026B0303000015], the National Natural Science Foundation of China [grant numbers 72495131, 82441027], Guangdong Provincial Key Laboratory of Mathematical Foundations for Artificial Intelligence [grant number 2023B1212010001], and Shenzhen Stability Science Program.

\bibliography{reference}
\bibliographystyle{plain}

\newpage
\appendix

\section{Limitations}
\label{app: limitations}

The proposed method has several limitations. First, the efficiency of the Newton phase may be affected by the conditioning of the `sparsified' Hessian. In particular, when the induced transport probabilities contain very small entries, the Newton system can become less well-conditioned, requiring more conjugate gradient iterations. Second, although the proposed sparsification strategy is efficient in practice, a complete worst-case complexity analysis of the overall inexact Bregman Newton framework remains open. Finally, our convergence analysis does not account for the inexact Newton directions produced by CG with finite stopping tolerances.

\section{Converting a matrix to a transportation plan}
\label{app: Converting a matrix to a transportation plan}

Given a general non-negative matrix $X \in \R^{m \times n}$, the authors in~\cite{altschuler2017near} put forward a simple algorithm that returns a transportation plan $\hat{X} \in \R^{m \times n}$ satisfying $\hat{X}\1_n = \vect{a}$ and $\hat{X}^\top \1_m = \vect{b}$, where $\1_m^\top \vect{a} = \1_n^\top \vect{b} = 1$. The algorithm is shown below as Algorithm~\ref{alg: rounding}.

\begin{algorithm}
\caption{Converting a matrix $X$ to a feasible transportation plan~\cite{altschuler2017near}}
\label{alg: rounding}
\begin{algorithmic}[1]
\STATE {\bfseries Input:} $X$, $\vect{a}$, $\vect{b}$.
\FOR{$i=1,...,m$}
\STATE $z_i \gets \min\left\{ \frac{a_i}{\sum_{j=1}^n X_{ij}}, 1\right\}$.
\ENDFOR
\STATE Compute $F = \diag{\vect{z}} X$.  \hfill \% row sum is no greater than $a_i$
\FOR{$j=1,...,n$}
\STATE $y_j \gets \min\left\{ \frac{b_j}{\sum_{i=1}^m F_{ij}}, 1\right\}$.
\ENDFOR
\STATE Update $F \gets F \diag{\vect{y}}$. \hfill \% column sum is no greater than $b_j$
\STATE Compute $\vect{e}_r \gets \vect{a} - F \1_n$ and $\vect{e}_c \gets \vect{b} - F^\top \1_m$.
\STATE Compute $\hat{X} \gets F + \vect{e}_r\vect{e}_c^\top / \|\vect{e}_r\|_1$.
\STATE {\bfseries Output:} $\hat{X}$.
\end{algorithmic}
\end{algorithm}




\section{Proofs of Theoretical Results}
\label{app: Proofs of Theoretical Results}

\subsection{Proof of Lemma~\ref{lem: semi-dual solution}}

\begin{proof}

    Let $(\vect{\gamma}', \vect{\zeta}')$ be the solution to the two-dual problem~\eqref{opt: two-dual formulation of subproblem}. Note that the objective function in the semi-dual problem~\eqref{opt: semi-dual problem} can be rewritten as $\lyp^k(\vect{\gamma}) = \min_{\vect{\zeta} \in \R^m} Z^k(\vect{\gamma}, \vect{\zeta}) + C$, where $C$ is a constant. Therefore, for any $\vect{\gamma} \in \R^n$,
    \begin{align*}
    \lyp^k(\vect{\gamma}) = \min_{\vect{\zeta} \in \R^m} Z^k(\vect{\gamma}, \vect{\zeta}) + C \geq \min_{\vect{\gamma} \in \R^n, \vect{\zeta} \in \R^m} Z^k(\vect{\gamma}, \vect{\zeta}) + C = Z^k(\vect{\gamma}', \vect{\zeta}') + C.
    \end{align*}
    On the other hand, 
    \begin{align*}
    \lyp^k(\vect{\gamma}') = \min_{\vect{\zeta} \in \R^m} Z^k(\vect{\gamma}', \vect{\zeta}) + C \leq Z^k(\vect{\gamma}', \vect{\zeta}') + C.
    \end{align*}
    Combining these two inequalities, we know
    \begin{align*}
    \lyp^k(\vect{\gamma}') \leq Z^k(\vect{\gamma}', \vect{\zeta}') + C \leq \lyp^k(\vect{\gamma}), \quad \forall \vect{\gamma} \in \R^n.
    \end{align*}
    Therefore, $\vect{\gamma}'$ is a solution to the problem~\eqref{opt: semi-dual problem}.
\end{proof}

\subsection{Proof of Theorem~\ref{thm: positive semidefiniteness}}

\begin{proof}
For simplicity, let $W = P_{\rho}^\top \diag{\vect{a}} P_{\rho}$. The matrix $H_\rho$ is obviously symmetric. For arbitrary $\vect{z} \in \R^n$, we know
\begin{equation}
\label{equ: positive semidefinite}
\begin{aligned}
    \vect{z}^\top H_{\rho} \vect{z} = \vect{z}^\top \left(\diag{W \1_n} - W \right) \vect{z}/\eta = {1\over 2\eta}\sum_j\sum_k W_{jk} (z_j - z_k)^2 \geq 0,
\end{aligned}
\end{equation}
Hence, $H_\rho$ is positive semidefinite. Furthermore, by construction, there exists an $i^\star$ such that $(P_{\rho})_{i^\star j} > 0$ for $\forall j$, then $W_{jk} > 0$ for $\forall j,k$ and the above inequality~\eqref{equ: positive semidefinite} holds with equality for $z_j = z_k$ for $\forall j, k$.
\end{proof}

\subsection{Proof of Theorem~\ref{thm: min and max eigenvalues}}

\begin{lemma}
\label{lem: eq1}
For any $\vect{x} \in \1_n^{\perp}$, we have
\begin{equation}
    \sum_{j}\sum_{k} (x_j - x_k)^2 = 2n\|\vect{x}\|^2.
\end{equation}
\end{lemma}

\begin{proof}Expanding the square within the summation, we have
    \begin{align*}
        \sum_j\sum_k (x_j - x_k)^2 &= \sum_j\sum_k(x_j^2 + x_k^2) - 2\sum_j\sum_kx_j x_k \\
        & = 2n \sum_{j} x_j^2  -  2\Big(\sum_{j}x_j\Big)^2   = 2n \|\vect{x}\|^2.
    \end{align*}
    This completes the proof.
\end{proof}

Now we are ready to prove Theorem~\ref{thm: min and max eigenvalues}. 

\begin{proof}
    Based on the construction of $H_{\rho}$, we know
    \begin{align*}
    W_{jk} = \sum_{i=1}^m a_i [P_{\rho}]_{ij}[P_{\rho}]_{ik} \geq a_{i^\star} [P_{\rho}]_{i^\star j}[P_{\rho}]_{i^\star k} \geq a_{i^\star} p^2.
    \end{align*}
    Therefore, for any $\vect{z} \in \1_n^{\perp}$ with $\|\vect{z}\|^2 = 1$, we have
    \begin{equation}
        \vect{z}^\top H_\rho \vect{z} = {1\over2\eta}\sum_{j}\sum_{k} W_{jk} (z_j - z_k)^2  \geq {a_{i^\star } p^2\over2\eta} \sum_{j}\sum_{k} (z_j - z_k)^2 =  {n a_{i^\star } p^2 \over \eta},
    \end{equation}
    where the first equation follows from Theorem~\ref{thm: positive semidefiniteness} and the final equation follows from Lemma~\ref{lem: eq1}.
    
    We consider the off-diagonal elements of $W$. Since $[P_{\rho}]_{ij}+\sum_{k\neq j}[P_{\rho}]_{ik}= 1$, we have
    \begin{align*}
    \sum_{k\neq j}W_{jk} = \sum_{k\neq j}\sum_{i=1}^m a_i [P_{\rho}]_{ij}[P_{\rho}]_{ik}=\sum_{i=1}^m a_i [P_{\rho}]_{ij}(1-[P_{\rho}]_{ij})\leq  \sum_{i=1}^m {a_i\over 4}={1\over 4}.
    \end{align*}

    Note that the eigenvector $\vect{v}$ corresponding to $\lambda\left(H_{\rho}\right) \neq 0$ satisfies $\1_n^\top \vect{v} = 0$, so $\vect{v} \in \1_n^{\perp}$. For any $\vect{z} \in \R^n$ with $\|\vect{z}\|^2 = 1$, 
    \begin{equation}
    \label{equ: upper bound for max eigenvalue}
    \begin{aligned}
        \vect{z}^\top H_\rho \vect{z} = {1\over2\eta}\sum_{j}\sum_{k} W_{jk} (z_j - z_k)^2  \leq {1\over\eta} \sum_{j}\sum_{k\neq j} W_{jk} (z_j^2 + z_k^2)  
        = {2\over\eta} \sum_{j}z_j^2 \sum_{k \neq j} W_{jk} \leq {1 \over 2\eta}.
    \end{aligned}
    \end{equation}
    This completes the proof.
\end{proof}

\subsection{Proof of Theorem~\ref{thm: error bound for hessian matrix}}

We first present the approximation error of $E_P:=P-P_{\rho}$, where $P_\rho$ is constructed via Algorithm~\ref{alg: Sparsifying the Hessian matrix}.

\begin{lemma}
\label{lem: bound for E_P}
Let $P$ and $P_{\rho}$ constructed by Algorithm~\ref{alg: Sparsifying the Hessian matrix} for any $\rho \in [0,1/n)$, then
    \begin{align*}
\|E_P\|_{\infty} \leq 2n\rho, \quad \|E_P\|_1 \leq 2mn\rho, \quad \text{and}\quad \|E_P\|\leq 2\sqrt{m}n\rho.
    \end{align*}
\end{lemma}

\begin{proof}
To begin with, denote $K_i = \{j: P_{ij} < \rho\}$ and $J_i = \{j: P_{ij} \geq \rho\}$. Write $P_{i\cdot}=Q_{i\cdot}+S_{i\cdot}$ with

    \begin{align*}
Q_{ij} = 
\begin{cases}
    P_{ij},  & j \in J_i\\
    0,  & j \in K_i\\
\end{cases},
\quad \quad
S_{ij} = 
\begin{cases}
    0,  & j \in J_i\\
    P_{ij},  & j \in K_i\\
\end{cases},
    \end{align*}
then $\sum_{j\in J_i}Q_{ij}=1-\sum_{j \in K_i}P_{ij} := 1-r_i$ and $\sum_{j\in K_i}S_{ij}=r_i \leq n\rho$.
Hence, for any $i \neq i^\star$,
\begin{align*}
    \|P_{i\cdot} - [P_{\rho}]_{i\cdot}\|_1
     & = \sum_{j\in J_i}\left|Q_{ij}-\frac{Q_{ij}}{1-r_i}\right| +\sum_{j\in K_i} \left|S_{ij}-0\right| \\
     & = \left(\frac{1}{1-r_i}-1\right)\sum_{j\in J_i}Q_{ij} + \sum_{j\in K_i}S_{ij} \\
     & = 2r_i \leq 2n \rho.
\end{align*}
Therefore, we obtain $\|E_P\|_{\infty} \leq 2n\rho$ and $\|E_P\|_1 \leq 2mn\rho$. In addition, 
$\|E_P\|\leq \sqrt{\|E_P\|_1 \|E_P\|_\infty} \leq 2 \sqrt{m}n\rho$.
\end{proof}

Now we show the approximation error between $H$ and sparsified $H_{\rho}$. 

\begin{proof}
To begin with, observe
\begin{align*}
H - H_{\rho} = \frac{1}{\eta}\left(
\underbrace{\diag{P^\top\vect{a} - P^\top_{\rho}\vect{a}}}_{=:D_1}
-
\left(\underbrace{P^\top \diag{\vect{a}}E_P}_{=:D_2} + \underbrace{(E_P)^\top \diag{\vect{a}} P_{\rho}}_{D_3} \right)
\right).
\end{align*}
For $D_1$, note that
\begin{align*}
\|D_1\| = \|\diag{(E_P)^\top \vect{a}}\| = \|E_P^\top \vect{a}\|_\infty \leq \|\vect{a}\|_\infty \|E_P\|_{1} 
\;\stackrel{\text{Lemma}~\ref{lem: bound for E_P}}{\leq} 2\rho mn\|\vect{a}\|_{\infty}.
\end{align*}
For $D_2$ and $D_3$, we have
\begin{align*}
\|D_2\| \leq \|P\|\ \|\vect{a}\|_\infty\ \|E_P\|,\qquad
\|D_3\| \leq \|E_P\|\ \|\vect{a}\|_\infty\ \|P_\rho\|.
\end{align*}
In addition, since $P_{\rho}$ satisfies $\|P_{\rho}\|_\infty=1$ and $\|P_{\rho}\|_1\le m$ for any $\rho \in [0,1/n)$, we have $\|P_{\rho}\| \leq \sqrt{m}$. Hence
\begin{align*}
\|D_2\| + \|D_3\| \leq 2\sqrt{m}\ \|\vect{a}\|_\infty\ \|E_P\| 
\;\stackrel{\text{Lemma}~\ref{lem: bound for E_P}}{\leq} 4\rho mn\,\|\vect{a}\|_\infty.
\end{align*}
Combining the two parts,
\begin{align*}
\|H - H_{\rho}\| \leq \frac{1}{\eta}\big(\|D_1\|+\|D_2\| + \|D_3\|\big)
\leq \frac{6mn\|\vect{a}\|_{\infty}}{\eta} \rho.
\end{align*}
\end{proof}

\subsection{Proof of Lemma~\ref{lem: sequence restrict in the subspace}}

\begin{proof}
In the Newton stage, recall that the gradient satisfies $\1_n^\top \vect{g}^{k,v} = 0$ and the regularized Newton system is given by: $\big(H^{k,v}_\rho+\|\vect{g}^{k,v}\|I\big)\Delta \vect{\gamma}^{k,v} = -\vect{g}^{k,v}$. Multiplying both sides of the Newton system by $\mathbf{1}_n^\top$ from the left yields
\begin{align*}
\1_n^\top H^{k,v}_\rho \Delta \vect{\gamma}^{k,v} + \|\vect{g}^{k,v}\| \1_n^\top \Delta \vect{\gamma}^{k,v} = 0.
\end{align*}
From Theorem~\ref{thm: positive semidefiniteness}, the sparsified Hessian satisfies $H^{k,v}_\rho \1_n = 0$. Due to the symmetry of $H^{k,v}_\rho$, we also have $\1_n^\top H^{k,v}_\rho  = 0$. Consequently, the first term vanishes, leaving $\1_n^\top \Delta \vect{\gamma}^{k,v} = 0$ if $\|\vect{g}^{k,v}\| \neq0$. This ensures that the update maintains $\1_n^\top\vect{\gamma}^{k,v+1} = 0$.
\end{proof}

\subsection{Proof of Theorem~\ref{thm:inner-convergence}}

In the following analysis, we fix the outer iteration $k$. We first present the following lemmas. 

\begin{lemma}
\label{lem: bounded level set + bounded below}
The set $S^{k}(\vect{\gamma}) := \{\vect{\gamma}\mid \lyp^k(\vect{\gamma}) \leq \lyp^k(\vect{\gamma}^{k,1}), \ \vect{\gamma} \in \1_n^\perp\}$ is bounded, closed, and convex, $\lyp^k(\vect{\gamma})$ is bounded below.
\end{lemma}

    

\begin{proof}
    Define $\kappa_k:=\min_{i,j}\{X^k_{ij}\exp(-C_{ij}/\eta)\}>0$. For $M(\vect{\gamma}):=\max_j\gamma_j$, the log-sum-exp term yields
    \begin{align*}
    \lyp^k(\vect{\gamma})
    &\geq -\vect{b}^{\top}\vect{\gamma}+M(\vect{\gamma})+\eta\log\kappa_k
    \geq n\min_j b_j\,M(\vect{\gamma})+\eta\log\kappa_k,
    \end{align*}
    where the last inequality uses $\1_n^{\top}\vect{\gamma}=0$. On $\1_n^\perp$, $\|\vect{\gamma}\|\to\infty$ implies $M(\vect{\gamma})\to\infty$. Thus, $\lyp^k$ is coercive on $\1_n^\perp$. Since $\lyp^k$ is continuous and convex, $S^k$ is closed, bounded and convex, and it is nonempty because it contains $\vect{\gamma}^{k,1}$. Continuity and coercivity also imply that $\lyp^k$ attains its minimum on this subspace. Finally, $\lyp^k(\vect{\gamma}+c\1_n)=\lyp^k(\vect{\gamma})$ for all $c\in\R$, so this minimum is also the minimum over $\R^n$.
\end{proof}

\begin{lemma}
\label{lem: uniform lower bound for min eigenvalue of hessian}
    Let $H_{\rho}^k(\vect{\gamma})$ be the matrix output by Algorithm~\ref{alg: Sparsifying the Hessian matrix}, then for $\rho \in [0, 1/n)$, there exist a constant $\underline{C}_k > 0$ such that 
    $\lambda_{\min}^+(H_{\rho}^k(\vect{\gamma})) \geq \underline{C}_k$ for all $\vect{\gamma} \in S^{k}(\vect{\gamma}).$
\end{lemma}

\begin{proof}
    By the boundedness of $S^{k}(\vect{\gamma})$, each element of $P^k(\vect{\gamma})$ as defined in~\eqref{equ: p update} is bounded below from zero over $S^{k}(\vect{\gamma})$. Recall the lower bound of $\lambda_{\min}^+(H_{\rho}^k(\vect{\gamma}))$ established in Theorem~\ref{thm: min and max eigenvalues}, we conclude that there exists a constant $\underline{C}_k > 0$ such that $\lambda_{\min}^+(H_{\rho}^k(\vect{\gamma})) \geq \underline{C}_k$ holds for any $\rho \in [0, 1/n)$ and all $\vect{\gamma} \in S^{k}(\vect{\gamma})$.
\end{proof}

As shown in Lemma~\ref{lem: sequence restrict in the subspace}, the convergence analysis can be safely restricted to $\1_n^{\perp}$ in the subsequent analysis. We then show that $\Delta \vect{\gamma}^{k,v}$ is a descent direction and provide the minimum stepsize. 

\begin{lemma}
\label{lem:armijo-lower-bound}
In the outer $k$ iteration, let $\{\vect{\gamma}^{k,v}\}_v$ be the sequence generated by Algorithm~\ref{alg: IBSN} satisfying $\1_n^\top\vect{\gamma}^{k,0} = 0$, $\Delta \vect{\gamma}^{k,v}$ be a direction for $\lyp^k$ and we choose $t^{k,v}$ as stepsize by Armijo backtracking, then 
\begin{enumerate}
    \item[(i)] $\Delta \vect{\gamma}^{k,v}$ is a descent direction for $\lyp^k$; 
    \item[(ii)] the backtracking terminates finitely with a step size $t^{k,v}$ bounded below by
\begin{equation}
\label{equ:tmin}
t^{k,v} \geq t^k_{\min} := \min\left\{1,\ \beta \frac{4\eta (1-\sigma)\underline{C}_k^2}{1/(2\eta)+2}\right\} > 0,
\end{equation}
where $\beta\in(0,1)$ is the backtracking contraction factor and $\sigma\in(0,1)$ is the Armijo parameter.
\end{enumerate}
\end{lemma}

\begin{proof}
(i) To begin with, it is easy to see that
\begin{align*}
(\vect{g}^{k,v})^\top \Delta \vect{\gamma}^{k,v}
= -(\vect{g}^{k,v})^\top\big(H^{k,v}_\rho+\|\vect{g}^{k,v}\|I\big)^{-1} (\vect{g}^{k,v})<0,
\end{align*}
where $\vect{g}^{k,v}$ is the gradient of $\lyp^k$ on $\vect{\gamma}^{k,v}$ and the inequality due to Theorem~\ref{thm: positive semidefiniteness}. Therefore, $\Delta \vect{\gamma}^{k,v}$ is a descent direction. 

(ii) Since $\lambda_{\max} (H_{\rho}) \leq 1/(2\eta)$ for any$\rho \in [0,1/n)$ from Theorem~\ref{thm: min and max eigenvalues}, $\vect{g}^{k,v}$ is $1/(2\eta)$-Lipschitz continuous, and the standard quadratic upper bound yields
\begin{align*}
\lyp^k(\vect{\gamma}^{k,v}+t \Delta \vect{\gamma}^{k,v}) \leq \lyp^k(\vect{\gamma}^{k,v})+t (\vect{g}^{k,v})^\top \Delta \vect{\gamma}^{k,v} + \frac{1}{4\eta}t^2\|\Delta \vect{\gamma}^{k,v}\|^2.
\end{align*}
Hence any $t$ with
\begin{align*}
0<t \leq \bar{t}^{k,v}
:= \frac{4\eta(1-\sigma)\,(-(\vect{g}^{k,v})^\top \Delta \vect{\gamma}^{k,v})}{\|\Delta \vect{\gamma}^{k,v}\|^2}
\end{align*}
satisfies the Armijo inequality $\lyp^k(\vect{\gamma}^{k,v}+t \Delta \vect{\gamma}^{k,v}) \leq \lyp^k(\vect{\gamma}^{k,v})+\sigma t(\vect{g}^{k,v})^\top \Delta \vect{\gamma}^{k,v}$.

We now provide a lower bound for $\bar{t}^{k,v}$. First, note that
\begin{align*}
    & -(\vect{g}^{k,v})^\top \Delta \vect{\gamma}^{k,v} = (\vect{g}^{k,v})^\top\big(H^{k,v}_\rho+\|\vect{g}^{k,v}\|I\big)^{-1} (\vect{g}^{k,v})\geq \frac{1}{\lambda_{\max}(H_{\rho}^{k,v})+\|\vect{g}^{k,v}\|}\|\vect{g}^{k,v}\|^2, \\
    & \|\Delta \vect{\gamma}^{k,v}\|= \left(H^{k,v}_\rho+\|\vect{g}^{k,v}\|I\right)^{-1} \vect{g}^{k,v} \leq \frac{1}{\lambda_{\min}^+(H^{k,v}_\rho) + \|\vect{g}^{k,v}\|}\,\|\vect{g}^{k,v}\| \leq \frac{1}{\lambda_{\min}^+(H^{k,v}_\rho)}\,\|\vect{g}^{k,v}\|.
\end{align*}
Combining these,
\[
\bar{t}^{k,v} \geq
\frac{4\eta(1-\sigma)[\lambda_{\min}^+(H^{k,v}_\rho)]^2}{\lambda_{\max}(H^{k,v}_\rho)+\|\vect{g}^{k,v}\|} \geq 
\frac{4\eta (1- \sigma)[\lambda_{\min}^+(H^{k,v}_\rho)]^2}{\lambda_{\max}(H^{k,v}_\rho)+2},
\]
where the final inequality holds since $\|\vect{g}^{k,v}\| \leq \|(P^k(\vect{\gamma}^{k,v}))^\top \vect{a}\| + \|\vect{b}\| \leq 2$. Backtracking with ratio $\beta$ accepts some $t^{k,v} \geq \beta\bar{t}^{k,v}$
and using the bound $\lambda_{\max}(H^{k,v}_\rho)\leq 1 / (2\eta)$ and $\lambda_{\min}^+(H^{k,v}_\rho)$ in Theorem~\ref{thm: min and max eigenvalues}. From the Armijo criterion, we know $\{\lyp^k(\vect{\gamma}^{k,v})\}_v$ is monotonically decreasing and $\vect{\gamma}^{k,v} \in S^k(\vect{\gamma})$ for all $v$. By Lemma~\ref{lem: uniform lower bound for min eigenvalue of hessian}, we know $\lambda_{\min}^+(H_{\rho}^{k,v}) \geq \underline{C}_k > 0$ for all $v$ and $\rho \in [0,1/n)$ over $S^{k}(\vect{\gamma})$. This allows us to define a uniform lower bound for the step size:
\begin{align*}
t_{\min}^{k,v} \geq t_{\min}^k := \min\left\{1,\ \beta \frac{4\eta (1-\sigma)\underline{C}_k^2}{1/(2\eta)+2}\right\},
\end{align*}
yields \eqref{equ:tmin}.
\end{proof}

Next, we provide the proof of Theorem~\ref{thm:inner-convergence}. 

\begin{proof}
(i) From Lemma~\ref{lem:armijo-lower-bound}, we know
\begin{align*}
-(\vect{g}^{k,v})^\top \Delta \vect{\gamma}^{k,v} \geq \frac{1}{\lambda_{\max}(H^{k,v})+\|\vect{g}^{k,v}\|}\|\vect{g}^{k,v}\|^2 \geq \frac{1}{1/(2\eta) + 2}\|\vect{g}^{k,v}\|^2,
\end{align*}
since $\|\vect{g}^{k,v}\| \leq 2$. It further yields
\begin{equation}
\label{equ: decrease property}
\lyp^k(\vect{\gamma}^{k,v+1}) - \lyp^k(\vect{\gamma}^{k,v})\leq - \sigma t^k_{\min} \cdot \frac{1}{1/(2\eta) + 2}\|\vect{g}^{k,v}\|^2,
\end{equation}
where $t_{\min}^k$ is defined in Lemma~\ref{lem:armijo-lower-bound}. 

In addition, since $\lyp^k$ is bounded below by Lemma~\ref{lem: bounded level set + bounded below}, the telescoping sum of~\eqref{equ: decrease property} implies $\sum_v \|\vect{g}^{k,v}\|^2 < \infty$. Therefore, we conclude that $\|\vect{g}^{k,v}\| \to 0$ as $v \to \infty$. \\

(ii) By the definition of $\vect{\gamma}^{k\star }$, we know $g^k(\vect{\gamma}^{k\star }) \in \text{span}\{\1_n\}$, where $g^k(\vect{\gamma})$ is defined in~\eqref{equ: grad and hess}. On the other hand, it is easy to check that for all $\vect{\gamma} \in \R^n$, $\1_n^\top g^k(\vect{\gamma}) = 0$. Combining the two relations implies $g^k(\vect{\gamma}^{k\star }) = \0_n$. 

Over $S^{k}(\vect{\gamma})$, since $\lyp^k$ is $\underline{C}_k$-strongly convex, the solution of $\lyp^k$ is unique and
\begin{align*}
\|\vect{g}^{k,v}\| = \|\vect{g}^{k,v} - g^k(\vect{\gamma}^{k\star })\| \geq \underline{C}_k \|\vect{\gamma}^{k,v} - \vect{\gamma}^{k\star }\|
\end{align*}
by~\cite[Theorem 2.1.10]{nesterov2018lectures}. Since $\|\vect{g}^{k,v}\|\to 0$, it follows that $\vect{\gamma}^{k, v}\to \vect{\gamma}^{k\star }$ as $v \to \infty$, where $\vect{\gamma}^{k\star }$ satisfies $\1_n^\top \vect{\gamma}^{k\star } = 0$.
\end{proof}

\subsection{Proof of Theorem~\ref{thm:sparse-newton-rate}}

The following lemma shows the global Lipschitz continuity of the exact Hessian $H^k(\vect{\gamma})$. 

\begin{lemma}
\label{lem:global-hessian-lipschitz}
For any outer iteration $k$, the exact Hessian $H^k(\vect{\gamma})$ of the subproblem objective $\lyp^k(\vect{\gamma})$ is globally Lipschitz continuous over $\R^n$. The global Lipschitz constant $L_H$ is given by:
\begin{equation*}
    L_H = \frac{6}{\eta^2}.
\end{equation*}
\end{lemma}

\begin{proof}
Recall that the directional derivative of a differentiable multivariate function $f(\vect{\gamma})$ along a vector $\vect{u}$ is defined as $\nabla_{\vect{u}} f(\vect{\gamma}) = \langle \nabla f(\vect{\gamma}), \vect{u} \rangle$. To establish global Lipschitz continuity, we evaluate the spectral norm of the third derivative tensor. Equivalently, we show that for any arbitrary vectors $\vect{y}, \vect{u} \in \R^n$, the third-order directional derivative of the objective is uniformly bounded by the constant $L_H$:
\begin{equation*}
    |\nabla_{\vect{u}} (\vect{y}^\top H^k(\vect{\gamma}) \vect{y})| \le L_H \|\vect{y}\|^2 \|\vect{u}\|.
\end{equation*}

For brevity, let $P = P^k(\vect{\gamma})$. We define a deterministic weighted sum operator $S_i: \R^n \to \R$, representing the inner product with the $i$-th row of the transport probability matrix $P$:
\begin{equation*}
    S_i(\vect{x}) = \sum_{j=1}^n P_{ij} x_j.
\end{equation*}
Because each row of $P$ is a valid probability distribution (i.e., $P_{ij} \ge 0$ and $\sum_j P_{ij} = 1$), the operator $S_i$ computes a strict convex combination. Its magnitude is universally bounded by the infinity norm, and consequently the Euclidean norm, of the input vector: $|S_i(\vect{x})| \le \|\vect{x}\|_\infty \le \|\vect{x}\|$.

Using this operator, the quadratic form $Q(\vect{\gamma}) := \vect{y}^\top H^k(\vect{\gamma}) \vect{y}$ can be explicitly expanded as:
\begin{equation*}
    Q(\vect{\gamma}) = \frac{1}{\eta} \sum_{i=1}^m a_i \left[ S_i(\vect{y}^2) - (S_i(\vect{y}))^2 \right],
\end{equation*}
where $\vect{y}^2$ denotes the Hadamard (element-wise) square of the vector $\vect{y}$.

To compute the directional derivative of $Q(\vect{\gamma})$ along $\vect{u}$, we first require the directional derivative of the transport probability matrix entries. From the definition of $P_{ij}$ in~\eqref{equ: p update}, we can write the logarithm of $P_{ij}$ as:
\begin{equation*}
    \ln P_{ij} = \ln X^k_{ij} + \frac{\gamma_j - C_{ij}}{\eta} - \ln Z_i,
\end{equation*}
where $Z_i = \sum_{l=1}^n X^k_{il} \exp((\gamma_l - C_{il})/\eta)$ is the row normalization factor. Taking the directional derivative of $\ln Z_i$ along $\vect{u}$ gives:
\begin{equation*}
    \nabla_{\vect{u}} \ln Z_i = \frac{1}{Z_i} \sum_{l=1}^n X^k_{il} \exp\left(\frac{\gamma_l - C_{il}}{\eta}\right) \frac{u_l}{\eta} = \frac{1}{\eta} \sum_{l=1}^n P_{il} u_l = \frac{1}{\eta} S_i(\vect{u}).
\end{equation*}
Applying the chain rule, $\nabla_{\vect{u}} \ln P_{ij} = \frac{1}{P_{ij}} \nabla_{\vect{u}} P_{ij}$, and differentiating the expression for $\ln P_{ij}$ yields:
\begin{equation*}
    \frac{1}{P_{ij}} \nabla_{\vect{u}} P_{ij} = \frac{u_j}{\eta} - \nabla_{\vect{u}} \ln Z_i = \frac{u_j}{\eta} - \frac{1}{\eta} S_i(\vect{u}).
\end{equation*}
Multiplying both sides by $P_{ij}$, we isolate the directional derivative of $P_{ij}$:
\begin{equation*}
    \nabla_{\vect{u}} P_{ij} = \frac{1}{\eta} P_{ij} (u_j - S_i(\vect{u})).
\end{equation*}

By linearity, the derivative of the operator $S_i(\vect{x})$ for a fixed vector $\vect{x}$ evaluates to:
\begin{equation*}
    \nabla_{\vect{u}} S_i(\vect{x}) = \sum_{j=1}^n (\nabla_{\vect{u}} P_{ij}) x_j = \frac{1}{\eta} \left[ S_i(\vect{x} \odot \vect{u}) - S_i(\vect{x}) S_i(\vect{u}) \right].
\end{equation*}

Let $B_i = S_i(\vect{y}^2) - (S_i(\vect{y}))^2$. Differentiating $B_i$ along $\vect{u}$ gives:
\begin{equation*}
    \nabla_{\vect{u}} B_i = \frac{1}{\eta} \Big[ S_i(\vect{y}^2 \odot \vect{u}) - S_i(\vect{y}^2)S_i(\vect{u}) - 2S_i(\vect{y}) \big( S_i(\vect{y} \odot \vect{u}) - S_i(\vect{y})S_i(\vect{u}) \big) \Big].
\end{equation*}
Applying the absolute magnitude bound $|S_i(\vect{x})| \le \|\vect{x}\|$, we have $|S_i(\vect{y}^2 \odot \vect{u})| \le \|\vect{y}^2 \odot \vect{u}\|_\infty \le \|\vect{y}\|^2 \|\vect{u}\|$, and similarly for the other terms. Using the triangle inequality, we can strictly bound the magnitude of $\nabla_{\vect{u}} B_i$:
\begin{equation*}
    |\nabla_{\vect{u}} B_i| \le \frac{1}{\eta} \left( \|\vect{y}\|^2\|\vect{u}\| + \|\vect{y}\|^2\|\vect{u}\| + 2\|\vect{y}\|\left( \|\vect{y}\|\|\vect{u}\| + \|\vect{y}\|\|\vect{u}\| \right) \right) = \frac{6}{\eta} \|\vect{y}\|^2 \|\vect{u}\|.
\end{equation*}

Finally, substituting this bound back into the derivative of the total quadratic form:
\begin{equation*}
    |\nabla_{\vect{u}} Q(\vect{\gamma})| \le \frac{1}{\eta} \sum_{i=1}^m a_i |\nabla_{\vect{u}} B_i| \le \frac{6}{\eta^2} \|\vect{y}\|^2 \|\vect{u}\| \sum_{i=1}^m a_i.
\end{equation*}
Since $\vect{a}$ is a marginal probability vector satisfying $\sum_{i=1}^m a_i = 1$, we obtain $|\nabla_{\vect{u}} Q(\vect{\gamma})| \le \frac{6}{\eta^2} \|\vect{y}\|^2 \|\vect{u}\|$. This confirms that the third-order directional derivative is uniformly bounded, thereby establishing the global Lipschitz constant as $L_H = 6/\eta^2$.
\end{proof}

The following lemma guarantees that after a sufficient number of iterations, the Armijo backtracking procedure accepts the full Newton step ($t^{k,v} = 1$).

\begin{lemma}
\label{lem: Armijo stepsize}
In the outer $k$ iteration, let $\{\vect{\gamma}^{k,v}\}_v$ be the sequence generated by Algorithm~\ref{alg: IBSN}. Then for sufficiently large $v$, the Armijo line search condition is strictly satisfied with $t^{k,v} = 1$.
\end{lemma}

\begin{proof}
From the lemma~\ref{lem:global-hessian-lipschitz}, the exact Hessian $H^k(\vect{\gamma})$ of the subproblem objective $\lyp^k(\vect{\gamma})$ is globally Lipschitz continuous with the constant $L_H$. By \cite[Lemma 1.2.4]{nesterov2018lectures}
\begin{equation}
\label{equ: taylor's theorem}
\lyp^k(\vect{\gamma}^{k,v} + \Delta\vect{\gamma}^{k,v}) \le \lyp^k(\vect{\gamma}^{k,v}) + (\vect{g}^{k,v})^\top \Delta\vect{\gamma}^{k,v} + \frac{1}{2} (\Delta\vect{\gamma}^{k,v})^\top H^{k,v} \Delta\vect{\gamma}^{k,v} + \frac{L_H}{6} \|\Delta\vect{\gamma}^{k,v}\|^3.
\end{equation}
Note that for sufficiently large $v$, we have
\begin{equation}
\label{equ: armijo inequality}
\begin{aligned}
    D_{\text{Armijo}} & :=  \lyp^k(\vect{\gamma}^{k,v} + \Delta\vect{\gamma}^{k,v}) - \lyp^k(\vect{\gamma}^{k,v}) - \sigma (\vect{g}^{k,v})^\top \Delta\vect{\gamma}^{k,v} \\
    & \leq (1 - \sigma) (\vect{g}^{k,v})^\top \Delta\vect{\gamma}^{k,v} + \frac{1}{2} (\Delta\vect{\gamma}^{k,v})^\top H^{k,v} \Delta\vect{\gamma}^{k,v} + \frac{L_H}{6} \|\Delta\vect{\gamma}^{k,v}\|^3 \\
    & = -(1 - \sigma) (\Delta\vect{\gamma}^{k,v})^\top H_{\rho}^{k,v} \Delta\vect{\gamma}^{k,v} - (1 - \sigma)\|\vect{g}^{k,v}\| \|\Delta\vect{\gamma}^{k,v}\|^2 \\
    &\quad + \frac{1}{2} (\Delta\vect{\gamma}^{k,v})^\top H^{k,v} \Delta\vect{\gamma}^{k,v} + \frac{L_H}{6} \|\Delta\vect{\gamma}^{k,v}\|^3 \\
    & = \frac{1}{2} (\Delta\vect{\gamma}^{k,v})^\top (H^{k,v} - H_{\rho}^{k,v}) \Delta\vect{\gamma}^{k,v} - \left(\frac{1}{2} - \sigma\right) (\Delta\vect{\gamma}^{k,v})^\top H_{\rho}^{k,v} \Delta\vect{\gamma}^{k,v} \\ 
    & \quad - (1 - \sigma)\|\vect{g}^{k,v}\| \|\Delta\vect{\gamma}^{k,v}\|^2 + \frac{L_H}{6} \|\Delta\vect{\gamma}^{k,v}\|^3 \\
    & \leq \left( \left(3 \|\vect{a}\|_{\infty} - (1-\sigma)\right) \|\vect{g}^{k,v}\| + \frac{L_H}{6} \|\Delta\vect{\gamma}^{k,v}\|- \left(\frac{1}{2} - \sigma\right) \lambda_{\min}^+(H_{\rho}^{k,v}) \right) \|\Delta\vect{\gamma}^{k,v}\|^2, 
\end{aligned}
\end{equation}
where $\sigma \in (0, 1/2)$ is the Armijo parameter. The first inequality comes from Equation~\eqref{equ: taylor's theorem}, the second equality comes from the fact that $(H^{k,v}_{\rho} + \|\vect{g}^{k,v}\|I) \Delta\vect{\gamma}^{k,v} = - \vect{g}^{k,v}$, and the final inequality holds since $\|H^{k,v} - H_{\rho}^{k,v}\| \leq \frac{6mn\|\vect{a}\|_{\infty}}{\eta} \rho = 6\|\vect{a}\|_{\infty} \|\vect{g}^{k,v}\|$ from Theorem~\ref{thm: error bound for hessian matrix} and $\1_n^{\top}\Delta\vect{\gamma}^{k,v} = 0$ from Lemma~\ref{lem: sequence restrict in the subspace}. In addition, following the proof of Lemma~\ref{lem:armijo-lower-bound}, we know $\|\Delta\vect{\gamma}^{k,v}\| \leq \|\vect{g}^{k,v}\| / [\lambda_{\min}^+(H_{\rho}^{k,v})]$ and $\lambda_{\min}^+(H_{\rho}^{k,v}) \geq \underline{C}_k >0$ for all $v$ and $\rho \in [0,1/n)$ over $S^{k}(\vect{\gamma})$. Therefore, we further bound the $D_{\text{Armijo}}$ by:

\begin{equation}
    D_{\text{Armijo}} \leq \left[\left(3 \|\vect{a}\|_{\infty} - (1-\sigma) + \frac{L_H}{6\underline{C}_k}\right) \|\vect{g}^{k,v}\| - \left(\frac{1}{2} - \sigma\right) \underline{C}_k\right]\|\Delta\vect{\gamma}^{k,v}\|^2. 
\end{equation}
Since $\|\vect{g}^{k,v}\| \to 0$ as $v \to \infty$, for sufficiently large $v$, we have $D_{\text{Armijo}} \leq 0$, which means $t^{k,v} = 1$.
\end{proof}





Now we begin to proof Theorem~\ref{thm:sparse-newton-rate}.

\begin{proof}
The following proof is motivated by~\cite{tang2024safe}. By Theorem~\ref{thm: min and max eigenvalues} and Lipschitz continuity of $H^k(\vect{\gamma})$, we have
\begin{align*}
    & \left\|\vect{g}^{k,v}\right\| \leq 1 / (2\eta) \left\|\vect{\gamma}^{k,v} - \vect{\gamma}^{k\star }\right\|, \\
    & \left\|\vect{g}^{k,v} - g^k\left(\vect{\gamma}^{k\star }\right)-H^{k,v}\left(\vect{\gamma}^{k,v} - \vect{\gamma}^{k\star }\right)\right\| \leq \frac{L_H}{2}\left\|\vect{\gamma}^{k,v} - \vect{\gamma}^{k\star }\right\|^2.
\end{align*}
In addition, denote $B_{\rho}^{k,v} := H_{\rho}^{k,v} + \left\|\vect{g}^{k,v}\right\|I$, we have
\begin{align*}
    \left\| H^{k,v} -B_{\rho}^{k,v} \right\| \leq \left\| H^{k,v} - H^{k,v}_{\rho} \right\| + \left\| \vect{g}^{k,v} \right\| \leq \left(6\|\vect{a}\|_{\infty} + 1\right) \left\| \vect{g}^{k,v} \right\| \leq \frac{6\|\vect{a}\|_{\infty} + 1}{2\eta} \left\|\vect{\gamma}^{k,v} - \vect{\gamma}^{k\star }\right\|, 
\end{align*}
where the second inequality holds due to Theorem~\ref{thm: error bound for hessian matrix}. 

In addition, since $\vect{\gamma}^{k,v} \in S^{k}(\vect{\gamma})$, we have
\begin{align*}
    \left\|(B_{\rho}^{k,v})^{-1}\right\| = \left(\lambda_{\min}^+(B_{\rho}^{k,v}) \right)^{-1} \leq \left(\lambda_{\min}^+(H_{\rho}^{k,v}) \right)^{-1} \leq \underline{C}_k^{-1},
\end{align*}
where $\underline{C}_k>0$ is defined in Lemma~\ref{lem: uniform lower bound for min eigenvalue of hessian}. As a result,
\begin{align*}
\left\|\vect{\gamma}^{k, v+1} - \vect{\gamma}^{k\star }\right\| & =\left\|\vect{\gamma}^{k, v} + \Delta \vect{\gamma}^{k,v} - \vect{\gamma}^{k\star }\right\|=\left\|\vect{\gamma}^{k, v} - \vect{\gamma}^{k\star } - (B_{\rho}^{k,v})^{-1} g^{k,v}\right\| \\
& =\left\|(B_{\rho}^{k,v})^{-1}\left(B_{\rho}^{k,v} \left(\vect{\gamma}^{k,v} - \vect{\gamma}^{k\star }\right) - g^{k,v}\right)\right\| \\
& \leq\left\|(B_{\rho}^{k,v})^{-1}\right\| \cdot\left\|g^{k,v} - g^k\left(\vect{\gamma}^{k\star }\right)-H^{k,v}\left(\vect{\gamma}^{k,v} - \vect{\gamma}^{k\star }\right) + \left(H^{k,v}-B_{\rho}^{k,v}\right)\left(\vect{\gamma}^{k,v} - \vect{\gamma}^{k\star }\right)\right\| \\
& \leq \underline{C}_k^{-1} \left(\left\|g^{k,v} - g^k\left(\vect{\gamma}^{k\star }\right)-H^{k,v}\left(\vect{\gamma}^{k,v} - \vect{\gamma}^{k\star }\right)\right\| + \left\|\left(H^{k,v} -B_{\rho}^{k,v}\right) \left(\vect{\gamma}^{k,v} - \vect{\gamma}^{k\star }\right)\right\|\right) \\
& \leq \underline{C}_k^{-1} \left(\frac{1}{2} L_H\left\|\vect{\gamma}^{k,v} - \vect{\gamma}^{k\star }\right\|^2 + \frac{6\|\vect{a}\|_{\infty} + 1}{2\eta} \left\|\vect{\gamma}^{k,v} - \vect{\gamma}^{k\star }\right\|^2\right) \\
& = \underline{C}_k^{-1} \left(\frac{1}{2}L_H + \frac{6\|\vect{a}\|_{\infty} + 1}{2\eta} \right) \left\|\vect{\gamma}^{k,v} - \vect{\gamma}^{k\star }\right\|^2 .
\end{align*}
By setting $L\coloneqq \underline{C}_k^{-1} \left(\frac{1}{2}L_H + \frac{6\|\vect{a}\|_{\infty} + 1}{2\eta} \right)$, we have
\begin{equation*}
    \|\vect{\gamma}^{k, v+1} - \vect{\gamma}^{k\star}\| \leq L \|\vect{\gamma}^{k, v} - \vect{\gamma}^{k\star}\|^2.
\end{equation*}
\end{proof}

The following lemma quantifies when the inner iterations transition from the globalization phase governed by Armijo line search to the pure Newton phase, and provides an explicit upper bound on the number of iterations required for this transition.

\begin{lemma}
\label{lem: phase_transition_complexity}
Fix an outer iteration $k$, and let $\{\vect{\gamma}^{k,v}\}_v$ be the inner sequence generated by Algorithm~\ref{alg: IBSN}. Let $\lyp^{k\star}:=\lyp^k(\vect{\gamma}^{k\star})$ denote the optimal value of the subproblem restricted to $\1_n^\perp$. Define
\[
    C_A
    :=
    3\|\vect a\|_\infty
    -(1-\sigma)
    +\frac{L_H}{6\underline{C}_k}.
\]
Then the following two cases hold.

\begin{enumerate}[label=(\roman*), leftmargin=*]
    \item \textbf{Case 1: $C_A\leq 0$.}
    The Armijo line search accepts the full Newton step at every inner iteration; that is, $t^{k,v}=1$, $\forall v\geq 1$.
    \item \textbf{Case 2: $C_A>0$.}
    The inner iterations enter the pure Newton phase once the gradient norm satisfies
    \[
        \|\vect g^{k,v}\|\leq \epsilon_{\mathrm{phase}},
    \]
    where $\epsilon_{\mathrm{phase}}:=\min\{\epsilon_{\mathrm{Newton}},\epsilon_{\mathrm{contract}}\}$, $\epsilon_{\mathrm{Newton}} := \frac{(1/2-\sigma)\underline{C}_k}{C_A}$, and $\epsilon_{\mathrm{contract}} := \frac{2\eta\underline{C}_k^2}{L}$. 
    From that point onward, the full Newton step is accepted at all subsequent inner iterations. Moreover, let $V_{\mathrm{global}}$ denote the maximum number of inner iterations required to enter this local quadratic convergence phase. Then
    \begin{equation}
    \label{eq:phase-transition-bound}
        V_{\mathrm{global}}
        \leq
        \left\lfloor
        \frac{
        \big(\lyp^k(\vect{\gamma}^{k,0})-\lyp^{k\star}\big)(1+4\eta)
        }{
        2\eta\sigma t_{\min}^k\epsilon_{\mathrm{phase}}^2
        }
        \right\rfloor,
    \end{equation}
    where $t_{\min}^k$ is the lower bound on the Armijo step size given in Lemma~\ref{lem:armijo-lower-bound}, and $L$ is the quadratic convergence constant in Theorem~\ref{thm:sparse-newton-rate}.
\end{enumerate}
\end{lemma}

\begin{proof}
To guarantee that the algorithm permanently transitions into the pure Newton phase, it is sufficient to reach a state where the full Newton step ($t^{k,v} = 1$) is unconditionally accepted and the subsequent gradient norm is non-increasing ($\|\vect{g}^{k,v+1}\| \le \|\vect{g}^{k,v}\|$). This ensures that once the iterates enter this local convergence region, they will not escape back into the global phase where the backtracking line search truncates the step.

First, we revisit the condition for accepting the full Newton step. In the proof of Lemma~\ref{lem: Armijo stepsize}, we established that the Armijo condition is naturally satisfied for $t^{k,v} = 1$ when the difference $D_{\text{Armijo}} \le 0$. We demonstrated this is algebraically guaranteed when:
\begin{equation*}
    C_A \|\vect{g}^{k,v}\| - \left(\frac{1}{2} - \sigma\right)\underline{C}_k \le 0.
\end{equation*}
This leads to two distinct cases. If $C_A \le 0$, the inequality holds trivially for any valid gradient norm because the subtracted term is strictly positive ($\sigma \in (0,1/2)$ and $\underline{C}_k > 0$). In this scenario, the Armijo condition never truncates the step, and the algorithm executes pure Newton steps globally from $v = 1$.

Conversely, if $C_A > 0$, isolating $\|\vect{g}^{k,v}\|$ yields the explicit finite threshold $\epsilon_{\text{Newton}}$. Thus, whenever $\|\vect{g}^{k,v}\| \le \epsilon_{\text{Newton}}$, the full step $t^{k,v} = 1$ is guaranteed to be accepted.

Second, we determine the non-increasing gradient threshold. From Theorem~\ref{thm:sparse-newton-rate}, a full Newton step yields the local quadratic bound $\|\vect{\gamma}^{k,v+1} - \vect{\gamma}^{k\star}\| \le L\|\vect{\gamma}^{k,v} - \vect{\gamma}^{k\star}\|^2$. Using the $1/(2\eta)$-Lipschitz continuity of the gradient established in Theorem~\ref{thm: min and max eigenvalues} (where $g^k(\vect{\gamma}^{k\star}) = \0_n$), and the strong convexity parameter $\underline{C}_k$ on the restricted subspace, we relate this distance bound back to the gradient norm:
\begin{equation*}
    \|\vect{g}^{k,v+1}\| \le \frac{1}{2\eta} \|\vect{\gamma}^{k,v+1} - \vect{\gamma}^{k\star}\| \le \frac{L}{2\eta} \|\vect{\gamma}^{k,v} - \vect{\gamma}^{k\star}\|^2 \le \frac{L}{2\eta \underline{C}_k^2} \|\vect{g}^{k,v}\|^2.
\end{equation*}
For the sequence of gradient norms to be non-increasing (i.e., $\|\vect{g}^{k,v+1}\| \le \|\vect{g}^{k,v}\|$), it is sufficient that the coefficient satisfies $\frac{L}{2\eta \underline{C}_k^2} \|\vect{g}^{k,v}\| \le 1$, which algebraically reduces to:
\begin{equation*}
    \|\vect{g}^{k,v}\| \le \frac{2\eta \underline{C}_k^2}{L} := \epsilon_{\text{contract}}.
\end{equation*}
Therefore, assuming $C_A > 0$, if an iterate satisfies the joint condition $\|\vect{g}^{k,v}\| \le \min(\epsilon_{\text{Newton}}, \epsilon_{\text{contract}}) := \epsilon_{\text{phase}}$, the algorithm accepts the full step $t^{k,v} = 1$ and ensures that $\|\vect{g}^{k,v+1}\| \le \|\vect{g}^{k,v}\| \le \epsilon_{\text{phase}}$. By mathematical induction, all subsequent inner iterations will remain bounded by this threshold, executing pure Newton steps indefinitely.

Consequently, to mathematically bound the maximum number of iterations $V_{\text{global}}$ prior to this permanent transition, we analyze the worst-case trajectory where the iterates have not yet satisfied this joint sufficient condition, meaning $\|\vect{g}^{k,v}\| > \epsilon_{\text{phase}}$. 

We recall the uniform sufficient decrease condition established in~\eqref{equ: decrease property}:
\begin{equation*}
    \lyp^k(\vect{\gamma}^{k,v}) - \lyp^k(\vect{\gamma}^{k,v+1}) \ge \sigma t^k_{\min} \frac{1}{1/(2\eta) + 2} \|\vect{g}^{k,v}\|^2 = \sigma t^k_{\min} \frac{2\eta}{1 + 4\eta} \|\vect{g}^{k,v}\|^2.
\end{equation*}
Coupling this with the condition $\|\vect{g}^{k,v}\| > \epsilon_{\text{phase}}$, the objective function is guaranteed to decrease by a strictly bounded minimum amount at every step prior to the guaranteed phase transition:
\begin{equation*}
    \lyp^k(\vect{\gamma}^{k,v}) - \lyp^k(\vect{\gamma}^{k,v+1}) > \frac{2\eta \sigma t^k_{\min}}{1 + 4\eta} \epsilon_{\text{phase}}^2 := \Delta_{\min}.
\end{equation*}
Because the objective function evaluates to $\lyp^k(\vect{\gamma}^{k,0})$ at initialization and is bounded globally from below by $\lyp^{k\star}$, the cumulative decrease over the first $V_{\text{global}}$ iterations cannot exceed $\lyp^k(\vect{\gamma}^{k,0}) - \lyp^{k\star}$. Summing the functional decreases over these initial global iterations yields:
\begin{equation*}
    V_{\text{global}} \cdot \Delta_{\min} \le \sum_{v=0}^{V_{\text{global}}-1} (\lyp^k(\vect{\gamma}^{k,v}) - \lyp^k(\vect{\gamma}^{k,v+1})) \le \lyp^k(\vect{\gamma}^{k,0}) - \lyp^{k\star}.
\end{equation*}
Rearranging this inequality for $V_{\text{global}}$ and substituting the definition of $\Delta_{\min}$ yields the final bound on the number of iterations required to guarantee reaching the permanent quadratic phase:
\begin{equation*}
    V_{\text{global}} \le \frac{\lyp^k(\vect{\gamma}^{k,0}) - \lyp^{k\star}}{\Delta_{\min}} = \frac{(\lyp^k(\vect{\gamma}^{k,0}) - \lyp^{k\star})(1 + 4\eta)}{2\eta \sigma t^k_{\min} \epsilon_{\text{phase}}^2}.
\end{equation*}
Taking the floor of this bound completes the proof.
\end{proof}

\subsection{Mapping to the inexact Bregman convergence theorem}
\label{app:inexact-bregman-convergence}

Let $\widetilde X^{k+1}:=P_\Omega(X^{k+1})\in\Omega$ denote the feasible point used in the stopping criterion. The subproblem~(\ref{opt: subproblem}) is a specialization of the inexact Bregman proximal point framework of~\cite{yang2022bregman} with constant proximal parameter $\eta$. In the OT setting, the dual first-order relation implies that the residual and subgradient-error terms in their framework vanish. Hence, the only remaining inexactness condition is
\begin{equation}
D_\phi(\widetilde X^{k+1},X^{k+1})\leq\mu_k.
\end{equation}
Therefore, if $\sum_{k=0}^\infty\mu_k<\infty$, \cite[Theorem~3.2]{yang2022bregman} guarantees that $\{X^k\}$ and $\{\widetilde X^k\}$ converge to the same optimal solution of the original OT problem.

To characterize the last-iterate rate, define $\xi_k:=f(\widetilde X^{k+1})-f(\widetilde X^k)$. Specializing~\cite[Eq.~(3.5)]{yang2022bregman} gives
\begin{equation}
\xi_k\leq\eta(\mu_{k-1}+\mu_k),
\end{equation}
where $\mu_{-1}:=0$. Since the proximal parameter is constant, the coefficient in their rate estimate satisfies $\sigma_{k-1}=k/\eta$. Thus,
\begin{equation}
\sum_{k=0}^{K-1}\sigma_{k-1}\xi_k\leq\sum_{k=1}^{K-1}k(\mu_{k-1}+\mu_k)\leq 2\sum_{k=0}^{K-1}(k+1)\mu_k.
\end{equation}
Hence, the last-iterate rate is determined by the growth of the weighted inexactness sequence.

For $\mu_k=O((k+1)^{-2})$, we have $\sum_k\mu_k<\infty$, so sequence convergence holds, while
\begin{equation}
\sum_{k=0}^{K-1}(k+1)\mu_k=O(\log K).
\end{equation}
\cite[Theorem~3.1]{yang2022bregman} therefore yields
\begin{equation}
f(\widetilde X^K)-f^\star=O\left(\frac{\eta\log K}{K}\right).
\end{equation}
In contrast, if $\mu_k=O((k+1)^{-p})$ with $p>2$, then
\begin{equation}
\sum_{k=0}^\infty(k+1)\mu_k<\infty,
\end{equation}
and the weighted error term remains uniformly bounded. Consequently,
\begin{equation}
f(\widetilde X^K)-f^\star=O\left(\frac{\eta}{K}\right).
\end{equation}
Thus, $p=2$ is sufficient for global sequence convergence but introduces a logarithmic factor in the general last-iterate bound, whereas $p>2$ recovers the standard $O(\eta/K)$ rate.

\section{Additional Experiment Details}
\label{app: Additional Experiment Details}



\subsection{Derivation of relative KKT residual $\Delta_{\text{kkt}}$}
\label{app: Derivation of relative KKT residual}

One can show that the dual problem of~\eqref{opt: problem} is 
\begin{equation}
    \max_{\vect{\gamma} \in \R^n, \vect{\zeta} \in \R^m} \vect{a}^\top \vect{\zeta} + \vect{b}^\top \vect{\gamma} \quad \text{s.t.} \quad U(\vect{\gamma}, \vect{\zeta}) := C - \vect{\zeta}\1_n^\top - \1_m \vect{\gamma}^\top \geq 0
\end{equation}
with the KKT system
\begin{equation}
\label{equ: kkt system for ot}
    X\1_n = \vect{a}, \quad X^\top \1_m = \vect{b}, \quad \langle X, U(\vect{\gamma}, \vect{\zeta}) \rangle = 0, \quad X \geq 0, \quad U(\vect{\gamma}, \vect{\zeta}) \geq 0.
\end{equation}

Based on~\eqref{equ: kkt system for ot}, the relative KKT residual $\Delta_{\text{kkt}}$ is defined as: 
\begin{equation}
\Delta_{\text{kkt}} := \max\{\Delta_p, \Delta_d, \Delta_c\},
\end{equation}
where $\Delta_p:=\max \left\{\frac{\left\|X \1_n-\vect{a}\right\|}{1+\|\vect{a}\|}, \frac{\left\|X^{\top} \1_m-\vect{b}\right\|}{1+\|\vect{b}\|}, \frac{\|\min \{X, 0\}\|_F}{1+\|X\|_F}\right\}$, $\Delta_d:= \frac{\|\min \{U(\vect{\gamma}, \vect{\zeta}), 0\}\|_F}{1+\|C\|_F}$, $\Delta_c:= \frac{\left| \langle X, U(\vect{\gamma}, \vect{\zeta}) \rangle \right|}{1+\|C\|_F}$, and $\|\cdot\|_F$ is the Frobenius norm of a matrix. 
Clearly, $\Delta_{\text{kkt}} = 0$ if and only if $(X, \vect{\gamma}, \vect{\zeta})$ satisfies the KKT system~\eqref{equ: kkt system for ot}, making $\Delta_{\text{kkt}}$ a natural measure of optimality. In our semi-dual formulation~\eqref{opt: semi-dual problem}, we compute $\vect{\zeta}$ via the \emph{c-transform}: $\zeta_i = \min_{j \in [n]} (C_{ij} - \gamma_j)$ $ \forall i \in [m]$.

\subsection{Impact of the regularization parameter $\eta$}
\label{app: Impact of the regularization parameter}

\begin{figure*}[!t]
{
\resizebox*{0.32 \textwidth}{!}{\includegraphics{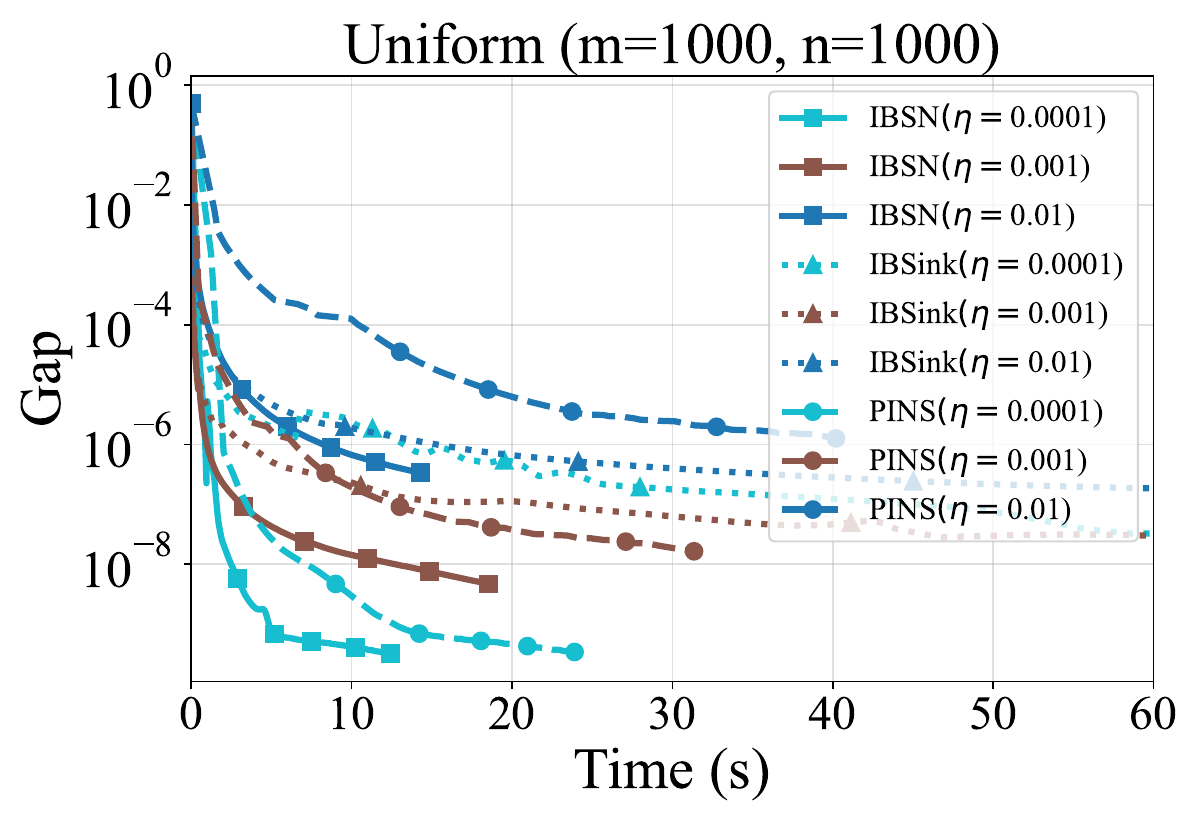}}}\hfill
{
\resizebox*{0.32 \textwidth}{!}{\includegraphics{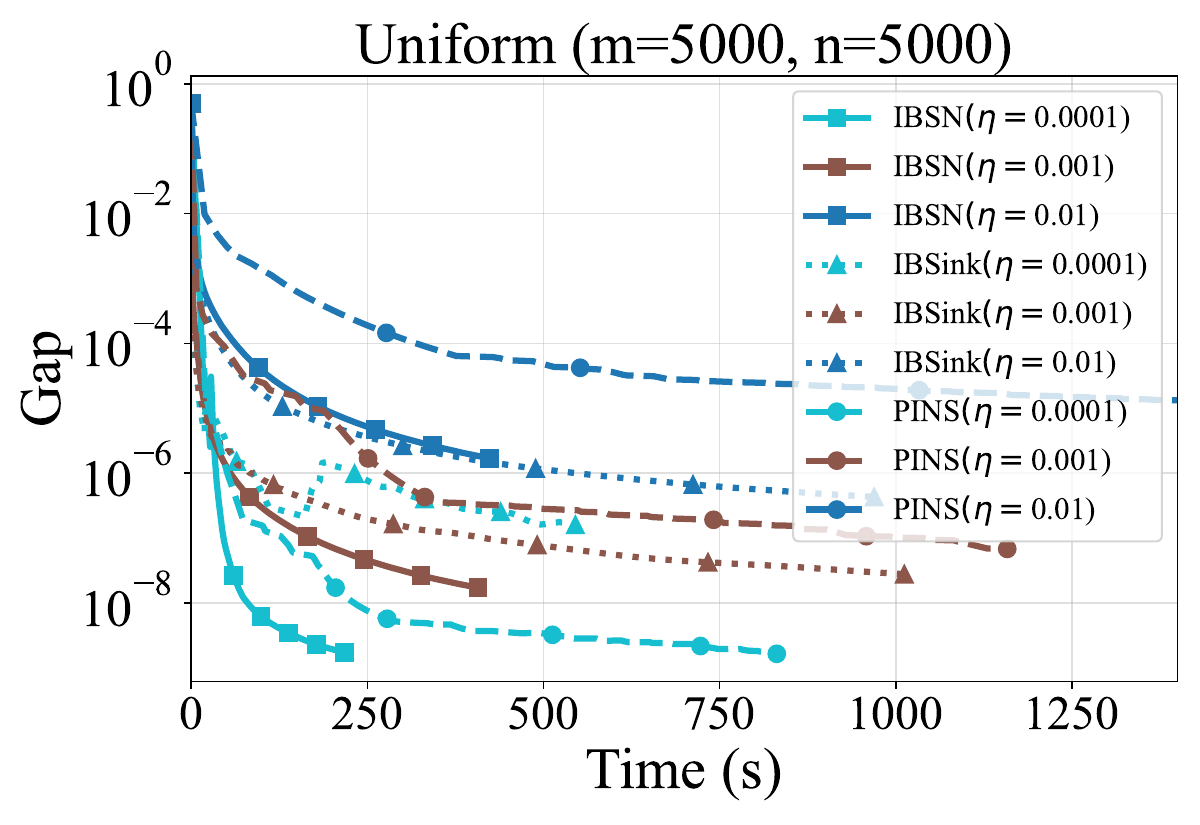}}}\hfill
{
\resizebox*{0.32 \textwidth}{!}{\includegraphics{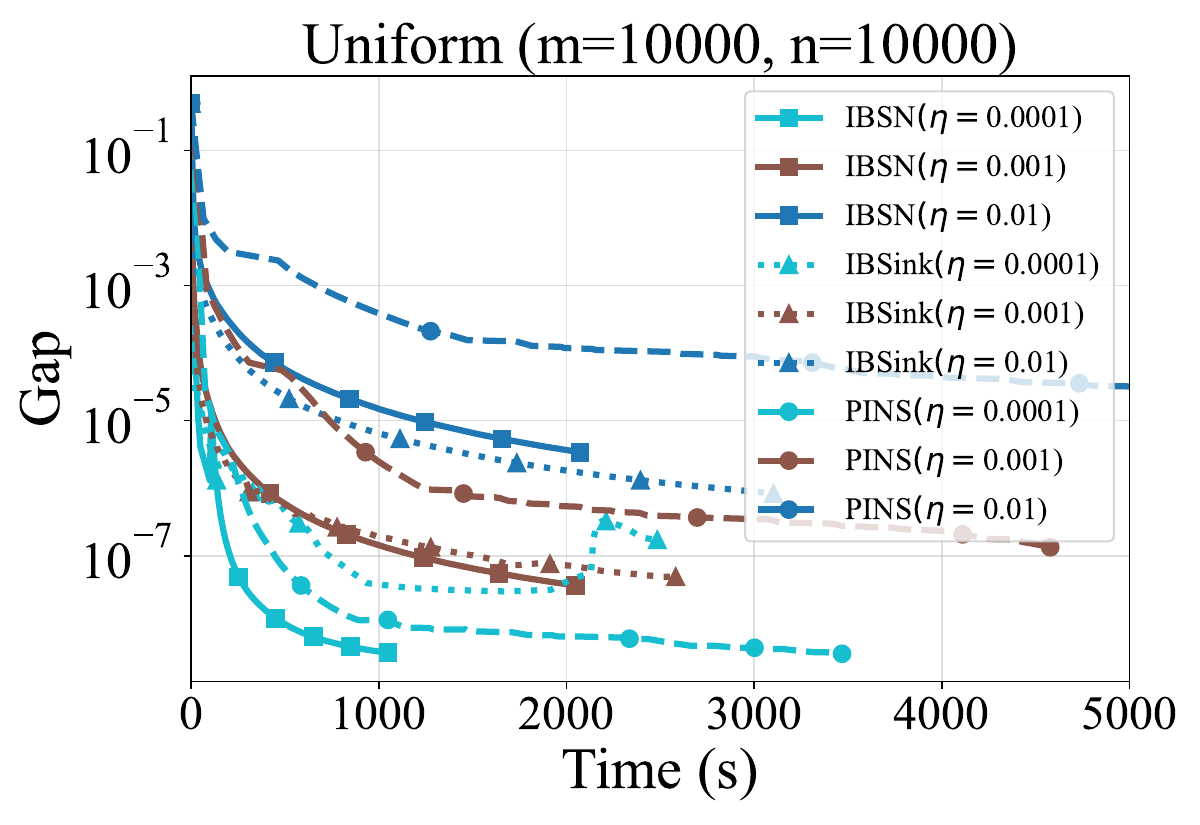}}}
{
\resizebox*{0.32 \textwidth}{!}{\includegraphics{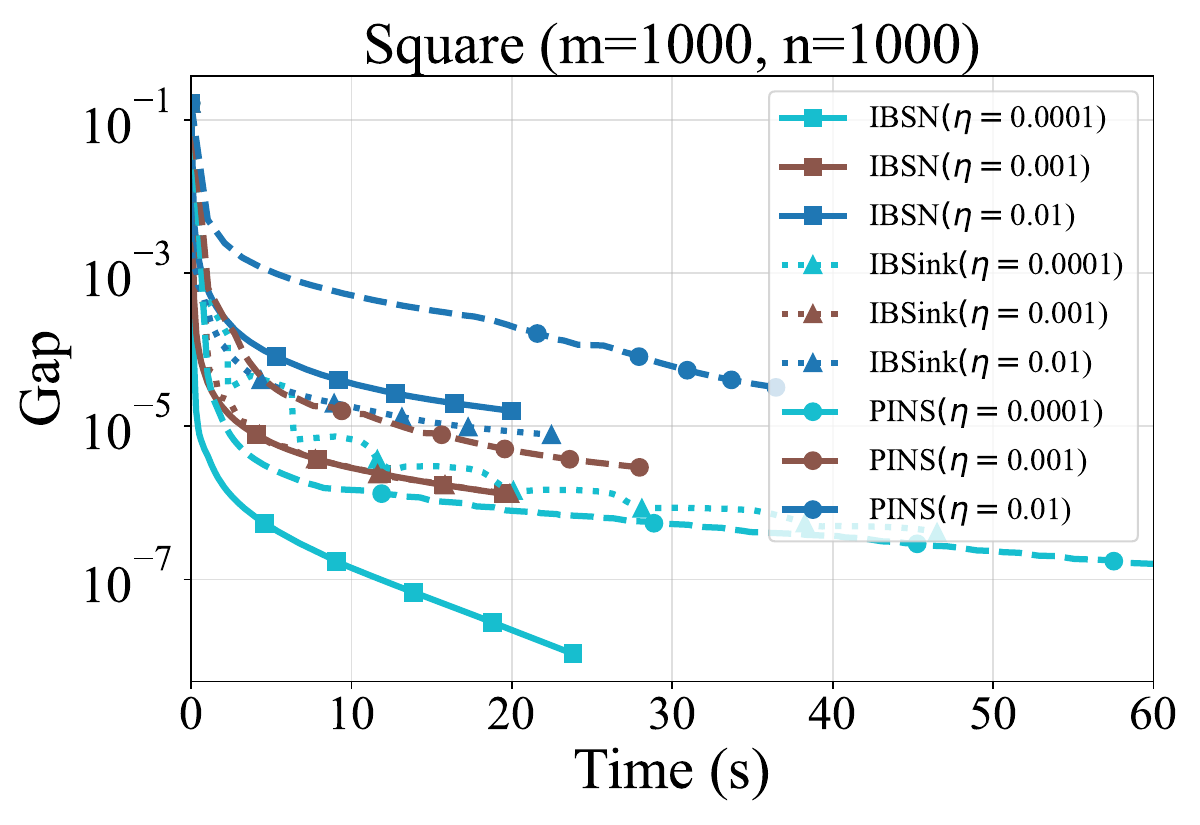}}}\hfill
{
\resizebox*{0.32 \textwidth}{!}{\includegraphics{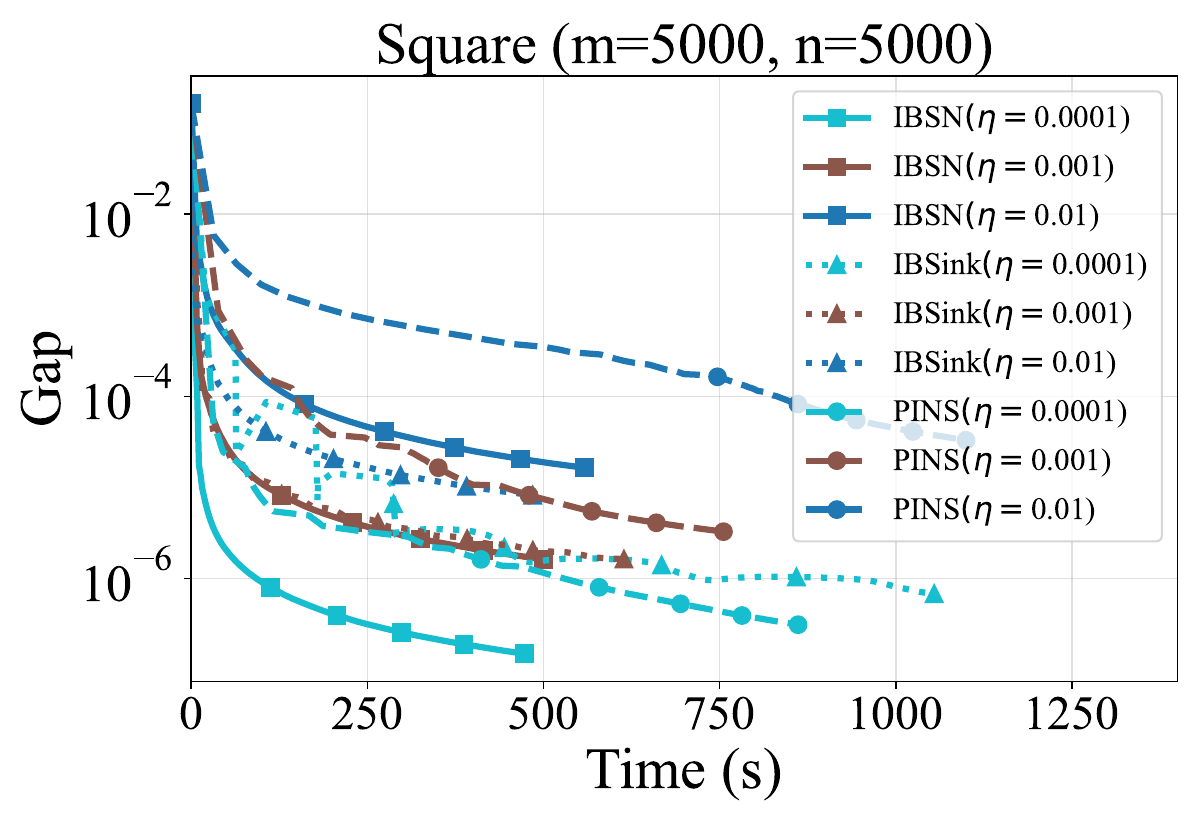}}}\hfill
{
\resizebox*{0.32 \textwidth}{!}{\includegraphics{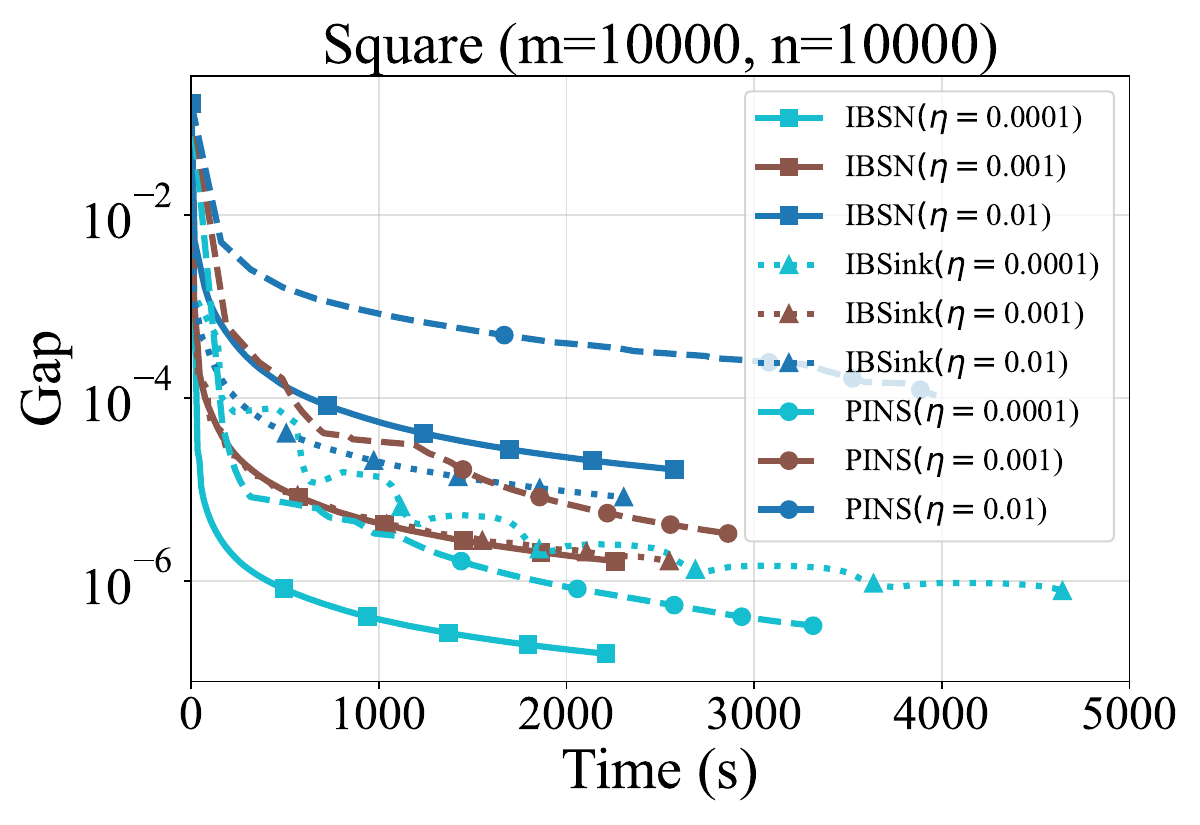}}}
\caption{Impact of the regularization parameter $\eta$ on IBSN, IBSink and PINS on synthetic data. Top: Uniform setting. Bottom: Square setting.}
\label{fig: Impact of the regularization parameter}
\end{figure*}

In this section, we investigate the effect of the regularization parameter $\eta$ on the performance of the IBSN, PINS, and IBSink algorithms under the synthetic settings.

The results, shown in Figure~\ref{fig: Impact of the regularization parameter}, demonstrate that IBSN consistently outperforms PINS across all configurations, achieving faster convergence and smaller objective gaps for the same $\eta$. In addition, as $\eta$ increases, both methods exhibit slower convergence, reflecting the effect of stronger regularization. 

\begin{table}[!t]
\centering
\caption{Performance comparison between IBSN and IBN (IBSN without sparsity process during the Newton stage).}
\label{tab: Performance comparison between IBSN and IBN during the Newton stage}
\resizebox{\textwidth}{!}{
\begin{threeparttable}
\begin{tabular}{ccccc|ccccc}
    \toprule[2pt]
    \multicolumn{10}{c}{Size $m=n=1000$} \\ 
    \cmidrule{1-10}
    \multicolumn{5}{c|}{Synthetic: Uniform} & \multicolumn{5}{c}{Synthetic: Square} \\
    \cmidrule{1-10}
    Alg & NIte & CGTime (s) & Time (s) & Gap & Alg & NIte & CGTime (s) & Time (s) & Gap \\
    \cmidrule{1-10}
    IBN & 177 & 28.38 & 33.47 & 3.18e-10 & IBN & 301 & 61.55 & 72.03 & 1.07e-08 \\
    IBSN & 176 & \textbf{3.86} & \textbf{11.04} & 3.19e-10 & IBSN & 301 & \textbf{11.13} & \textbf{24.68} & 1.07e-08 \\
    \cmidrule{1-10}
    \multicolumn{10}{c}{Size $m=n=5000$} \\ 
    \cmidrule{1-10}
    \multicolumn{5}{c|}{Synthetic: Uniform} & \multicolumn{5}{c}{Synthetic: Square} \\
    \cmidrule{1-10}
    Alg & NIte & CGTime (s) & Time (s) & Gap & Alg & NIte & CGTime (s) & Time (s) & Gap \\
    \cmidrule{1-10}
    IBN & 174 & 497.19 & 659.57 & 1.70e-09 & IBN & 301 & 803.53 & 1101.43 & 1.49e-07\\
    IBSN & 174 & \textbf{9.69} & \textbf{219.22} & 1.70e-09 & IBSN & 301 & \textbf{74.37} & \textbf{457.10} & 1.49e-07 \\
    \cmidrule{1-10}
    \multicolumn{10}{c}{Size $m=n=10000$} \\ 
    \cmidrule{1-10}
    \multicolumn{5}{c|}{Synthetic: Uniform} & \multicolumn{5}{c}{Synthetic: Square} \\
    \cmidrule{1-10}
    Alg & NIte & CGTime (s) & Time (s) & Gap & Alg & NIte & CGTime (s) & Time (s) & Gap \\
    \cmidrule{1-10}
    IBN & 167 & 1688.99 & 2522.38 & 3.65e-09 & IBN & 300 & 2654.52 & 4303.08 & 1.59e-07 \\
    IBSN & 170 & \textbf{16.82} & \textbf{921.60} & 3.63e-09 & IBSN & 300 & \textbf{267.15} & \textbf{2205.74} & 1.59e-07 \\
    \bottomrule[2pt]
\end{tabular}
\begin{tablenotes}
    \footnotesize               
    \item Bold values indicate the most favorable results. The algorithm terminates when $\Delta_{\text{kkt}}<10^{-11}$ or after 300 outer iterations. ``NIte", ``CGTime" and ``Time" represent the total number of Newton iterations, total runtime for finding the Newton directions, and total running time for the algorithm, respectively.
\end{tablenotes}
\end{threeparttable}
}
\end{table}

\subsection{Impact of the sparsity in the Newton phase}
\label{app: Impact of the sparsity in the Newton phase}

To assess the effectiveness of the sparsity process introduced in the Newton phase, we compare the proposed IBSN algorithm with its non-sparse counterpart, IBN (IBSN without sparsity during the Newton stage).
Both methods are evaluated on the Uniform and Square synthetic cost settings with problem sizes $m = n \in \{1000, 5000, 10000\}$. The results are given in Table~\ref{tab: Performance comparison between IBSN and IBN during the Newton stage}.

As shown in the table, introducing sparsity in the Hessian computation leads to a substantial reduction in the computational cost of solving the Newton systems while achieving identical final gaps. In addition, the benefit of sparsity becomes more pronounced in finding a Newton direction as the problem size increases. We also observe that the total runtime advantage decreases as problems get larger. This is primarily because other computational components, such as line search and the evaluation of the stopping criterion based on $D_{\phi}(\Proj_{\Omega}(X^{k+1}), X^{k+1})$, begin to dominate the overall cost. Nevertheless, the significant reduction in the time required for Newton direction computation clearly demonstrates the effectiveness of the proposed sparsity process in accelerating the Newton phase.

\subsection{Comparison with Network Simplex}
\label{app:network_simplex}

We further compare IBSN with a classical network simplex solver for exact OT. For this comparison, we implement IBSN in C++ and use the network simplex solver~\cite{bonneel2011displacement} available through the Python Optimal Transport (POT) package~\cite{flamary2021pot,flamary2024pot}. 

Table~\ref{tab:network_simplex} reports the final objective gap and running time. Neither method uniformly dominates the other. Network Simplex is faster on several balanced instances and often achieves smaller final objective gap. In contrast, IBSN demonstrates competitive performance and is faster on all three tested unbalanced DOTmark instances with $m=4096, n=1024$.

These results indicate that the relative performance of the two methods depends on the problem dimensions and cost structure. In particular, the semi-dual formulation of IBSN reduces the Newton system dimension from $m+n$ to $n$, which can be advantageous when $m$ is substantially larger than $n$. Overall, the results demonstrate that Network Simplex remains highly competitive, while IBSN can offer computational advantages on certain unbalanced OT instances.

\begin{table}[t]
\centering
\caption{Comparison between IBSN and Network Simplex using C++ implementations. Bold values indicate the lower runtime.}
\label{tab:network_simplex}
\small
\setlength{\tabcolsep}{4pt}
\begin{tabular}{@{}llrrrr@{}}
\toprule
& & \multicolumn{2}{c}{IBSN}
& \multicolumn{2}{c}{Network Simplex} \\
\cmidrule(lr){3-4}\cmidrule(l){5-6}
Instance & $(m,n)$ & Gap & Time (s) & Gap & Time (s) \\
\midrule
\multicolumn{6}{@{}l}{\textit{Synthetic: Uniform}} \\
Uniform & $(1000,1000)$ & 8.690e-10 & \textbf{2.459} & 8.092e-11 & 3.319 \\
Uniform & $(2000,1000)$ & 1.431e-07 & \textbf{3.584} & 1.435e-07 & 5.117 \\
Uniform & $(4000,1000)$ & 3.746e-10 & \textbf{7.160} & 1.884e-10 & 8.576 \\
Uniform & $(5000,1000)$ & 5.609e-10 & \textbf{8.605} & 3.532e-11 & 10.055 \\
Uniform & $(5000,500)$  & 9.425e-10 & \textbf{4.184} & 5.765e-11 & 6.257 \\
Uniform & $(5000,5000)$ & 5.628e-07 & 49.00 & 5.644e-11 & \textbf{43.44} \\
\midrule
\multicolumn{6}{@{}l}{\textit{Real: Fashion-MNIST}} \\
IDs 0--1 & $(784,784)$ & 1.342e-09 & \textbf{1.05} & 1.302e-09 & 1.20 \\
IDs 22--54698 & $(784,784)$ & 7.571e-10 & 1.04 & 5.362e-10 & \textbf{0.71} \\
IDs 333--12580 & $(784,784)$ & 1.580e-09 & 1.33 & 1.180e-16 & \textbf{0.62} \\
\midrule
\multicolumn{6}{@{}l}{\textit{Real: DOTmark, balanced case ($m=n$)}} \\
Shapes64 & $(4096,4096)$ & 5.890e-06 & 23.71 & 3.844e-10 & \textbf{14.83} \\
ClassicImages64 & $(4096,4096)$ & 4.315e-06 & \textbf{24.31} & 4.016e-10 & 36.37 \\
WhiteNoise64 & $(4096,4096)$ & 4.141e-06 & 19.15 & 3.242e-16 & \textbf{11.11} \\
\midrule
\multicolumn{6}{@{}l}{\textit{Real: DOTmark, unbalanced case ($m \neq n$)}} \\
Shapes & $(4096,1024)$ & 4.546e-09 & \textbf{6.70} & 1.250e-09 & 10.88 \\
ClassicImages & $(4096,1024)$ & 5.517e-09 & \textbf{12.44} & 4.945e-10 & 17.44 \\
WhiteNoise & $(4096,1024)$ & 5.403e-10 & \textbf{13.31} & 3.609e-16 & 14.70 \\
\bottomrule
\end{tabular}
\end{table}

\subsection{Impact of the semi-dual formulation on Newton direction computation}
\label{app: Impact of the semi-dual formulation on Newton direction computation}


\begin{figure*}[t]
{
\resizebox*{0.32 \textwidth}{!}{\includegraphics{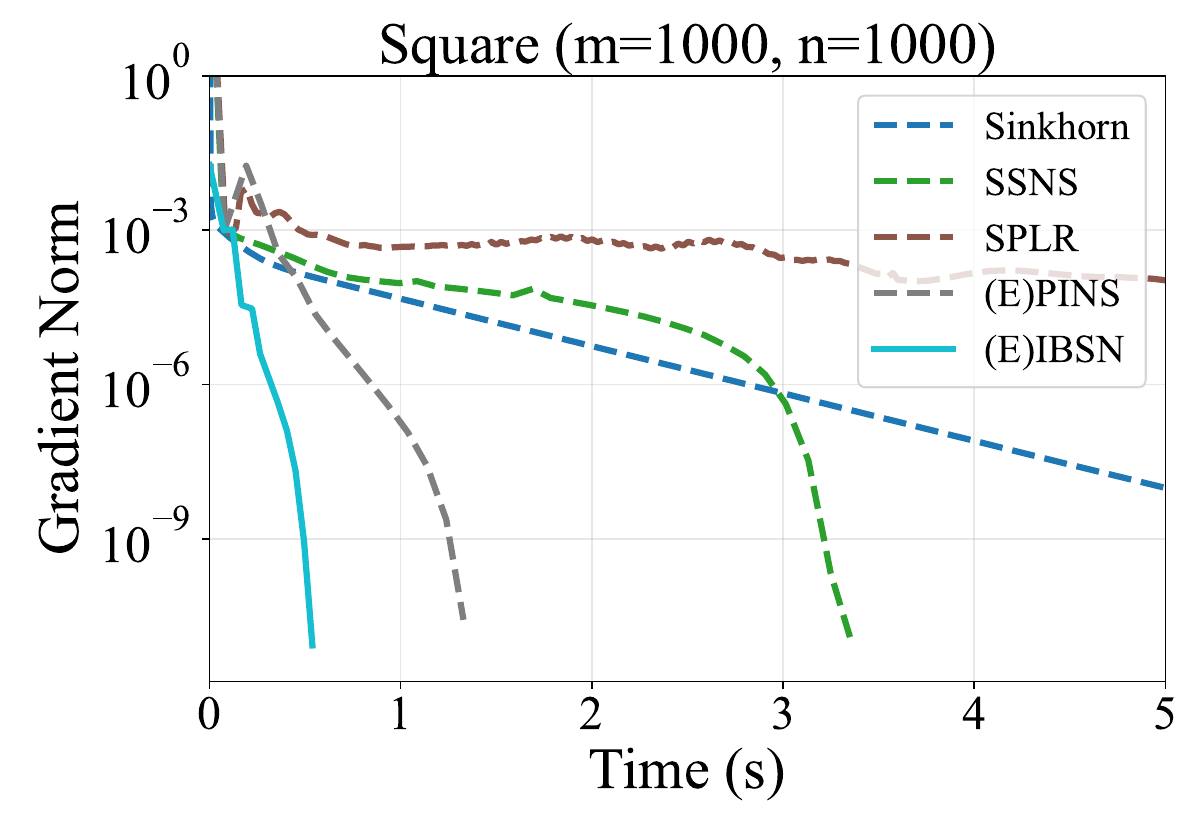}}}\hfill
{
\resizebox*{0.32 \textwidth}{!}{\includegraphics{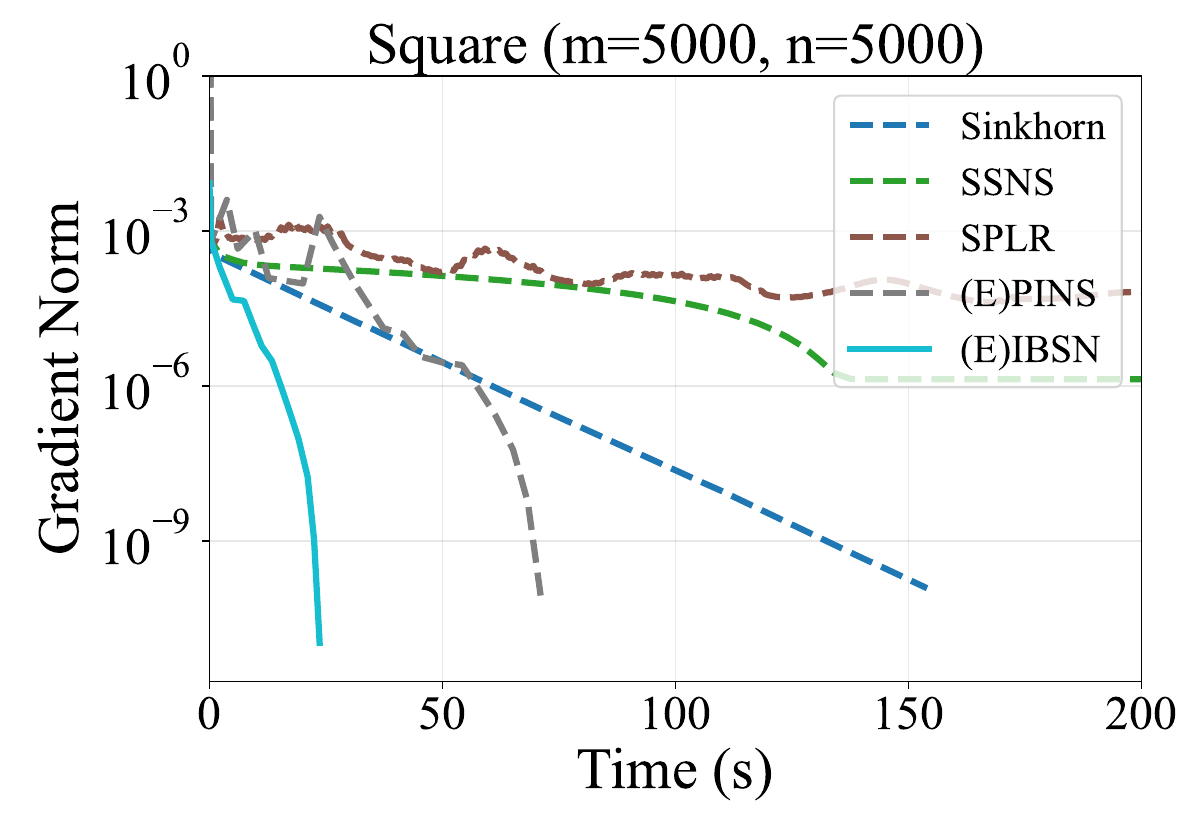}}}\hfill
{
\resizebox*{0.32 \textwidth}{!}{\includegraphics{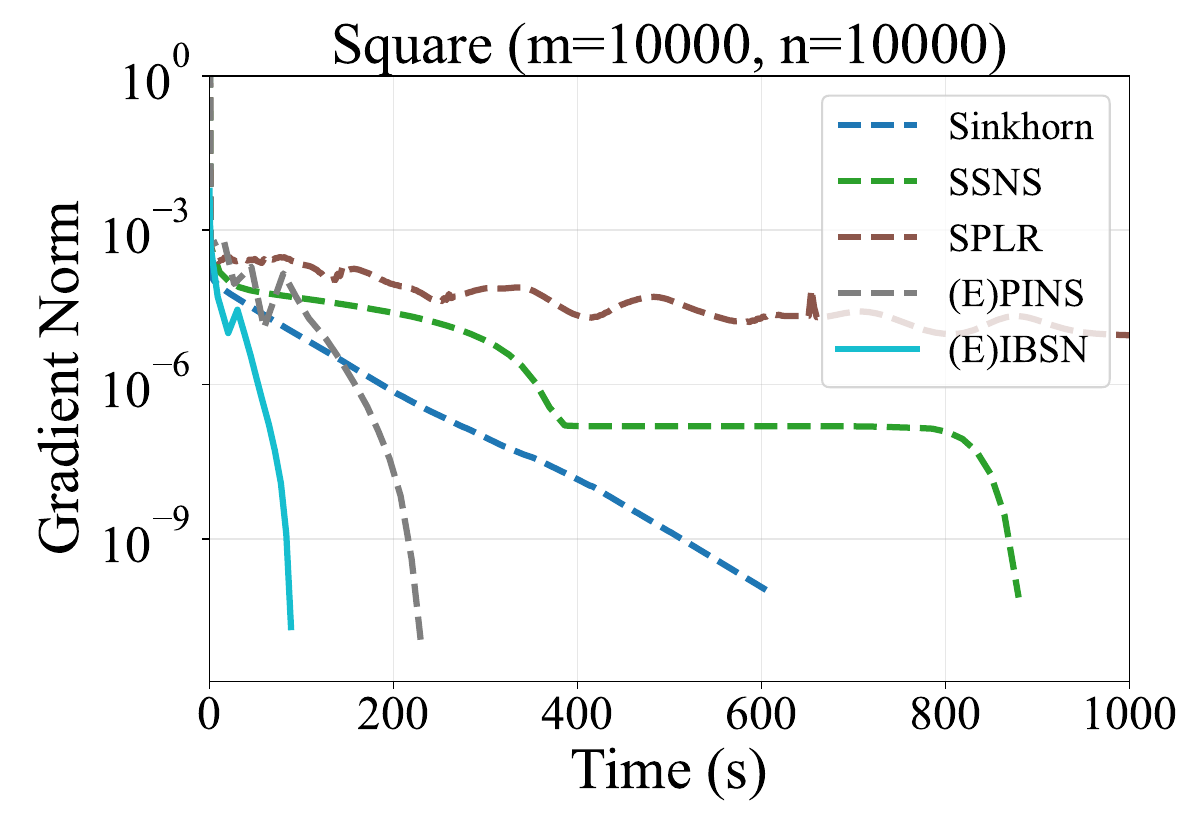}}}
{
\resizebox*{0.32 \textwidth}{!}{\includegraphics{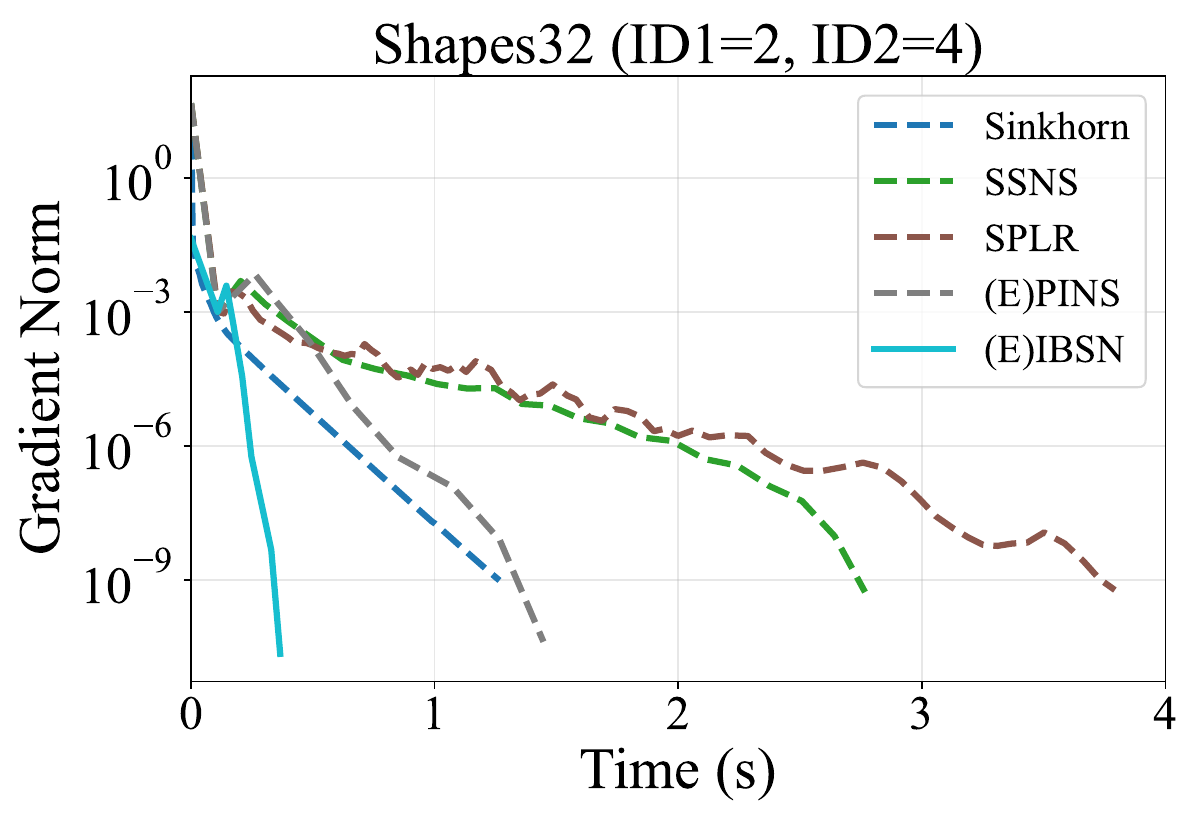}}}\hfill
{
\resizebox*{0.32 \textwidth}{!}{\includegraphics{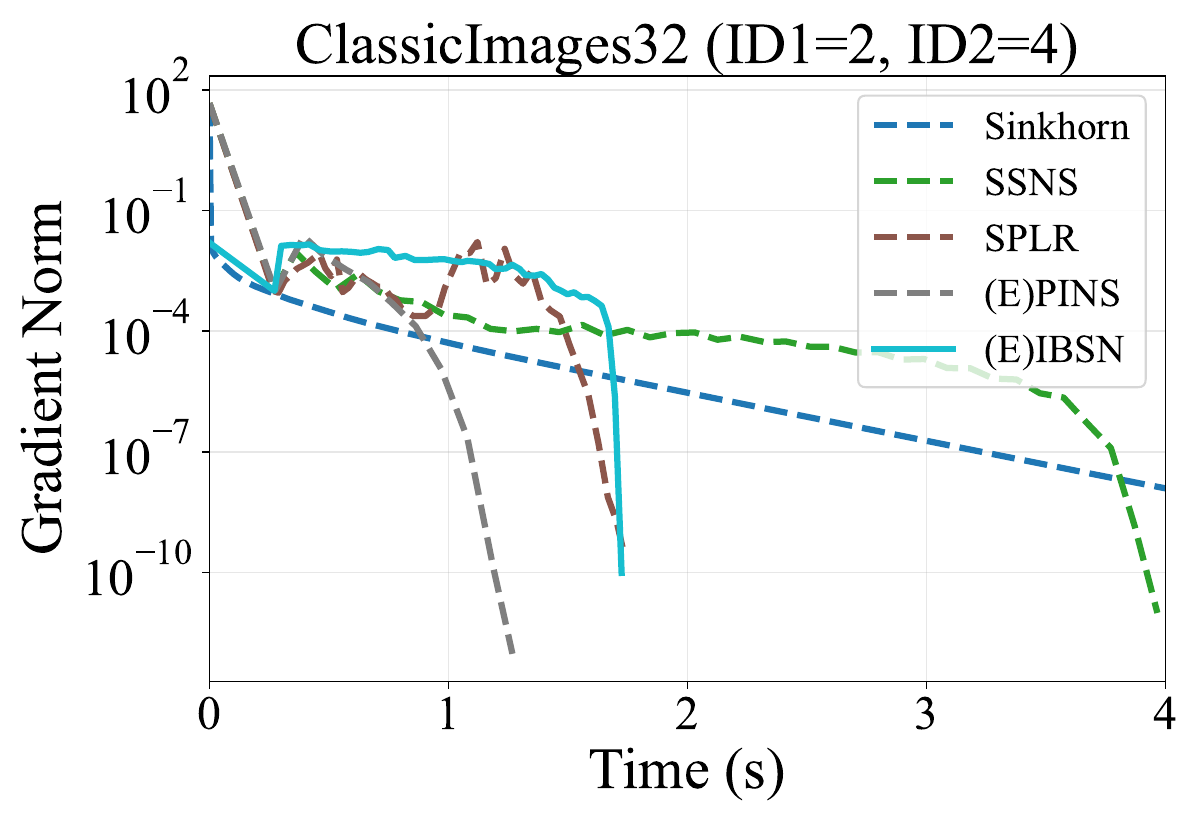}}}\hfill
{
\resizebox*{0.32 \textwidth}{!}{\includegraphics{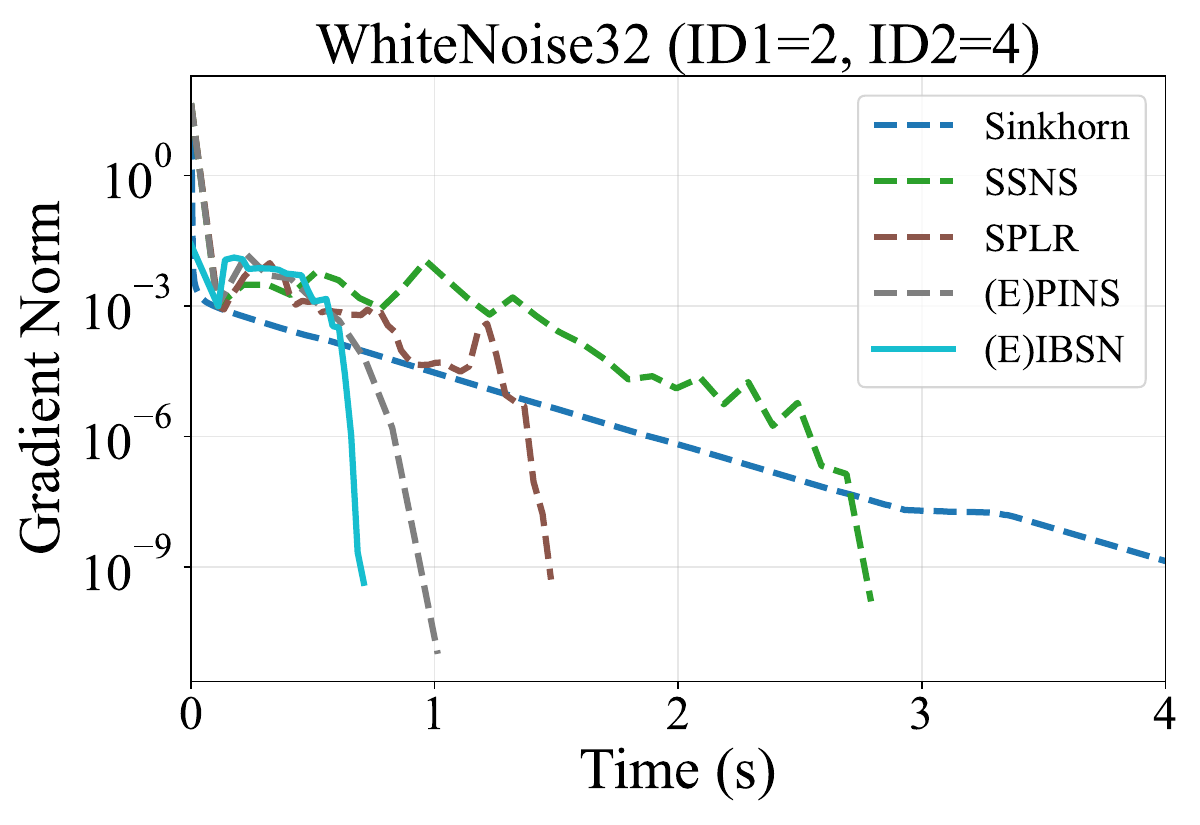}}}
{
\resizebox*{0.32 \textwidth}{!}{\includegraphics{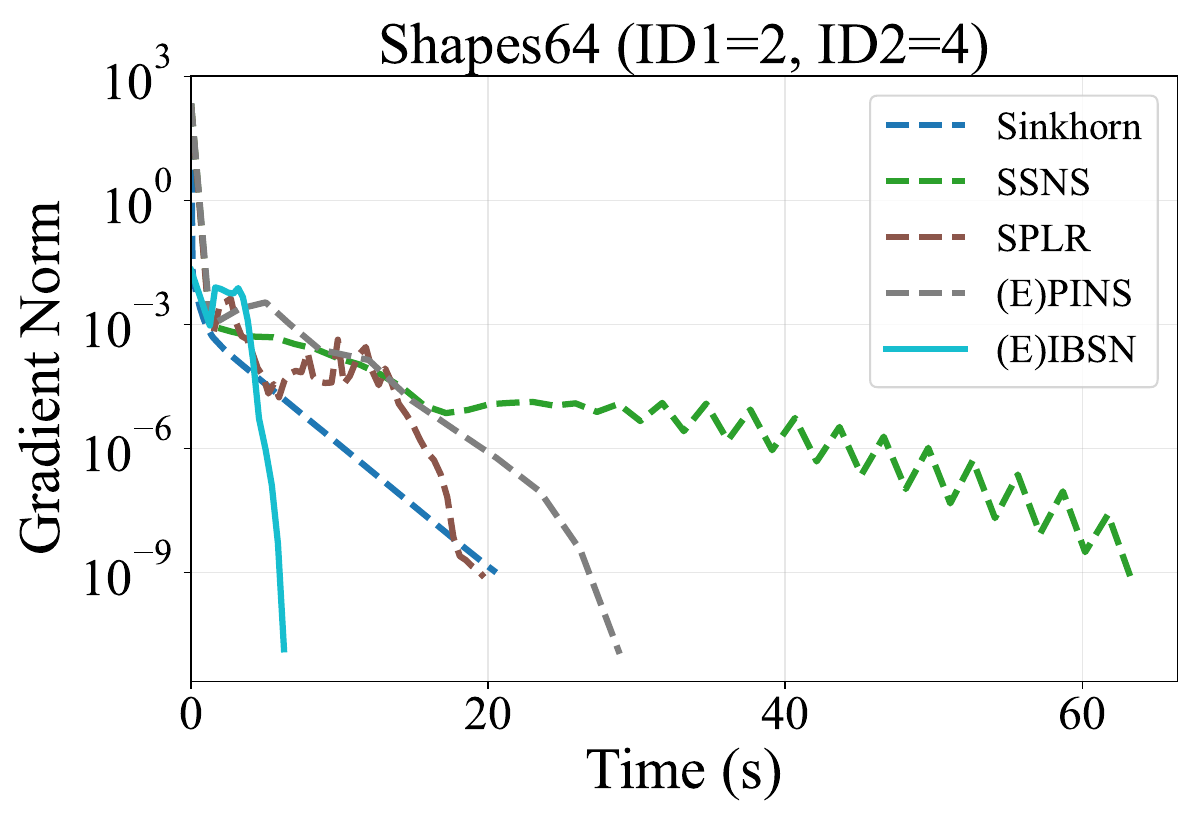}}}\hfill
{
\resizebox*{0.32 \textwidth}{!}{\includegraphics{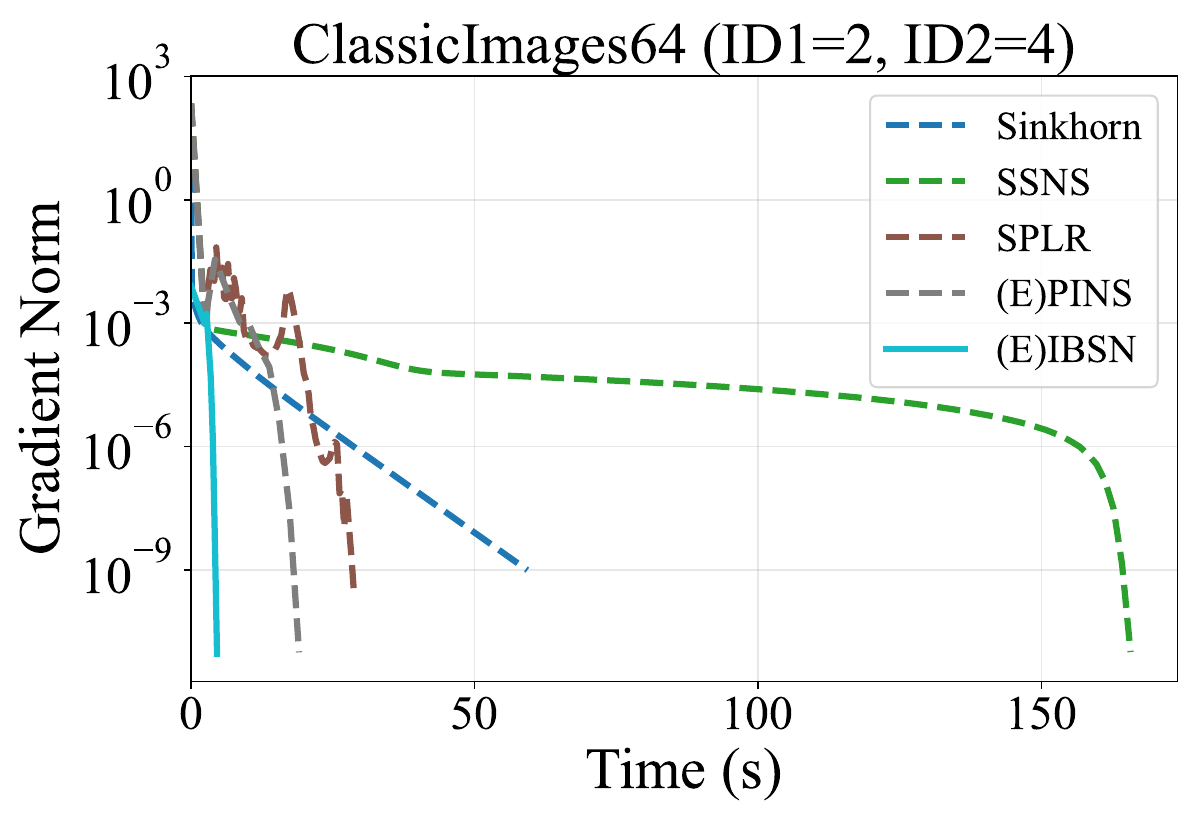}}}\hfill
{
\resizebox*{0.32 \textwidth}{!}{\includegraphics{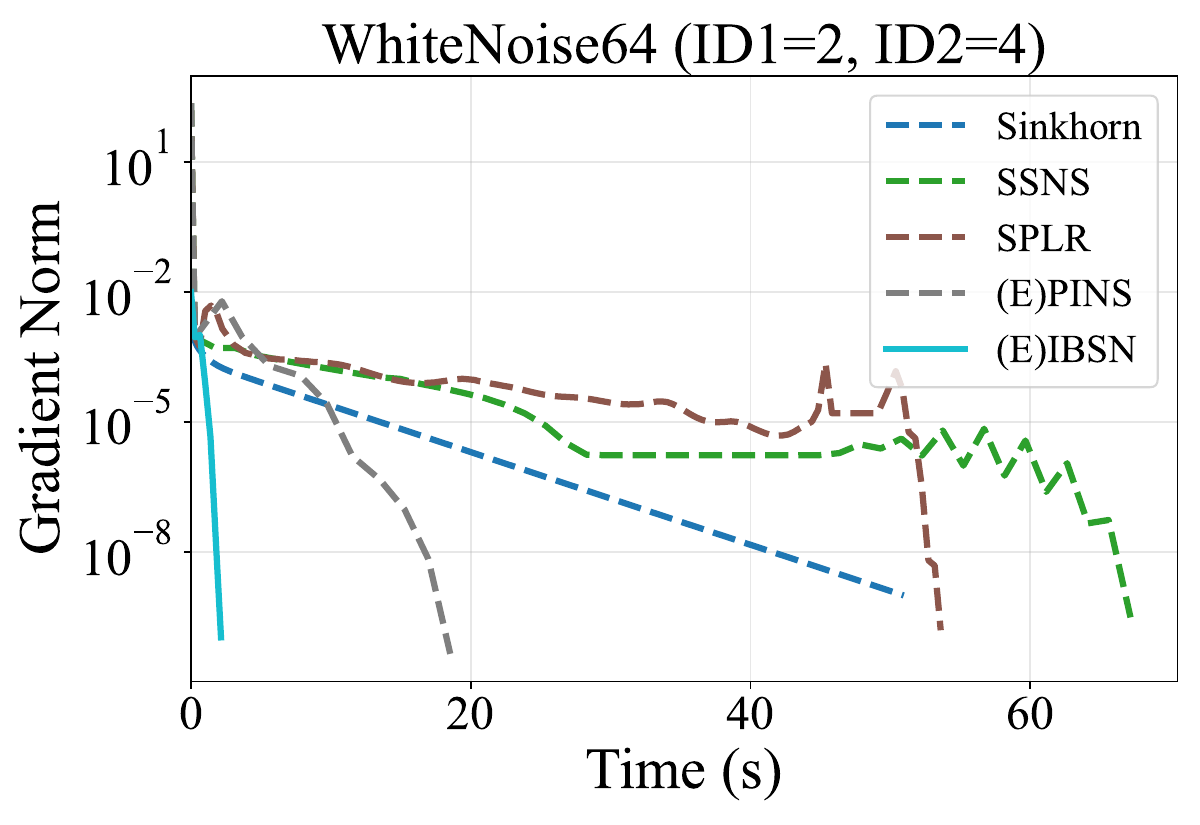}}}
\caption{Performance of different algorithms on EOT problem~\eqref{opt: EOT}. Top: Square setting. Middle: DOTmark with image size $32 \times 32$. Bottom: DOTmark with image size $64 \times 64$.}
\label{fig: Performance of different algorithms on EOT problem.}
\end{figure*}

In this subsection, we investigate the impact of the semi-dual formulation on the efficiency of computing the Newton direction. To do this, we compare the proposed IBSN algorithm with three second-order methods, PINS, SSNS~\cite{tang2024safe}, and SPLR~\cite{wang2025sparse}. The baselines use the standard dual formulation with two dual variables, while IBSN works in the semi-dual space. We also use a shared warm start: run the Sinkhorn algorithm until the gradient norm is below $10^{-3}$, then switch to the Newton-type phase. This reduces sensitivity to initialization and makes the comparison focus on the Newton step. In addition, we also provide the Sinkhorn algorithm as a reference. 

Since PINS and IBSN are originally designed to solve the exact OT problem (which involves a sequence of subproblems), we adapted them for this specific experiment. In both frameworks, each subproblem can be regarded as an EOT problem with a modified cost matrix $C - \eta \log(X^k)$. To ensure they solve the exact same EOT problem as the other methods, we fix $X^k = \1_m \1_n^\top$ and limit the algorithms to solve the subproblem just once, rather than iteratively. This modification ensures that all algorithms address the same EOT problem with the original cost matrix $C$. We refer to these adapted versions as (E)PINS and (E)IBSN. For completeness, we also include their non-sparse counterparts, denoted by (E)PINS(NS) and (E)IBN, respectively.

In the following experiments, we evaluate the impact of the semi-dual formulation across two distinct scenarios: (i) the balanced setting where $m=n$, and (ii) the unbalanced setting where $m > n$. The algorithm terminates when the gradient norm is less than $10^{-9}$ or when the maximum number of iterations is reached.

\paragraph{Performance on balanced $m=n$.} 
We first evaluate performance using the synthetic Square setting and the real DOTmark dataset, with the regularization parameter fixed at $\eta=10^{-4}$. As shown in Figure~\ref{fig: Performance of different algorithms on EOT problem.}, our proposed algorithm (E)IBSN performs well with both synthetic and real data. In addition, second-order methods typically exhibit a faster convergence rate than the Sinkhorn algorithm once they enter the Newton regime. Table~\ref{tab: Performance of different algorithms in Newton direction computation.} presents more details about the cost of the Newton step. Across all experimental settings, (E)IBSN and (E)IBN require the fewest CG iterations, typically around half the number required by the other second-order methods. This improvement is mainly due to the semi-dual formulation, which reduces the dimension of the dual variable and leads to a smaller and simpler linear system. In addition, the CG iteration counts of (E)IBSN and its non-sparse counterpart (E)IBN are close, suggesting that the proposed sparsification strategy does not noticeably increase the difficulty of solving the Newton system.

\begin{table}[!t]
\centering
\caption{Performance of different algorithms in Newton direction computation under EOT problem~\eqref{opt: EOT}.}
\label{tab: Performance of different algorithms in Newton direction computation.}
\resizebox{\textwidth}{!}{
\begin{threeparttable}
\begin{tabular}{c|ccc|ccc|ccc}
    \toprule[2pt]
    \multicolumn{10}{c}{Synthetic: Square} \\ 
    \cmidrule{1-10}
     & \multicolumn{3}{c|}{size $m=n=1000$} & \multicolumn{3}{c|}{size $m=n=5000$} & \multicolumn{3}{c}{size $m=n=10000$}\\
    \cmidrule{1-10}
    Alg & MeanCG & Ite & Time (s) & MeanCG & Ite & Time (s) & MeanCG & Ite & Time (s) \\
    \cmidrule{1-10}
    SSNS & 305.8 (69) & 34 & 3.36 & 712.4 (92) & 116 & 373.37 & 749.4 (135) & 61 & 879.52 \\
    SPLR & - (-) & - & 20.44 & - (-) & - & 554.80 & - (-) & - & 2068.48 \\
    (E)PINS & 292.8 (49) & 14 & 1.33 & 418.0 (143) & 22 & 71.24 & 425.5 (173) & 17 & 229.77 \\
    (E)PINS(NS) & 292.8 (49) & 14 & 3.12 & 418.3 (143) & 22 & 178.41 & 441.7 (166) & 16 & 580.58 \\
    (E)IBN & 167.2 (12) & \textbf{10} & 1.42 & \textbf{290.4} (88) & 13 & 73.49 & 241.3 (97) & \textbf{13} & 274.29\\
    (E)IBSN & \textbf{163.6} (16) & 11 & \textbf{0.54} & \textbf{290.4} (88) & \textbf{13} & \textbf{23.64} & \textbf{241.2} (97) & \textbf{13} & \textbf{88.63} \\
    \cmidrule{1-10}
    \multicolumn{10}{c}{Real: DOTmark (ID1 = 2, ID2 = 4) with size $m = n = 1024$} \\ 
    \cmidrule{1-10}
    & \multicolumn{3}{|c|}{Shapes32} & \multicolumn{3}{c|}{ClassicImages32} & \multicolumn{3}{c}{WhiteNoise32}\\
    \cmidrule{1-10}
    Alg & MeanCG & Ite & Time (s) & MeanCG & Ite & Time (s) & MeanCG & Ite & Time (s) \\
    \cmidrule{1-10}
    SSNS & 1071.3 (383) & 23 & 2.77 & 371.9 (89) & 40 & 3.97 & 346.3 (129) & 29 & 2.79 \\
    SPLR & 694.9 (528) & 80 & 3.79 & 210.3 (144) & 50 & 1.74 & 233.3 (110) & 47 & 1.48 \\
    (E)PINS & 763.2 (251) & 9 & 1.45 & 318.6 (67) & \textbf{11} & \textbf{1.27} & 317.7 (76) & \textbf{10} & 1.01 \\
    (E)PINS(NS) & 771.9 (263) & 9 & 2.91 & 327.9 (54) & \textbf{10} & 1.86 & 317.7 (76) & \textbf{10} & 1.67 \\
    (E)IBN & \textbf{359.1} (77) & \textbf{8} & 1.15 & 164.1 (34) & 38 & 3.49 & 166.4 (26) & 19 & 1.91\\
    (E)IBSN & \textbf{360.3} (79) & \textbf{8} & \textbf{0.37} & \textbf{163.7} (27) & 43 & 1.72 & \textbf{164.6} (34) & 20 & \textbf{0.71} \\
    \cmidrule{1-10}
    \multicolumn{10}{c}{Real: DOTmark (ID1 = 2, ID2 = 4) with size $m=n=4096$} \\ 
    \cmidrule{1-10}
    & \multicolumn{3}{|c|}{Shapes64} & \multicolumn{3}{c|}{ClassicImages64} & \multicolumn{3}{c}{WhiteNoise64}\\
    \cmidrule{1-10}
    Alg & MeanCG & Ite & Time (s) & MeanCG & Ite & Time (s) & MeanCG & Ite & Time (s) \\
    \cmidrule{1-10}
    SSNS & 770.4 (190) & 43 & 63.24 & 275.5 (99) & 113 & 165.73 & 396.4 (87) & 46 & 67.17 \\
    SPLR & 509.4 (205) & 46 & 19.73 & 151.0 (110) & 71 & 28.75 & 219.8 (76) & 123 & 53.60 \\
    (E)PINS & 617.9 (206) & \textbf{12} & 28.85 & 265.1 (57) & 11 & 19.05 & 324.5 (78) & 11 & 18.53 \\
    (E)PINS(NS) & 617.9 (206) & \textbf{12} & 55.95 & 265.1 (57) & 11 & 27.53 & 324.5 (78) & 11 & 30.70 \\
    (E)IBN & 236.1 (113) & 14 & 24.84 & \textbf{162.3} (27) & \textbf{7} & 10.86 & \textbf{195.2} (25) & \textbf{6} & 10.00 \\
    (E)IBSN & \textbf{236.0} (113) & 14 & \textbf{6.26} & \textbf{162.3} (27) & \textbf{7} & \textbf{4.58} & \textbf{195.2} (25) & \textbf{6} & \textbf{2.14} \\
    \bottomrule[2pt]
\end{tabular}
\begin{tablenotes}
    \footnotesize               
    \item \normalsize Bold values indicate the most favorable results. ``MeanCG" denotes the mean number (standard deviation) of CG iterations in finding a Newton direction, and ``Ite" denotes the total number of iterations. ``-'' indicates the algorithm attains its maximum number of iterations, and its gradient norm is larger than $10^{-9}$. 
\end{tablenotes}
\end{threeparttable}
}
\end{table}


\begin{figure*}[!t]
{
\resizebox*{0.32 \textwidth}{!}{\includegraphics{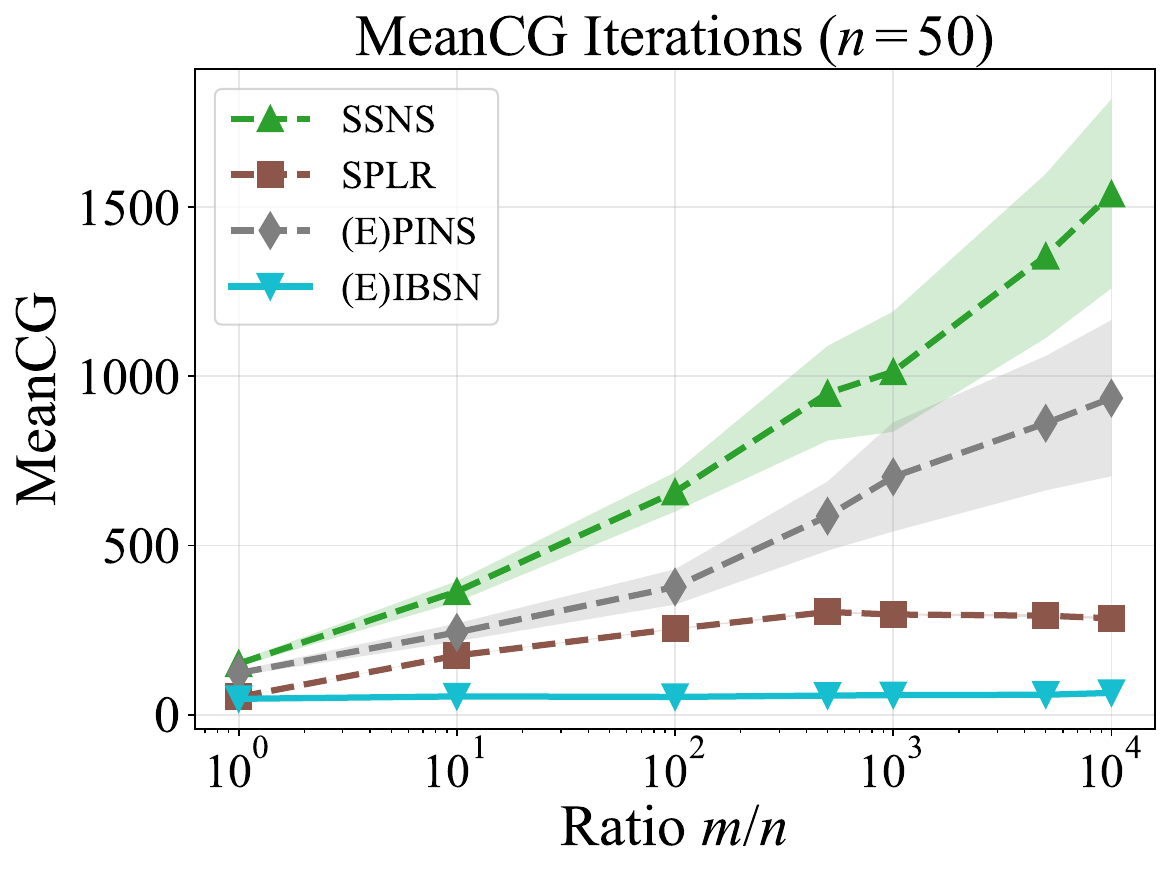}}}\hfill
{
\resizebox*{0.32 \textwidth}{!}{\includegraphics{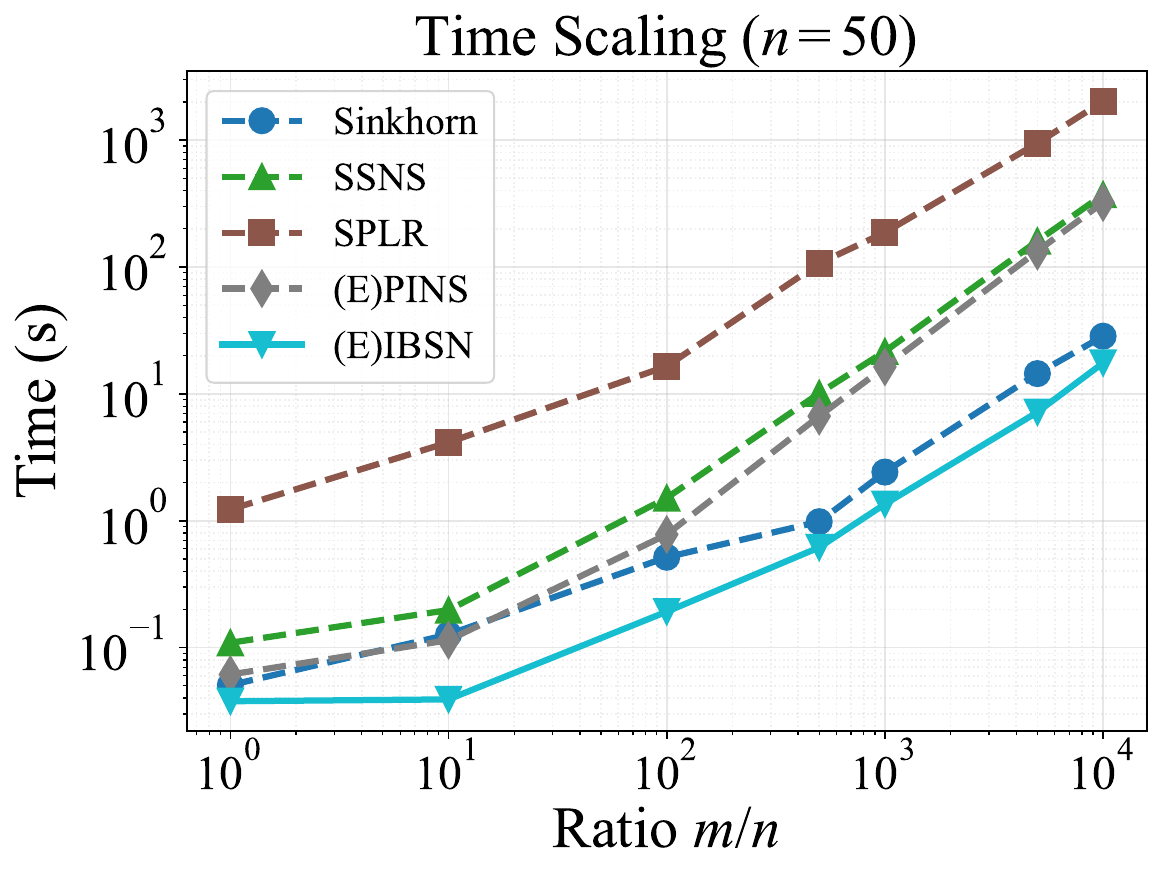}}}\hfill
{
\resizebox*{0.32 \textwidth}{!}{\includegraphics{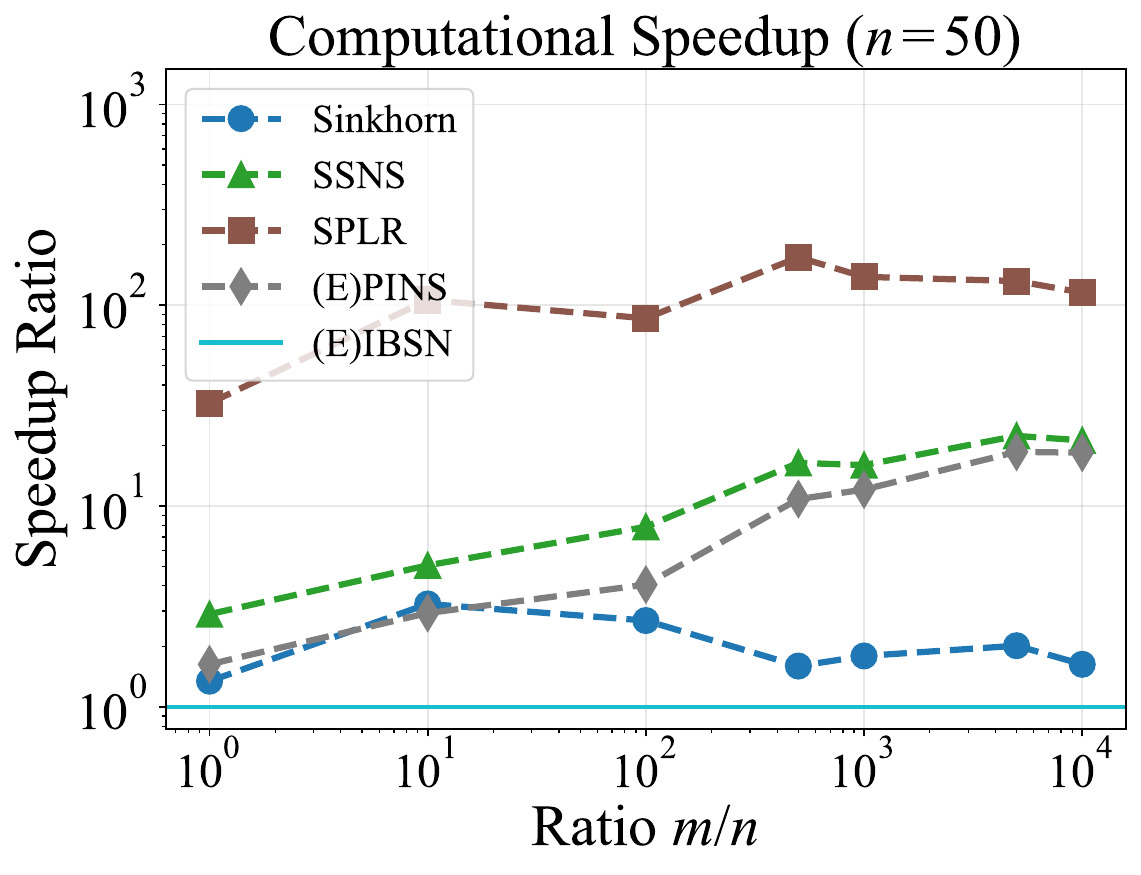}}}
{
\resizebox*{0.32 \textwidth}{!}{\includegraphics{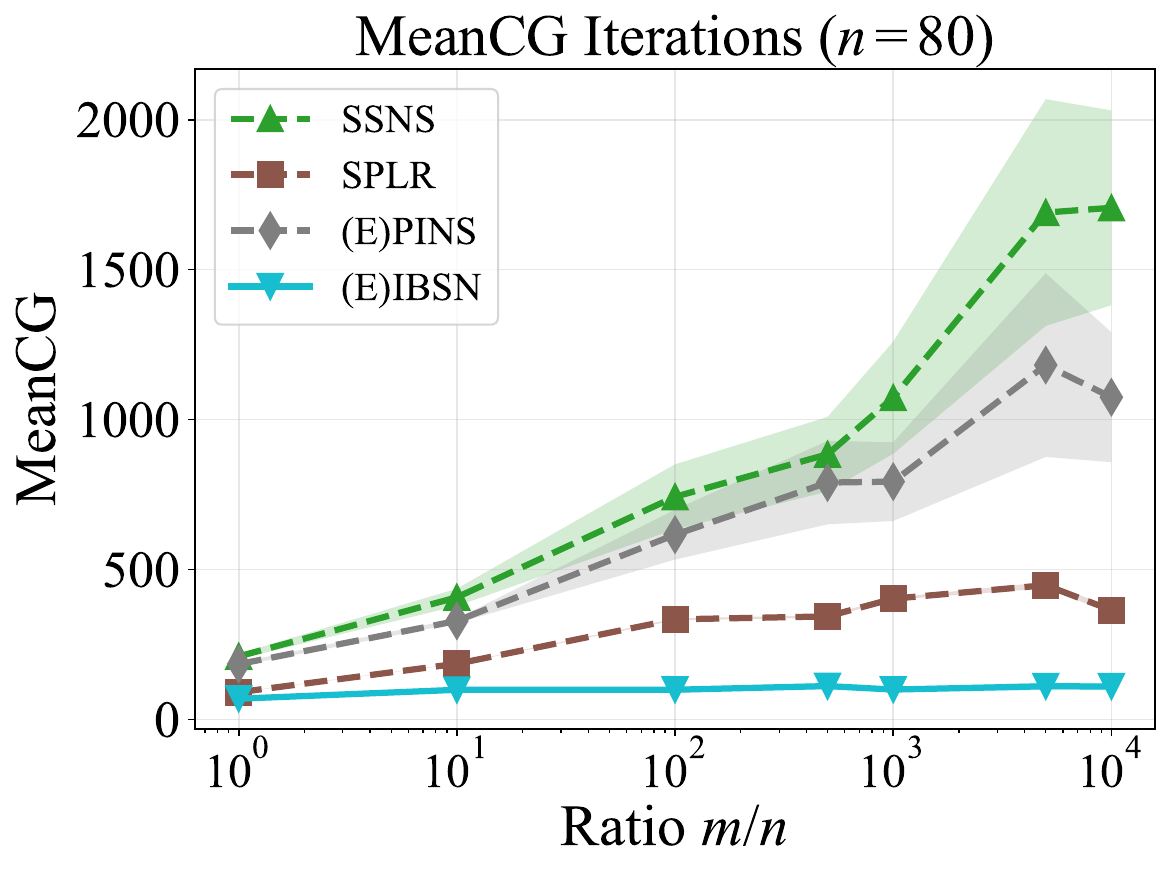}}}\hfill
{
\resizebox*{0.32 \textwidth}{!}{\includegraphics{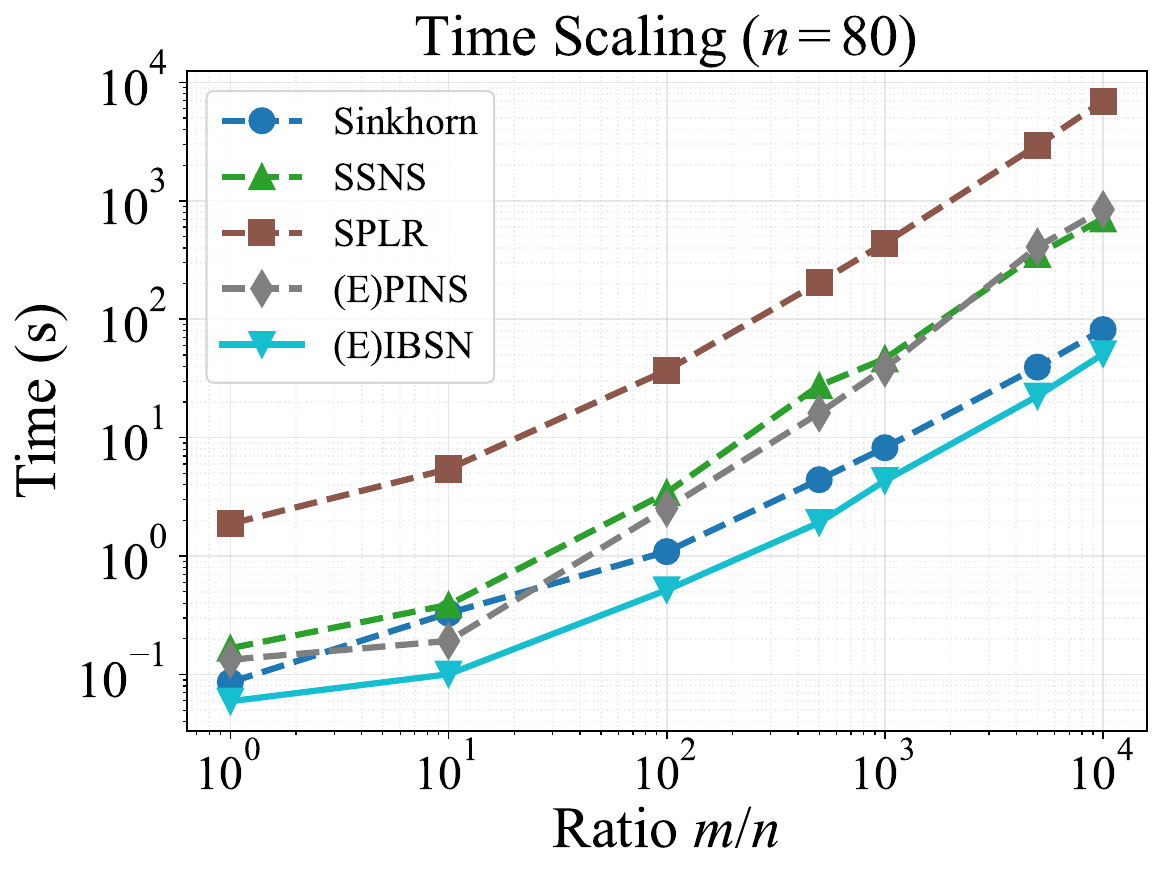}}}\hfill
{
\resizebox*{0.32 \textwidth}{!}{\includegraphics{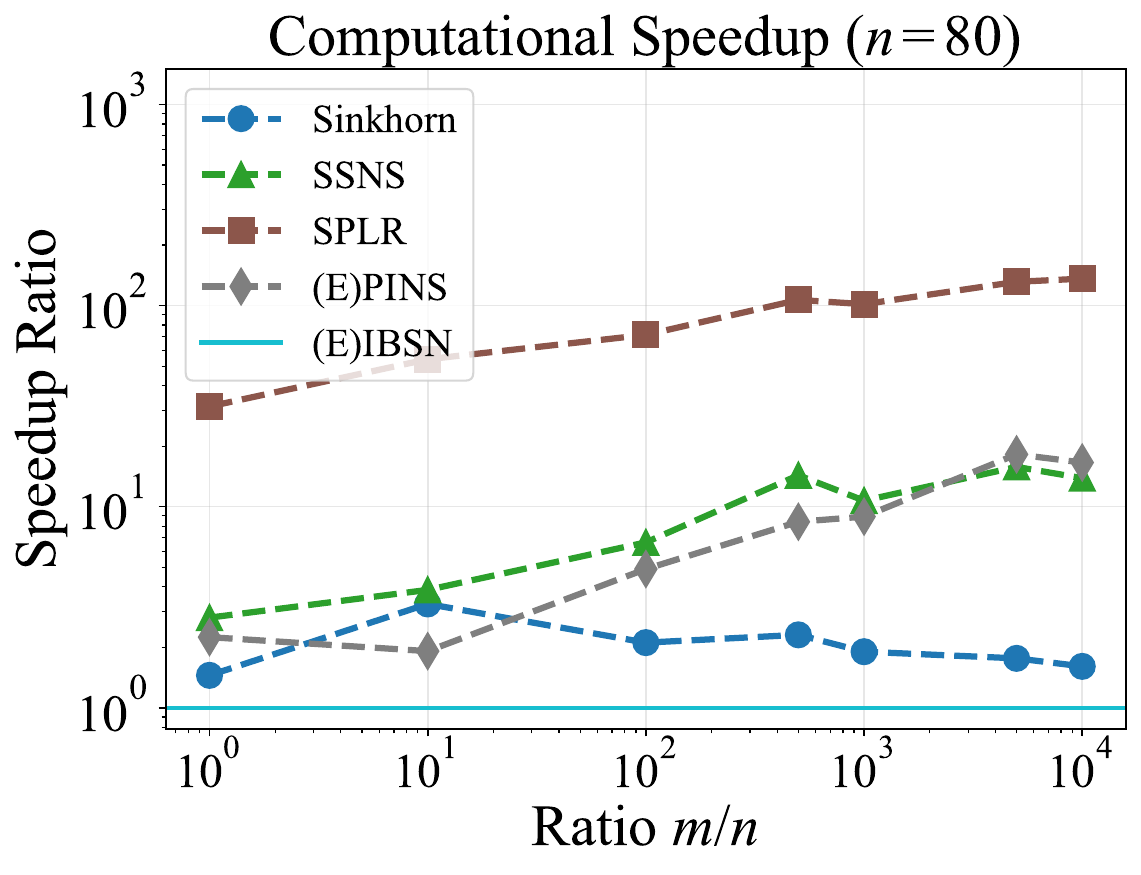}}}
\caption{Scalability analysis on highly unbalanced spherical OT problems. The target dimension is fixed at $n=50$ (top) and $n=80$ (bottom). Left: Mean number of CG iterations required to compute a Newton direction, with shaded regions representing the standard error. Middle: Total running time in log-log scale. Right: Speedup ratio of (E)IBSN relative to other algorithms. A value of $c > 1$ indicates that (E)IBSN is $c$ times faster than the competing algorithm.}
\label{fig:nonsquare_performance}
\end{figure*}

\paragraph{Performance on unbalanced $m > n$.}

Next, we investigate the performance of the algorithms in unbalanced settings where $m > n$. To do this, we conduct experiments on a spherical transport problem between supply and demand locations, inspired by~\cite{amos2023meta}. The cost matrix $C \in \R^{m \times n}$ is given by the spherical distance, i.e., $C_{ij} = \arccos(\langle \vect{x}_i, \vect{y}_j \rangle)$, where $\vect{x}_i, \vect{y}_j \in \mathbb{S} := \{\vect{z} \in \mathbb{R}^3 : \|\vect{z}\| = 1\}$. For the data generation, the demand locations $\{\vect{x}_i\}_{i=1}^m$ from Earth's population density data~\cite{doxsey2015taking}, while the supply locations $\{\vect{y}_j\}_{j=1}^n$ are sampled uniformly from Earth's landmass. The marginal weights $\vect{a}$ and $\vect{b}$ are sampled independently from a uniform distribution on $(0,1)$ and normalized such that $\1_m^\top \vect{a} = \1_n^{\top}\vect{b} = 1$. We set the regularization parameter to $\eta=10^{-3}$.

Figure~\ref{fig:nonsquare_performance} demonstrates that (E)IBSN requires the fewest mean number of CG iterations to find a Newton direction and consistently outperforms among all algorithms as the source dimension $m$ increases. This efficiency is attributed to our semi-dual formulation, which restricts the Hessian matrix to a low dimension regardless of the growth in the source dimension $m$. 

\subsection{Impact of the sparsification choices on EOT problem~\eqref{opt: EOT}}

In this section, we conduct an ablation study to assess the effect of different Hessian sparsification strategies for solving the EOT problem~\eqref{opt: EOT}. In particular, we compare the proposed threshold-based strategy with a simple magnitude-based sorting strategy. 

\begin{itemize}
    \item \textbf{Adaptive thresholding (ours):} We set the threshold as $\rho = \frac{\eta}{mn}\|\vect{g}\|$ and apply Algorithm~\ref{alg: Sparsifying the Hessian matrix} to control the sparsity level of $P_\rho$. The threshold is therefore adjusted dynamically according to the current optimality residual.
    \item \textbf{$q/n$-sort (baseline):} This strategy uses a prescribed hyperparameter $q$ to control the sparsity level of $P_{\rho}$. At each iteration, all entries of $P$ are sorted by magnitude, and only the largest $q$ entries are retained. The resulting matrix is then renormalized to obtain $P_{\rho}$, from which the sparse Hessian $H_{\rho}$ is constructed.
\end{itemize}

We evaluate both strategies on synthetic EOT instances with $m=n$ and fix the regularization parameter at $\eta = 10^{-4}$. For the sorting baseline, we consider the predefined capacities $q \in \{500,750,1000\}$.

As shown in Figure~\ref{fig: Performance of different sparsification strategies on EOT problems}, the proposed adaptive thresholding strategy maintains a stable proportion of nonzero entries across different problem dimensions. In contrast, the nonzero proportion produced by the sorting baseline decreases rapidly as $n$ increases. Although this aggressive sparsification can reduce the worst-case cost of the CG solver, it also leads to a less accurate Hessian approximation and may introduce numerical instability. As a result, the sorting-based strategy often fails to reach the target tolerance of $10^{-9}$ for larger problem sizes.

Another limitation of the sorting baseline is its dependence on the capacity parameter $q$. In practice, the best choice of $q$ varies with the problem size and the structure of the cost matrix, making the method sensitive to hyperparameter tuning. By contrast, the proposed adaptive thresholding rule adjusts automatically to the current residual, problem dimension, and regularization level, thereby providing a more robust and tuning-free sparsification mechanism.

\begin{figure*}[t]
{
\resizebox*{0.32 \textwidth}{!}{\includegraphics{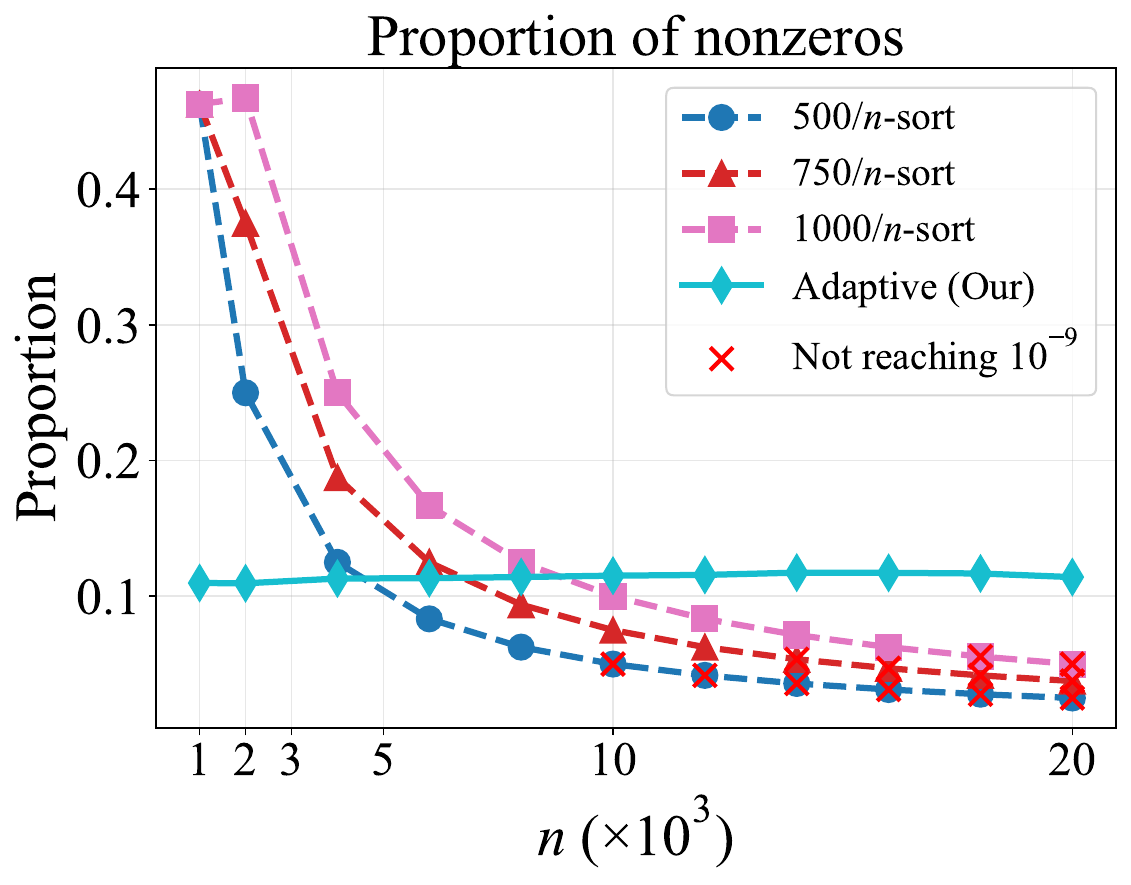}}}\hfill
{
\resizebox*{0.32 \textwidth}{!}{\includegraphics{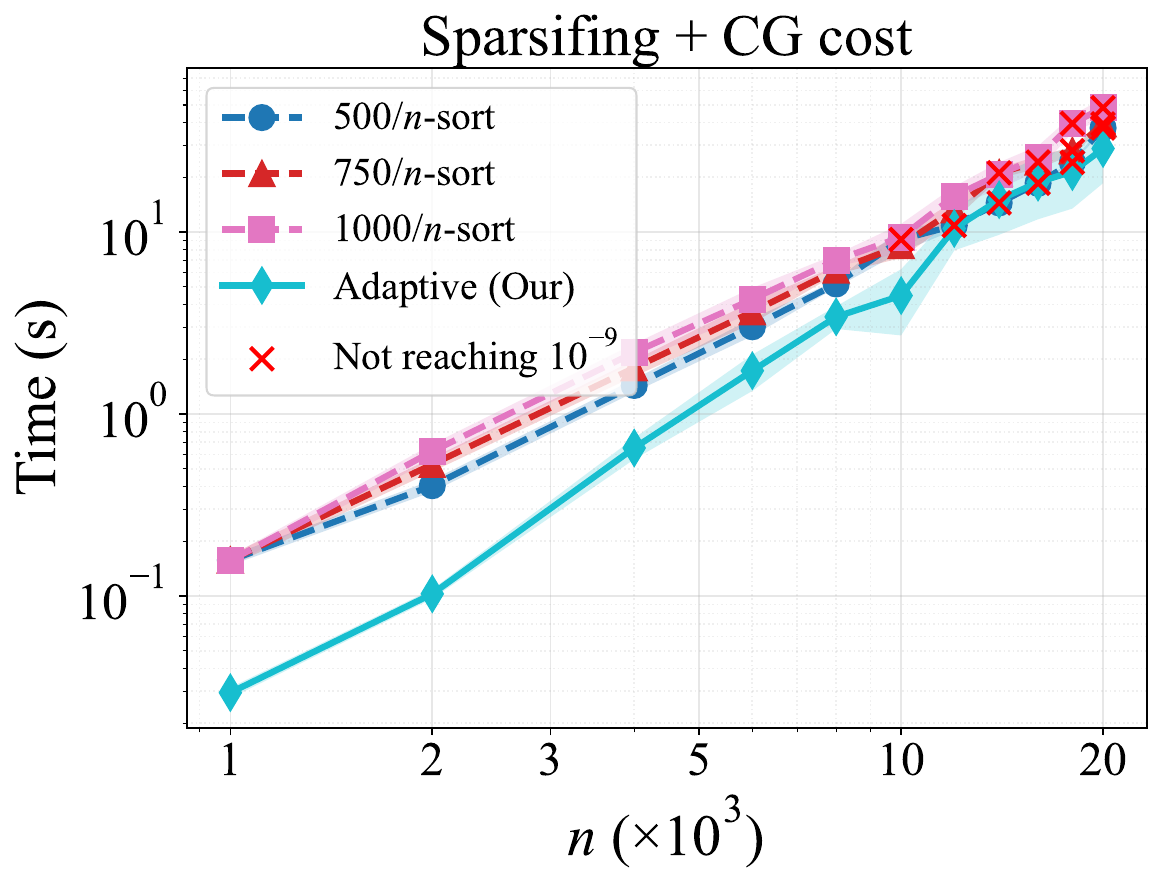}}}\hfill
{
\resizebox*{0.32 \textwidth}{!}{\includegraphics{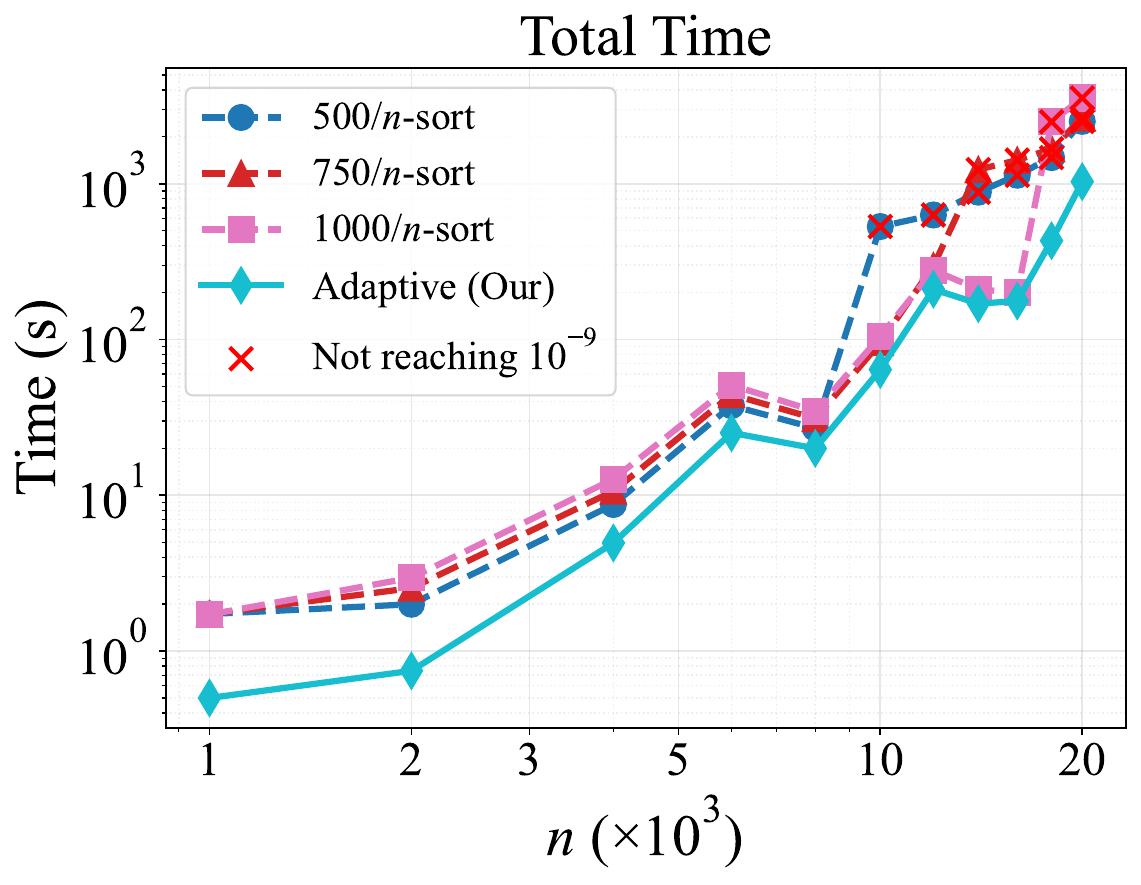}}}
{
\resizebox*{0.32 \textwidth}{!}{\includegraphics{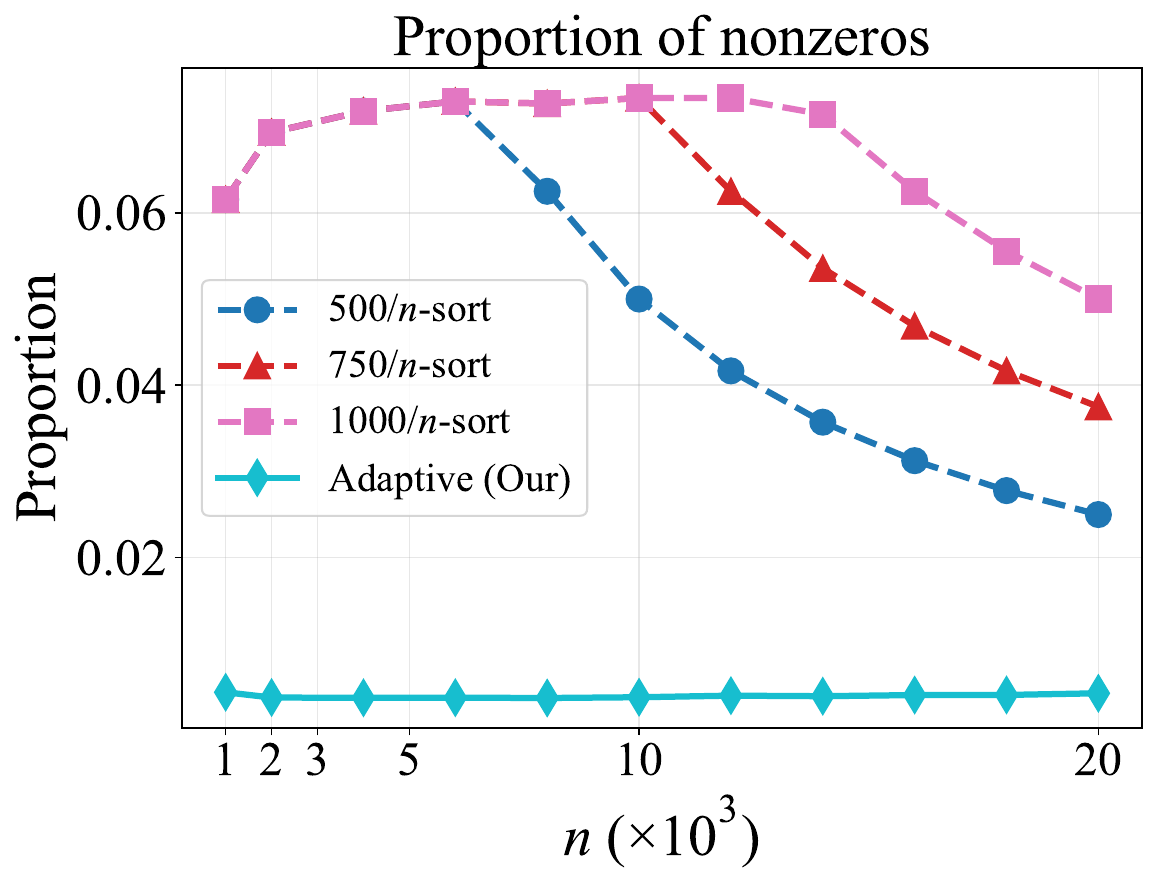}}}\hfill
{
\resizebox*{0.32 \textwidth}{!}{\includegraphics{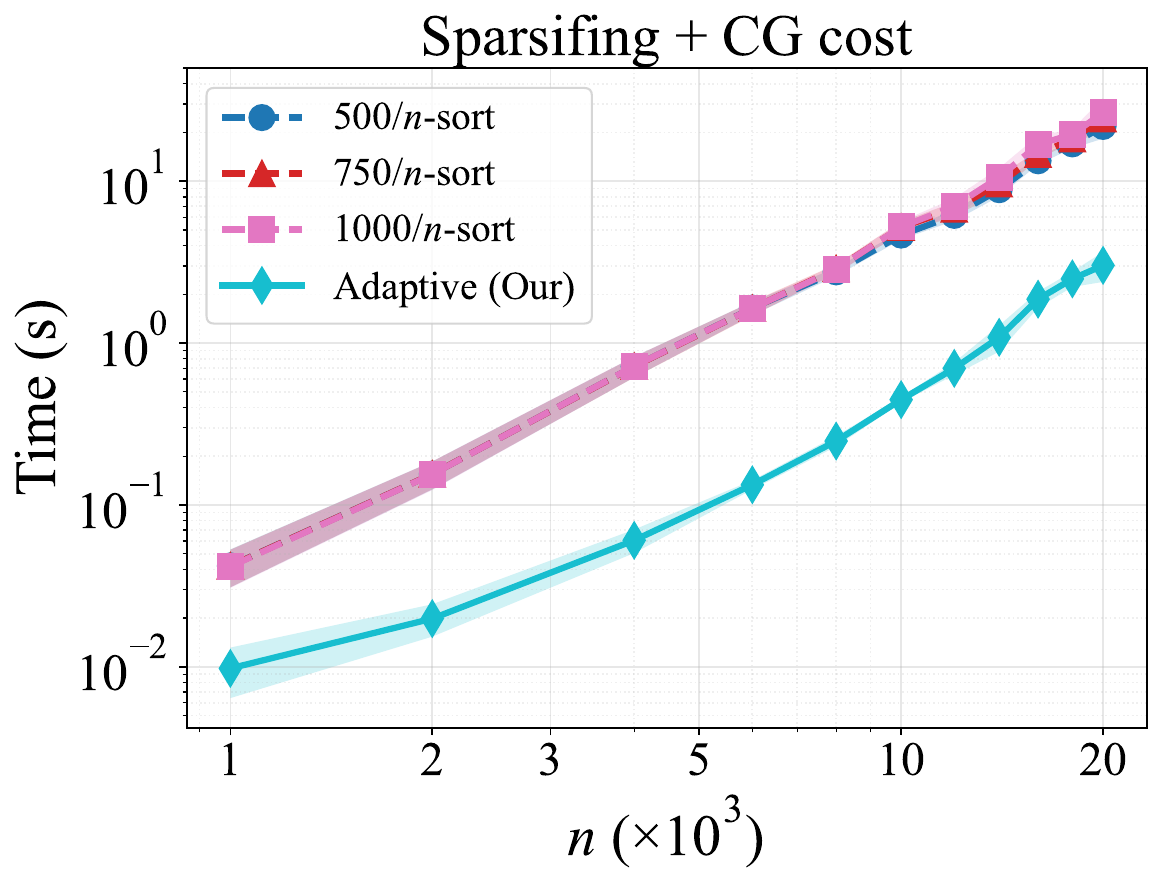}}}\hfill
{
\resizebox*{0.32 \textwidth}{!}{\includegraphics{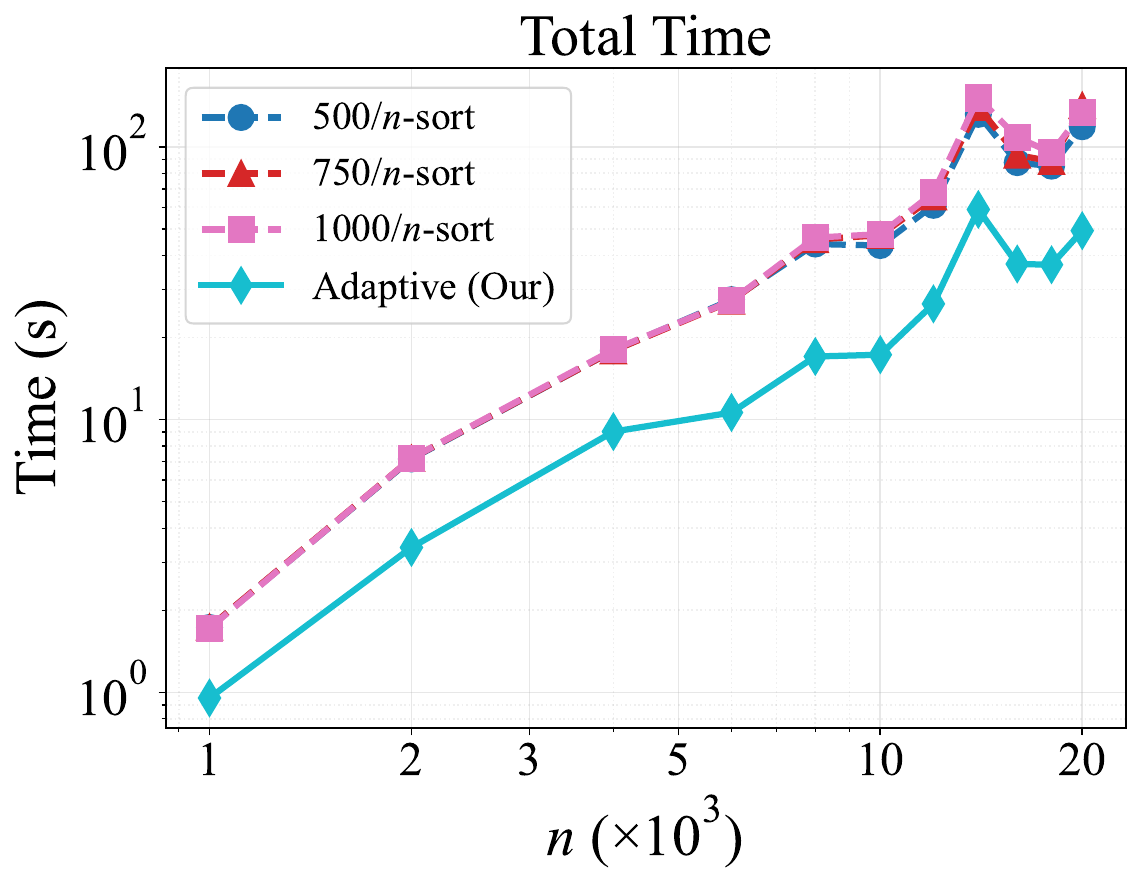}}}
\caption{Performance comparison of different sparsification strategies on EOT problems~\eqref{opt: EOT} under the Square setting (top) and Uniform setting (bottom). From left to right, the subplots evaluate three key metrics: the proportion of non-zero elements retained in $P_\rho$, the combined computational cost of sparsification and the CG solver to find a Newton direction, and the total time required to reach a convergence tolerance of $10^{-9}$. Instances where a strategy fails to converge are marked with a \textcolor{red}{$\times$}.}
\label{fig: Performance of different sparsification strategies on EOT problems}
\end{figure*}

\subsection{More comparisons on EOT problem~\eqref{opt: EOT}}
In this subsection, we extend our comparison to include the following methods, all implemented using the POT package~\cite{flamary2021pot, flamary2024pot}:

\begin{itemize}
    \item S2EOT: Solve the EOT problem using sparsity-constrained regularization~\cite{liu2022sparsity} (sparsity is chosen as 2). 
    \item LogSink: Solve the EOT problem using Sinkhorn-Knopp matrix scaling algorithm with the log stabilization~\cite{schmitzer2019stabilized}. 
    \item GCG: Solve the OT problem using $\ell_2$ and entropic regularization with the generalized conditional gradient algorithm~\cite{courty2016optimal,rakotomamonjy2015generalized}.
    \item L-BFGS-B~\cite{blondel2018smooth,cuturi2018semidual}: Solve the EOT problem (semi-dual formulation) using L-BFGS-B. 
\end{itemize}

All regularization parameters are set to be $10^{-4}$. Table~\ref{tab: Comparison of different algorithms in solving EOT problem} presents the results. Our proposed (E)IBSN algorithm consistently outperforms these baselines on both synthetic and real datasets, achieving significantly lower gradient norms in less computational time. These results demonstrate that (E)IBSN is a highly effective and robust solver for EOT problems~\eqref{opt: EOT}.

\begin{table}[t]
\centering
\caption{Comparison of different algorithms in the POT package when solving the EOT problem~\eqref{opt: EOT}.}
\label{tab: Comparison of different algorithms in solving EOT problem}
\resizebox{\textwidth}{!}{
\begin{threeparttable}
\begin{tabular}{ccc|ccc|ccc}
    \toprule[2pt]
    \multicolumn{9}{c}{Synthetic: Square} \\ 
    \cmidrule{1-9}
    \multicolumn{3}{c|}{size $m=n=1000$} & \multicolumn{3}{c|}{size $m=n=5000$} & \multicolumn{3}{c}{size $m=n=10000$}\\
    \cmidrule{1-9}
    Alg & GradNorm & Time (s) & Alg & GradNorm & Time (s) & Alg & GradNorm & Time (s) \\
    \cmidrule{1-9}
    S2EOT & 2.61e-02 & 5.65 & S2EOT & 9.57e-03 & 209.74 & S2EOT & 5.88e-03 & 833.64\\
    LogSink & 5.14e-02 & 25.49 & LogSink & 2.31e-02 & 774.38 & LogSink & 1.63e-02 & 2530.27\\
    GCG & 1.37e-04 & 2.33 & GCG & 1.43e-04 & 64.16 & GCG & 3.75e-05 & 217.92\\
    L-BFGS-B & 2.75e-04 & 0.93 & L-BFGS-B & 3.99e-05 & 39.79 & L-BFGS-B & 3.14e-05 & 100.43 \\
    (E)IBSN & \textbf{8.44e-12} & \textbf{0.55} & (E)IBSN & \textbf{1.08e-11} & \textbf{22.89} & (E)IBSN & \textbf{1.92e-11} & \textbf{90.91} \\
    \cmidrule{1-9}
    \multicolumn{9}{c}{Real: DOTmark (ID1 = 2, ID2 = 4) with size $m = n = 1024$} \\ 
    \cmidrule{1-9}
    \multicolumn{3}{c|}{Shapes32} & \multicolumn{3}{c|}{ClassicImages32} & \multicolumn{3}{c}{WhiteNoise32}\\
    \cmidrule{1-9}
    Alg & GradNorm & Time (s) & Alg & GradNorm & Time (s) & Alg & GradNorm & Time (s) \\
    \cmidrule{1-9}
    S2EOT & 8.02e-01 & 14.34 & S2EOT & 2.62e-01 & 10.75 & S2EOT & 7.74e-02 & 9.99 \\
    LogSink & 6.20e-02 & 28.92 & LogSink & 4.75e-02 & 27.41 & LogSink & 5.09e-02 & 27.04 \\
    GCG & 4.56e-06 & 2.15 & GCG & 2.43e-04 & 2.20 & GCG & 1.89e-04 & 2.17 \\
    L-BFGS-B & 4.46e-04 & 3.10 & L-BFGS-B & 2.34e-06 & 1.88 & L-BFGS-B & 2.65e-06 & 2.73 \\
    (E)IBSN & \textbf{2.25e-11} & \textbf{0.37} & (E)IBSN & \textbf{9.72e-11} & \textbf{1.72} & (E)IBSN & \textbf{4.30e-10} & \textbf{0.71} \\
    \cmidrule{1-9}
    \multicolumn{9}{c}{Real: DOTmark (ID1 = 2, ID2 = 4) with size $m=n=4096$} \\ 
    \cmidrule{1-9}
    \multicolumn{3}{c|}{Shapes64} & \multicolumn{3}{c|}{ClassicImages64} & \multicolumn{3}{c}{WhiteNoise64}\\
    \cmidrule{1-9}
    Alg & GradNorm & Time (s) & Alg & GradNorm & Time (s) & Alg & GradNorm & Time (s) \\
    \cmidrule{1-9}
    S2EOT & 3.19e-01 & 276.25 & S2EOT & 8.48e-02 & 219.66 & S2EOT & 2.68e-02 & 199.39 \\ 
    LogSink & 3.12e-02 & 439.69 & LogSink & 2.40e-02 & 440.81 & LogSink & 2.55e-02 & 435.39 \\ 
    GCG & 4.03e-06 & 34.64 & GCG & 1.22e-04 & 41.49 & GCG & 4.09e-05 & 40.48 \\
    L-BFGS-B & 3.75e-07 & 272.54 & L-BFGS-B & 3.28e-07 & 28.75 & L-BFGS-B & 5.69e-07 & 49.35 \\
    (E)IBSN & \textbf{1.39e-11} & \textbf{6.26} & (E)IBSN & \textbf{9.19e-12} & \textbf{4.58} & (E)IBSN & \textbf{1.08e-10} & \textbf{2.14} \\
    \bottomrule[2pt]
\end{tabular}
\begin{tablenotes}
    \footnotesize               
    \item Bold values indicate the most favorable results. The algorithm terminates when the gradient norm (GradNorm) is less than $10^{-9}$ or when the maximum number of iterations is reached.
\end{tablenotes}
\end{threeparttable}
}
\end{table}



\section{Color transfer}
\label{app: Color Transfer}

Color transfer aims to modify the color distribution of a source image to match the color palette of a target image while preserving the source image's structural content. Specifically, by interpreting image colors as probability distributions in a suitable color space, the optimal transport framework provides a method for transferring color characteristics between images. 

Following the experimental protocol in~\cite{zhang2025hot}, we conduct the color transfer experiments in the CIE-Lab color space. We apply the K-means clustering algorithm with $K = 512$ to quantize each image's color space and use the resulting cluster centers and their normalized frequencies to construct discrete color distributions. The proposed IBSN algorithm is then employed to compute the optimal transport plan between these two distributions, and the resulting mapping is used to transfer the target’s color style onto the source image. The results of color transfer for selected image pairs can be found in Figure~\ref{fig: color transfer}.

\begin{figure*}[t]
\centering
{
\resizebox*{0.8 \textwidth}{!}{\includegraphics{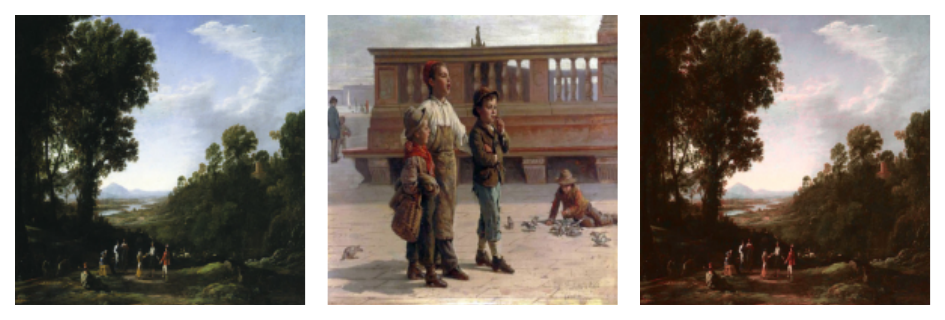}}}
{
\resizebox*{0.8 \textwidth}{!}{\includegraphics{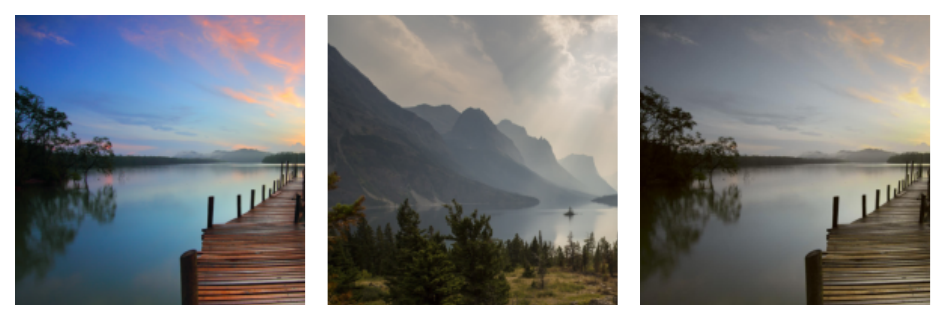}}}
{
\resizebox*{0.8 \textwidth}{!}{\includegraphics{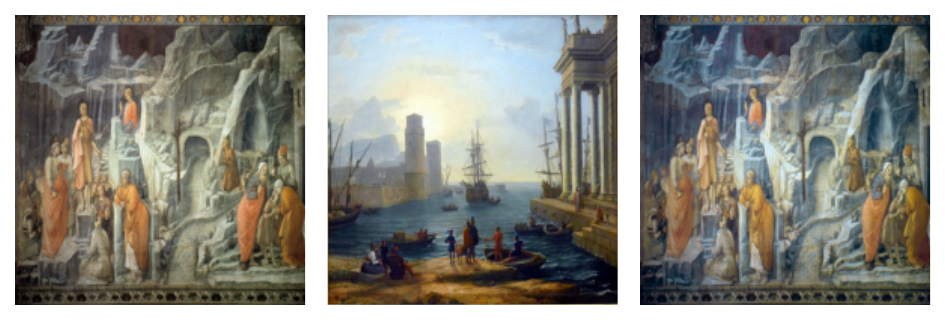}}}
\caption{Color transfer results obtained using the IBSN algorithm. Each triplet shows the source image (left), the target image (middle), and the resulting image after applying color transfer (right).}
\label{fig: color transfer}
\end{figure*}

\clearpage
\section*{NeurIPS Paper Checklist}

\begin{enumerate}

\item {\bf Claims}
    \item[] Question: Do the main claims made in the abstract and introduction accurately reflect the paper's contributions and scope?
    \item[] Answer: \answerYes{} 
    \item[] Justification: We have provided detailed derivations and numerical experiments to support the claims made in the abstract.

\item {\bf Limitations}
    \item[] Question: Does the paper discuss the limitations of the work performed by the authors?
    \item[] Answer: \answerYes{} 
    \item[] Justification: We have mentioned in the Appendix. 

\item {\bf Theory assumptions and proofs}
    \item[] Question: For each theoretical result, does the paper provide the full set of assumptions and a complete (and correct) proof?
    \item[] Answer: \answerYes{} 
    \item[] Justification: We listed the assumptions of the theoretical results and provided the proof in the Appendix.

    \item {\bf Experimental result reproducibility}
    \item[] Question: Does the paper fully disclose all the information needed to reproduce the main experimental results of the paper to the extent that it affects the main claims and/or conclusions of the paper (regardless of whether the code and data are provided or not)?
    \item[] Answer: \answerYes{} 
    \item[] Justification: We have provided the code to reproduce the experiments in the article.

\item {\bf Open access to data and code}
    \item[] Question: Does the paper provide open access to the data and code, with sufficient instructions to faithfully reproduce the main experimental results, as described in supplemental material?
    \item[] Answer: \answerYes{} 
    \item[] Justification: All datasets used are publicly available. 

\item {\bf Experimental setting/details}
    \item[] Question: Does the paper specify all the training and test details (e.g., data splits, hyperparameters, how they were chosen, type of optimizer) necessary to understand the results?
    \item[] Answer: \answerYes{} 
    \item[] Justification: The detailed parameter settings are reflected in the article, Appendix, as well as the programming code.

\item {\bf Experiment statistical significance}
    \item[] Question: Does the paper report error bars suitably and correctly defined or other appropriate information about the statistical significance of the experiments?
    \item[] Answer: \answerNA{} 
    \item[] Justification: The experiments focus on deterministic optimization convergence and runtime comparisons under fixed problem instances, rather than statistical estimation or predictive performance across random trials.

\item {\bf Experiments compute resources}
    \item[] Question: For each experiment, does the paper provide sufficient information on the computer resources (type of compute workers, memory, time of execution) needed to reproduce the experiments?
    \item[] Answer: \answerYes{} 
    \item[] Justification: We have provided the details in the article. 
    
\item {\bf Code of ethics}
    \item[] Question: Does the research conducted in the paper conform, in every respect, with the NeurIPS Code of Ethics \url{https://neurips.cc/public/EthicsGuidelines}?
    \item[] Answer: \answerYes{} 
    \item[] Justification: The work develops numerical optimization algorithms for OT and EOT computation and uses standard synthetic and publicly available benchmark datasets. It does not involve human subjects, sensitive personal data, or high-risk deployment settings.

\item {\bf Broader impacts}
    \item[] Question: Does the paper discuss both potential positive societal impacts and negative societal impacts of the work performed?
    \item[] Answer: \answerNA{} 
    \item[] Justification: This article is mostly a theoretical and computational work.
    
\item {\bf Safeguards}
    \item[] Question: Does the paper describe safeguards that have been put in place for responsible release of data or models that have a high risk for misuse (e.g., pre-trained language models, image generators, or scraped datasets)?
    \item[] Answer: \answerNA{} 
    \item[] Justification: The article poses no such risks. 

\item {\bf Licenses for existing assets}
    \item[] Question: Are the creators or original owners of assets (e.g., code, data, models), used in the paper, properly credited and are the license and terms of use explicitly mentioned and properly respected?
    \item[] Answer: \answerYes{} 
    \item[] Justification: We have listed the original papers that provide the datasets. 

\item {\bf New assets}
    \item[] Question: Are new assets introduced in the paper well documented and is the documentation provided alongside the assets?
    \item[] Answer: \answerNA{} 
    \item[] Justification: The article does not release new assets.

\item {\bf Crowdsourcing and research with human subjects}
    \item[] Question: For crowdsourcing experiments and research with human subjects, does the paper include the full text of instructions given to participants and screenshots, if applicable, as well as details about compensation (if any)? 
    \item[] Answer: \answerNA{} 
    \item[] Justification: The paper does not involve crowdsourcing nor research with human subjects.

\item {\bf Institutional review board (IRB) approvals or equivalent for research with human subjects}
    \item[] Question: Does the paper describe potential risks incurred by study participants, whether such risks were disclosed to the subjects, and whether Institutional Review Board (IRB) approvals (or an equivalent approval/review based on the requirements of your country or institution) were obtained?
    \item[] Answer: \answerNA{} 
    \item[] Justification: The article does not involve crowdsourcing nor research with human subjects.

\item {\bf Declaration of LLM usage}
    \item[] Question: Does the paper describe the usage of LLMs if it is an important, original, or non-standard component of the core methods in this research? Note that if the LLM is used only for writing, editing, or formatting purposes and does \emph{not} impact the core methodology, scientific rigor, or originality of the research, declaration is not required.
    \item[] Answer: \answerNA{} 
    \item[] Justification: Our core methodology and method development do not involve LLMs. They were strictly used for writing and language polishing purposes only.

\end{enumerate}

\end{document}